\documentclass[english]{article}
\usepackage[T1]{fontenc}
\usepackage[latin9]{inputenc}
\usepackage{geometry}
\usepackage{babel}
\usepackage{array}
\usepackage{verbatim}

\usepackage{enumitem}

\usepackage{amsmath, amsfonts, amsthm, bm,  amssymb, longtable, multirow, mathrsfs}
\usepackage{epstopdf, graphicx, subcaption,dsfont}
\usepackage{algorithm}  
\usepackage{algorithmicx}  
\usepackage{algpseudocode}
\usepackage{booktabs}
\usepackage{array}
\usepackage{booktabs}

\RequirePackage{natbib}
\RequirePackage[colorlinks,citecolor=blue,urlcolor=blue]{hyperref}

\DeclareMathOperator{\ind}{\mathds{1}}  

\allowdisplaybreaks

\theoremstyle{plain}
\newtheorem{thm}{Theorem}
\newtheorem{lem}{Lemma}

\newtheorem{cor}{Corollary}
\newtheorem{remark}{\textbf{Remark}}

\theoremstyle{remark}

\theoremstyle{definition}
\newtheorem{defn}{Definition}

\newtheorem{assumption}{\textbf{Assumption}}

\theoremstyle{remark}

\usepackage{listings}
\usepackage{xcolor}
\definecolor{mygreen}{rgb}{0,0.6,0}
\definecolor{mygray}{rgb}{0.5,0.5,0.5}
\definecolor{mymauve}{rgb}{0.58,0,0.82}
\definecolor{yxc}{RGB}{255,0,0}
\definecolor{cc}{RGB}{125,0,0}

\begin{document}
	\title{Asymmetry Helps: Eigenvalue and Eigenvector Analyses of \\ Asymmetrically Perturbed Low-Rank Matrices\footnotetext{Author names are sorted alphabetically. Corresponding
author: Yuxin Chen.}}
	\author{Yuxin Chen\thanks{Department of Electrical Engineering, Princeton University, Princeton, NJ 08544, USA; Email:
		\texttt{yuxin.chen@princeton.edu}.} \\
		\and Chen Cheng\thanks{Department of Statistics, Stanford University, Stanford, CA 94305, USA; Email: \texttt{chencheng@stanford.edu}.} \\
		\and Jianqing Fan\thanks{Department of Operations Research and Financial Engineering, Princeton University, Princeton, NJ 08544, USA; Email: \texttt{jqfan@princeton.edu}.}
		}
	\date{November 2018; ~Revised February 2020}
	\maketitle
	
\begin{abstract}
	This paper is concerned with the interplay between statistical asymmetry and  spectral methods.	
	Suppose we are interested in estimating a rank-1 and symmetric matrix $\bm{M}^{\star}\in \mathbb{R}^{n\times n}$, yet only a randomly perturbed version $\bm{M}$ is observed.   The noise matrix $\bm{M}-\bm{M}^{\star}$ is composed of independent (but not necessarily homoscedastic)  entries and is, therefore, not symmetric in general. This might arise if, for example, we have two independent samples for each entry of $\bm{M}^{\star}$ and arrange them in an {\em asymmetric} fashion. 
 The aim is to estimate the leading eigenvalue and the leading eigenvector of $\bm{M}^{\star}$. 
	
	We demonstrate that the leading eigenvalue of the data matrix $\bm{M}$ can be $O(\sqrt{n})$ times more accurate (up to some log factor) than its (unadjusted) leading singular value of $\bm{M}$ in eigenvalue estimation. Moreover, the eigen-decomposition approach is fully adaptive to heteroscedasticity of noise,  without the need of any prior knowledge about the noise distributions. In a nutshell, this curious phenomenon arises since the statistical asymmetry automatically mitigates the bias of the eigenvalue approach, thus eliminating the need of careful bias correction. 	
	Additionally, we develop appealing non-asymptotic eigenvector perturbation bounds; in particular, we are able to bound the perterbation of any  linear function of the leading eigenvector of $\bm{M}$ (e.g.~entrywise eigenvector perturbation). 
	We also provide partial theory for the more general rank-$r$ case. 
	The takeaway message is this:  arranging the data samples in an asymmetric manner and performing eigen-decomposition could sometimes be quite beneficial.

\end{abstract}


	\tableofcontents

		\section{Introduction}

	Consider an unknown symmetric and low-rank matrix $\bm{M}^{\star} \in \mathbb{R}^{n \times n}$. What we have observed is a corrupted version
	\begin{align}
		\bm{M}=\bm{M}^{\star}+\bm{H},
	\end{align}
	with $\bm{H}$ denoting a noise matrix.  A classical  problem is concerned with estimating the leading eigenvalues and eigenspace of $\bm{M}^{\star}$ given observation $\bm{M}$.

	The current paper concentrates on a scenario where  the noise matrix $\bm{H}$ (and hence $\bm{M}$) consists of independently generated random entries and is hence {\em asymmetric} in general. This might arise, for example, when we have available multiple (e.g.~two)  samples for each entry of $\bm{M}^{\star}$ and  arrange the samples in an asymmetric fashion.         A natural approach that immediately comes to mind is based on singular value decomposition (SVD), which employs the leading singular values (resp.~subspace) of $\bm{M}$ to approximately estimate the eigenvalues (resp.~eigenspace) of $\bm{M}^{\star}$. By contrast, a much less popular alternative is based on eigen-decomposition of the asymmetric data matrix $\bm{M}$, which attempts approximation using the leading eigenvalues and eigenspace of $\bm{M}$. Given that  eigen-decomposition of an asymmetric matrix is in general not as  numerically stable as SVD, conventional wisdom often favors the SVD-based approach, unless certain symmetrization step is implemented prior to eigen-decomposition.

	When comparing these two approaches numerically, however,
a curious phenomenon arises, which largely motivates the research in this paper. Let us generate  $\bm{M}^{\star} $ as a random rank-1 matrix with leading eigenvalue $\lambda^{\star}=1$, and let $\bm{H}$ be a Gaussian random matrix whose entries are i.i.d.~$\mathcal{N}(0,\sigma^2)$ with $\sigma = 1 / \sqrt{n\log n}$.  Fig.~\ref{fig:eigenvalue-perturb}(a) compares the empirical accuracy of estimating the 1st eigenvalue of $\bm{M}^{\star}$ via the leading eigenvalue (the blue line) and via the leading singular value of $\bm{M}$ (the red line).  As it turns out,  eigen-decomposition significantly outperforms vanilla SVD in  estimating $\lambda^{\star}$, and the advantage seems increasingly more remarkable as the dimensionality $n$ grows. To facilitate comparison,  we include an additional green line in Fig.~\ref{fig:eigenvalue-perturb}(a), obtained by rescaling the red line by $2.5/\sqrt{n}$. Interestingly, this green line coincides almost perfectly with the blue line, thus suggesting orderwise gain of eigen-decomposition compared to SVD.
What is more, this phenomenon does not merely happen under i.i.d.~noise.  Similar numerical behaviors are observed in the problem of matrix completion --- as displayed in Fig.~\ref{fig:eigenvalue-perturb}(b) --- even though the components of the equivalent perturbation matrix  are apparently far from identically distributed or homoscedastic.

	\begin{figure}[htbp!]
		\centering
		\begin{tabular}{cc}
			\includegraphics[width=0.4\linewidth]{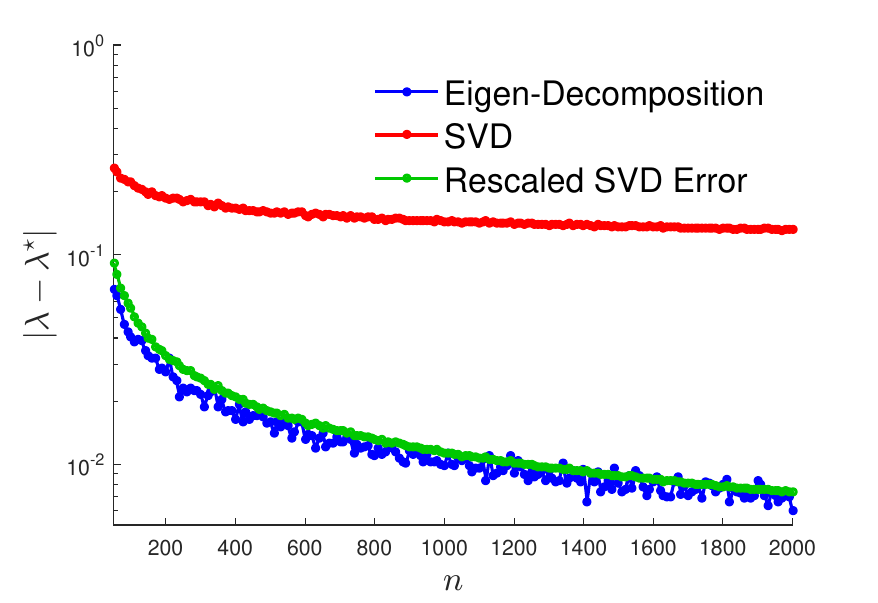} & \includegraphics[width=0.4\linewidth]{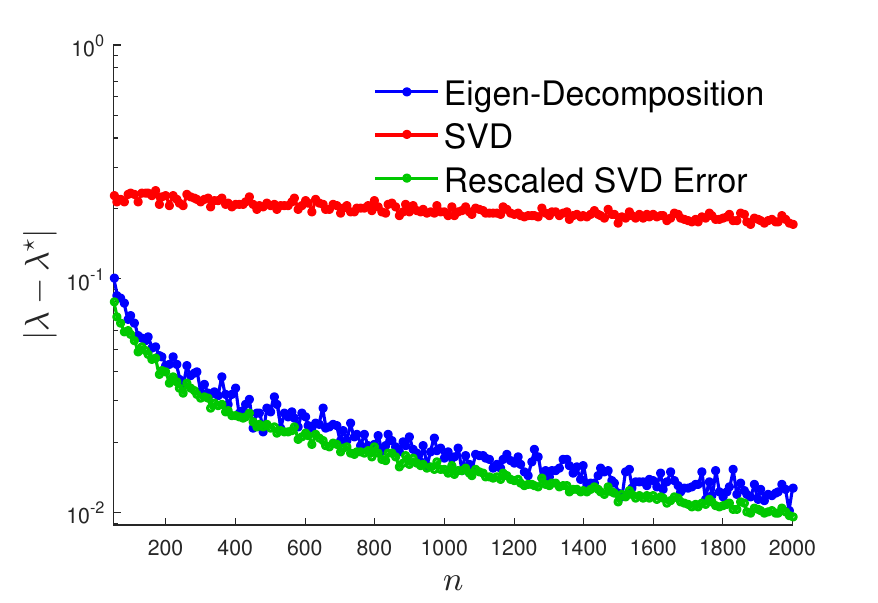}
 \tabularnewline
			(a) eigenvalue perturbation for matrix estimation & (b) eigenvalue perturbation for matrix completion \tabularnewline	
		\end{tabular}
		\caption{Numerical error $|\lambda- \lambda^{\star}|$ vs.~the matrix dimension $n$, where $\lambda$ is either the leading eigenvalue (the blue line) or the leading singular value (the red line) of $\bm{M}$. Here, (a) is the case when $\{H_{ij}\}$ are i.i.d.~$\mathcal{N}(0,\sigma^2)$ with $\sigma = 1/\sqrt{n\log n}$, and (b) is the matrix completion case with sampling rate $p =3\log n/n$, where $M_{i,j}= \frac{1}{p}M_{i,j}^{\star}$ independently with probability $p$ and 0 otherwise.   The results are averaged over 100 independent trials. The green lines are obtained by rescaling the corresponding red lines by $2.5/\sqrt{n}$.
		} \label{fig:eigenvalue-perturb}
	\end{figure}

	The goal of the current paper is thus to develop a systematic understanding of this phenomenon, that is, why  statistical asymmetry empowers eigen-decomposition and how to exploit this feature in statistical estimation. Informally, our findings suggest that: when $\bm{M}^{\star}$ is rank-1 and $\bm{H}$ is composed of zero-mean and independent (but not necessarily identically distributed or homoscedastic) entries,
	{\setlist{rightmargin=\leftmargin}
	\begin{itemize}
		\item[1.] the leading eigenvalue of  $\bm{M}$ could be $O(\sqrt{n})$ times (up to some logarithmic factor) more accurate than  the (unadjusted) leading singular value of $\bm{M}$ when estimating the 1st eigenvalue of $\bm{M}^{\star}$;\footnote{More precisely, this gain is possible when $\|\bm{H}\|$ is nearly as large as $\|\bm{M}^{\star}\|$ (up to some logarithmic factor). }
		\item[2.] the perturbation of the leading eigenvector is well-controlled along an {\em arbitrary} deterministic direction; for example, the eigenvector perturbation is well-controlled in any coordinate, indicating that
the eigenvector estimation error is  spread out across all coordinates.
	\end{itemize}
	}
	\noindent   We will further provide partial theory to accommodate the rank-$r$ case. As an important application,  such a theory allows us to estimate the leading singular value and singular vectors of an asymmetric rank-1 matrix via eigen-decomposition of a certain dilation matrix,  which also outperforms the vanilla SVD approach.
	
	We would  like to immediately remark that: for some scenarios (e.g.~the case with i.i.d.~Gaussian noise), it is possible to  adjust the leading singular value of $\bm{M}$ to obtain the same accuracy as the leading eigenvalue of $\bm{M}$. As it turns out, the advantages of the eigen-decomposition approach may become more evident in the presence of heteroscedasticity --- the case where the noise has location-varying and unknown variance. We shall elaborate on this  point in Section \ref{sec:eigenvalue-perturb-rank1}.

All in all,   when it comes to low-rank matrix estimation, arranging the observed matrix samples in an asymmetric manner and invoking eigen-decomposition properly could sometimes be statistically beneficial.

	\section{Problem formulation}

	\subsection{Models and assumptions}

	In this section, we formally introduce our models and assumptions.  Consider a symmetric and low-rank matrix $\bm{M}^{\star}  =[M^{\star}_{ij}]_{1\leq i,j\leq  n} \in \mathbb{R}^{n \times n}$.  		Suppose we are given a random copy of $\bm{M}^{\star} $ as follows
	\begin{equation}
		\label{eq_def_M}
		\bm{M} = \bm{M}^{\star} + \bm{H},
	\end{equation}
	where $\bm{H} = [H_{ij}]_{1\leq i,j\leq  n}$ is a random noise matrix.

	The current paper concentrates on  independent --- but not necessarily identically distributed or homoscedastic ---  noise. Specifically, we impose the following assumptions on $\bm{H}$ throughout this paper.
	%
	%
	%
	\begin{assumption}
	\label{assumption:H}
	\begin{enumerate}
		\item[]
		\item[1.] \textbf{(Independent entries)} The entries $\{H_{ij}\}_{1\leq i,j\leq n}$ are independently generated;
		\item[2.] \textbf{(Zero mean)} $\mathbb{E}[H_{ij}] = 0$ for all $1 \leq i, j \leq n$;
		\item[3.] \textbf{(Variance)} $\mathsf{Var}(H_{ij}) = \mathbb{E}\left[H_{ij}^2\right] \leq \sigma_n^2$ for all $1 \leq i, j \leq n$;
		\item[4.] \textbf{(Magnitude)}  Each $H_{ij}$ ($1 \leq i, j \leq n$) satisfies either of the following conditions:
		\begin{enumerate}
			\item[(a)] $|H_{ij}| \leq B_n$;
			\item[(b)] $H_{ij}$ has a symmetric distribution obeying $\mathbb{P}\{ |H_{ij}| > B_n \} \leq c_{\mathrm{b}} n^{-12} $ for some universal constant $c_{\mathrm{b}} > 0$.
		\end{enumerate}
	\end{enumerate}
	\end{assumption}
	\begin{remark}[Notational convention]
		In what follows, the dependency of $\sigma_n$ and $B_n$ on $n$ shall often be suppressed whenever it is clear from the context, so as to simplify notation.
	\end{remark}
	\noindent Note that we do not enforce the constraint $H_{ij}=H_{ji}$, and hence $\bm{H}$ and $\bm{M}$ are in general  asymmetric matrices. Also, Condition 3 does not require the $H_{ij}$'s to have equal variance across different locations; in fact, they can be heteroscedastic.
	In addition, while Condition 4(a) covers the class of bounded random variables,  Condition 4(b) allows us to accommodate a large family of heavy-tailed distributions (e.g.~sub-exponential distributions).
	An immediate consequence of Assumption \ref{assumption:H} is the following bound on the spectral norm $\|\bm{H}\|$ of $\bm{H}$.
	\begin{lem}
		\label{lem:H-norm}
		Under Assumption \ref{assumption:H}, there exist some universal constants $c_0, C_0>0$ such that with probability exceeding $1- C_0 n^{-10}$,
		\begin{equation}
			\| \bm{H} \| \leq c_0  \sigma \sqrt{ n \log n } + c_0 B \log n.
		\end{equation}
	\end{lem}
	\begin{proof}
		This is a standard {\em non-asymptotic} result that follows immediately from the matrix Bernstein inequality \cite{Tropp:2015:IMC:2802188.2802189} and the union bound (for Assumption 4(b)). We omit the details for conciseness.
	\end{proof}

	\subsection{Our goal}

	The aim is to develop {\em non-asymptotic} eigenvalue and eigenvector perturbation bounds under this family of random and asymmetric noise matrices. Our theoretical development is divided into two parts. Below we introduce our goal as well as some notation used throughout.

	\paragraph{\bf Rank-1 symmetric case.} For  the rank-1 case, we assume  the eigen-decomposition of $\bm{M}^{\star}$ to be
	\begin{align}
		\label{eq:M-star-eigen-rank-1-a}
		\bm{M}^\star & = \lambda^\star \bm{u}^\star  \bm{u}^{\star\top}
	\end{align}
	with $\lambda^{\star}$ and $\bm{u}^{\star}$ being its leading eigenvalue and eigenvector, respectively.
	We also denote by  $\lambda$ and $\bm{u}$ the leading eigenvalue and eigenvector of $\bm{M}$, respectively.
			The following quantities are the focal points of this paper (see Section \ref{sec:rank-1}):
	\begin{itemize}
		\item[1.] {\em Eigenvalue perturbation:} $|\lambda-\lambda^{\star}|$;
		\item[2.] {\em Perturbation of linear forms of eigenvectors:} $\min \{ | \bm{a}^{\top} ( \bm{u} - \bm{u}^{\star})|, | \bm{a}^{\top} ( \bm{u} + \bm{u}^{\star})| \}$ for any fixed unit vector $\bm{a}\in \mathbb{R}^n$;
		\item[3.] {\em Entrywise eigenvector perturbation:} $\min \{ \|\bm{u} - \bm{u}^{\star} \|_\infty, \|\bm{u} + \bm{u}^{\star} \|_\infty \}$.
	\end{itemize}

	\paragraph{\bf Rank-$r$ symmetric case.} For the general rank-$r$ case, we let the eigen-decomposition of $\bm{M}^{\star}$ be
	\begin{align}
		\label{eq:M-star-eigen}
		\bm{M}^\star & = \bm{U}^\star \bm{\Sigma}^\star \bm{U}^{\star\top},
	\end{align}
	where the columns of $\bm{U}^\star = [ \bm{u}_1^\star, \cdots,  \bm{u}_r^\star] \in \mathbb{R}^{n \times r}$ are the eigenvectors, and $\bm{\Sigma}^\star = \mathsf{diag}(\lambda_{1}^{\star}, \cdots, \lambda_{r}^{\star})
	\in \mathbb{R}^{r \times r}$ is a diagonal matrix with the eigenvalues arranged in descending order by their magnitude, i.e.~$|\lambda_1^\star| \geq \cdots \geq |\lambda_r^\star|$. We let $\lambda_{\mathrm{max}}^\star = |\lambda_1^\star|$ and $ \lambda_{\mathrm{min}}^\star = |\lambda_r^{\star}|$.
In addition,
we let the top-$r$ eigenvalues (in magnitude) of $\bm{M}$ be $\lambda_1,\cdots, \lambda_r$ (obeying $ |\lambda_1| \geq \cdots \geq |\lambda_r| $) and their corresponding normalized eigenvectors be $\bm{u}_1, \cdots, \bm{u}_r$. We will present partial eigenvalue perturbation results  for this more general case, as detailed in Section \ref{sec:main-rank-r}.

As is well-known,  eigen-decomposition can be applied to estimate the singular values and singular vectors of an asymmetric matrix $\bm{M}^{\star}$ via the standard dilation trick \cite{Tropp:2015:IMC:2802188.2802189}.  As a consequence, our results are also applicable for singular value and singular vector estimation.  See Section \ref{sec:asymmetric-rank1} for details.

	\subsection{Incoherence conditions}

	Finally, we single out an incoherence parameter that plays an important role in our theory, which captures how well the energy of the eigenvectors is spread out across all entries.
	\begin{defn}[\textbf{Incoherence parameter}] \label{def_incoherence_parameter} The incoherence parameter of
		a rank-$r$ symmetric matrix $\bm{M}^\star$ with eigen-decomposition $\bm{M}^\star = \bm{U}^\star \bm{\Sigma}^\star \bm{U}^{\star \top}$ is defined to be the smallest quantity $\mu$ obeying
	\begin{equation}
		\left\|\bm{U}^\star\right\|_{\infty} \leq
		\sqrt{\frac{\mu }{n}} ,
	\end{equation}
		where $\left\|\cdot\right\|_{\infty}$ denotes the entrywise $\ell_{\infty}$ norm.
	\end{defn}
	\begin{remark}\label{remark:alternative-incoherence}
		An alternative definition of the incoherence parameter \citep{ExactMC09,KesMonSew2010,chi2018nonconvex,chen2019noisy} is  the smallest quantity $\mu_0$ satisfying $\left\|\bm{U}^\star\right\|_{2,\infty} \leq \sqrt{ \mu_0 r / n}$.  This is a weaker assumption than  Definition \ref{def_incoherence_parameter},  as it only requires the energy of $\bm{U}^{\star}$ to be spread out across all of its rows  rather than all of its entries.
	Note that these two incoherent parameters are consistent in the rank-$1$ case; in the rank-$r$ case one has $\mu_0 \leq \mu \leq \mu_0 r$.
		
	\end{remark}

	\subsection{Notation}

	The standard basis vectors in $\mathbb{R}^n$ are denoted by $\bm{e}_1,\cdots,\bm{e}_n$. For any vector $\bm{z}$, we let $\|\bm{z}\|_2$ and $\|\bm{z}\|_{\infty}$ denote the $\ell_2$ norm and the $\ell_{\infty}$ norm of $\bm{z}$, respectively.
	For any matrix $\bm{M}$, denote by $\|\bm{M}\|$, $\|\bm{M}\|_{\mathrm{F}}$ and $\|\bm{M}\|_{\infty}$ the spectral norm, the Frobenius norm and the entrywise $\ell_{\infty}$ norm (the largest magnitude of all entries) of $\bm{M}$, respectively.
	Let $[n]:=\{1,\cdots,n\}$.
	In addition, the  notation $f(n)=O\left(g(n)\right)$ or
$f(n)\lesssim g(n)$ means that there is a constant $c>0$ such
that $\left|f(n)\right|\leq c|g(n)|$, $f(n)\gtrsim g(n)$ means that  there is a constant $c>0$ such
that $|f(n)|\geq c\left|g(n)\right|$, and $f(n)\asymp g(n)$ means that there exist constants $c_{1},c_{2}>0$
such that $c_{1}|g(n)|\leq|f(n)|\leq c_{2}|g(n)|$.

		\section{Preliminaries}
	\label{sec:preliminary}

	Before continuing, we gather several preliminary facts that will be useful throughout. The readers familiar with matrix perturbation theory may proceed directly to the main theoretical development in Section \ref{sec:rank-1}.

	\subsection{Perturbation of eigenvalues of asymmetric matrices}

	We begin with a standard result concerning eigenvalue perturbation of a diagonalizable matrix \cite{bauer1960norms}. Note that the matrices under study might be asymmetric. 
	\begin{thm}[\textbf{Bauer-Fike Theorem}] \label{thm_BauerFike} 
	Consider a diagonalizable matrix $\bm{A} \in \mathbb{R}^{n \times n}$ with eigen-decomposition $\bm{A} = \bm{V} \bm{\Lambda} \bm{V}^{-1}$, where $\bm{V} \in \mathbb{C}^{n \times n}$ is a non-singular eigenvector matrix and $\bm{\Lambda}$ is diagonal.  Let $\tilde{\lambda}$ be an eigenvalue of $\bm{A} + \bm{H}$. Then there exists an eigenvalue $\lambda$ of $\bm{A}$ such that
	\begin{equation}
		|\lambda - \tilde{\lambda}| \leq \left\|\bm{V}\right\| \left\|\bm{V}^{-1}\right\| \left\| \bm{H}\right\|.
	\end{equation}
	In addition, if $\bm{A}$ is symmetric, then there exists an eigenvalue $\lambda$ of $\bm{A}$ such that
	\begin{equation}
		|\lambda - \tilde{\lambda}| \leq \left\| \bm{H}\right\|.
	\end{equation}
	\end{thm}

	%
	However, caution needs to be exercised as the Bauer-Fike Theorem  does not specify which eigenvalue of $\bm{A}$ is close to an eigenvalue of $\bm{A}+\bm{H}$. 
	Encouragingly, in the low-rank case of interest, the Bauer-Fike Theorem together with certain continuity of the spectrum allows one to localize the leading eigenvalues of the perturbed matrix.  
	%
	%
	\begin{lem}
		\label{lem:perturbation-BF-rank-r}
		Suppose $\bm{M}^{\star}$ is a rank-$r$ symmetric matrix whose top-$r$ eigenvalues obey $\big| \lambda_1^{\star} \big| \geq \cdots \geq \big|\lambda_r^{\star} \big| >0$. If $\|\bm{H}\| < \big|\lambda^{\star}_r\big|/2$, then the top-$r$ eigenvalues $\lambda_1,\cdots,\lambda_r$ of $\bm{M}=\bm{M}^{\star}+\bm{H}$, sorted by modulus, obey that: for any $1\leq l\leq r$, 
	\begin{equation}
		|\lambda_l - \lambda^{\star}_j|  \leq \left\|\bm{H}\right\| \quad \text{for some } 1\leq j\leq r. \label{eq:lambda-coarse-rank-r}
	\end{equation}
	In addition, if $r=1$, then both the leading eigenvalue and the leading eigenvector of $\bm{M}$ are  real-valued. 
	\end{lem}

	This result, which we establish in Appendix \ref{appendix:proof-lem:perturbation-BF},   parallels Weyl's inequality for symmetric matrices.  Note, however, that the above bound \eqref{eq:lambda-coarse-rank-r} might be quite loose for specific settings.  We will establish  much sharper perturbation bounds when $\bm{H}$ contains independent random entries (see, e.g.~Corollary \ref{cor:eigenvalue-pertubation}).

	\subsection{The Neumann trick and eigenvector perturbation}

	Next, we introduce  a classical result dubbed as the ``Neumann trick'' \cite{pmlr-v83-eldridge18a}. This theorem, which is derived based on the Neumann series for a matrix inverse, has been applied to analyze eigenvectors in various settings  \cite{erdHos2013spectral,jain2015fast,pmlr-v83-eldridge18a}. 

	\begin{thm}[\textbf{Neumann trick}] 
	\label{thm_Neumann}  
		Consider the  matrices $\bm{M}^{\star}$ and $\bm{M}$ (see~\eqref{eq:M-star-eigen} and \eqref{eq_def_M}). Suppose $\left\|\bm{H}\right\| < |\lambda_l|$ for some  $1\leq l \leq n$. Then
	\begin{equation}
		\bm{u}_l = \sum_{j=1}^r \frac{\lambda_j^\star}{\lambda_l} \big(\bm{u}_j^{\star \top} \bm{u}_l \big) \left\{  \sum_{s=0}^{\infty} \frac{1}{\lambda_l^s} \bm{H}^s\bm{u}_j^\star\right\} .
	\end{equation}
	\end{thm}
	\begin{proof} We supply the proof in Appendix \ref{appendix:proof-thm_Neumann} for self-containedness. \end{proof}
	\begin{remark}
	In particular, if $\bm{M}^{\star}$ is a rank-1 matrix and $\|\bm{H}\|< |\lambda_1|$, then 
	\begin{equation}
		\bm{u}_1 =  \frac{\lambda_1^\star}{\lambda_1} \big(\bm{u}_1^{\star \top} \bm{u}_1 \big) \left\{  \sum_{s=0}^{\infty} \frac{1}{\lambda_1^s} \bm{H}^s\bm{u}_1^\star\right\} .
	\end{equation}
	\end{remark}

	An immediate consequence of the Neumann trick is the following lemma, which asserts that each of the top-$r$ eigenvectors of $\bm{M}$ resides almost within the  top-$r$ eigen-subspace of $\bm{M}^{\star}$, provided that $\|\bm{H}\|$ is sufficiently small. The proof is deferred to Appendix \ref{appendix:proof-lem_l2_convergence}. 
	\begin{lem} 
	\label{lem_l2_convergence-rank-r} 
	Suppose $\bm{M}^{\star}$ is a rank-$r$ symmetric matrix with $r$ non-zero eigenvalues obeying $1=\lambda_{\max}^{\star} = \big|\lambda_1^{\star} \big| \geq \cdots \geq \big|\lambda_r^{\star}\big| = \lambda_{\min}^{\star}>0$ and associated eigenvectors $\bm{u}_1^{\star},\cdots,\bm{u}_r^{\star}$. Define $\kappa \triangleq \lambda_{\mathrm{max}}^{\star} / \lambda_{\mathrm{min}}^{\star}$. 
	If $\left\|\bm{H}\right\| \leq 1/(4\kappa)$, then the top-$r$ eigenvectors $\bm{u}_1,\cdots,\bm{u}_r$ of $\bm{M}=\bm{M}^{\star}+\bm{H}$ obey 
	\begin{equation}
	\begin{aligned}
		\sum_{j=1}^r |\bm{u}_j^{\star\top}\bm{u}_l|^2 & \geq 1- \frac{64\kappa^4}{9} \left\|\bm{H}\right\|^2, \qquad  1\leq l\leq r. 
	\end{aligned}
	\end{equation}
	In addition, if $r=1$, then one further has
	\begin{align}
		\min\{\|\bm{u}_1-\bm{u}_1^{\star}\|_{2},\|\bm{u}_1+\bm{u}_1^{\star}\|_{2}\}  \leq\frac{8\sqrt{2}}{3}\|\bm{H}\|.
	\end{align}
	\end{lem}

	\section{Perturbation analysis for the rank-1 case}
	\label{sec:rank-1}

	\subsection{Main results: the rank-1 case}
	This section presents perturbation analysis results when the truth $\bm{M}^{\star}$ is a symmetric rank-1 matrix. We shall start by presenting a master bound which, as we will see, immediately  leads to our main findings.
	%
	%

	\subsubsection{A master bound}

	Our master bound is concerned with the perturbation of linear forms of eigenvectors, as stated below.

	\begin{thm}[\textbf{Perturbation of linear forms of eigenvectors (rank-1)}] \label{thm_Linear_Forms_1}
		Consider a rank-1 symmetric matrix $\bm{M}^\star = \lambda^{\star} \bm{u}^\star \bm{u}^{\star \top}\in \mathbb{R}^{n\times n}$ with incoherence parameter $\mu$ (cf.~Definition \ref{def_incoherence_parameter}).  Suppose the noise matrix $\bm{H}$ obeys Assumption \ref{assumption:H}, and assume the existence of some sufficiently small constant $c_1>0$ such that
		\begin{equation}
			\label{Bsigma_cond_2}
			\max\left\{\sigma \sqrt{n \log n}, B \log n\right\} \leq c_1  \big| \lambda^{\star} \big| .
		\end{equation}
		Then for any fixed  vector $\bm{a} \in \mathbb{R}^{n}$ with $\|\bm{a}\|_2=1$,  with probability at least $1-O(n^{-10})$ one has
		\begin{equation}
			\left|\bm{a}^\top \left(\bm{u}-\frac{\bm{u}^{\star\top} \bm{u}}{\lambda/\lambda^{\star}} \bm{u}^\star\right) \right| \lesssim  \frac{ \max\left\{ \sigma \sqrt{n  \log n }, B \log n\right\} }{ \big| \lambda^{\star} \big| } \sqrt{\frac{\mu}{n}}.
		\end{equation}
	\end{thm}

\begin{remark}[The noise size]\label{remak:noise-size} We would like to  remark on the range of the noise size covered by our theory.
If the incoherence parameter of the  truth $\bm{M}^{\star}$ obeys $\mu\asymp 1$, then even the magnitude of the largest entry of $\bm{M}^{\star}$ cannot exceed the order of $|\lambda^{\star}|/n$. 	One can thus interpret the condition \eqref{Bsigma_cond_2} in this case as
\[
\sigma\lesssim\sqrt{\frac{n}{\log n}} \,\|\bm{M}^{\star}\|_{\infty}\quad\text{and}\qquad B\lesssim\frac{n}{\log n} \,\|\bm{M}^{\star}\|_{\infty}.
\]
In other words, the standard deviation $\sigma$ of each noise component is allowed to be substantially larger (i.e.~$\sqrt{n/\log n}$ times larger) than the magnitude of any of the true entries. In fact, this condition \eqref{Bsigma_cond_2} matches, up to some log factor, the one required for spectral methods to perform noticeably better than random guessing.
\end{remark}

	In words, Theorem \ref{thm_Linear_Forms_1} tells us that: the quantity $\frac{\bm{u}^{\star\top} \bm{u}}{\lambda/\lambda^{\star}} \bm{a}^{\top}\bm{u}^\star$ serves as a  remarkably accurate approximation  of the linear form $\bm{a}^{\top}\bm{u}$.
	In particular,  the approximation error is at most  $O(1/\sqrt{n})$ under the condition \eqref{Bsigma_cond_2} for incoherent matrices.
 Encouragingly, this approximation accuracy holds true for an arbitrary deterministic direction (reflected by  $\bm{a}$). As a consequence, one can roughly interpret Theorem \ref{thm_Linear_Forms_1} as
\begin{align}
	\label{eq:approx-u-M}
	\bm{u}\approx  \frac{\lambda^{\star}}{\lambda} \bm{u}^{\star}\bm{u}^{\star\top}\bm{u} =\frac{1}{\lambda} \bm{M}^{\star}\bm{u} ,
\end{align}
where such an approximation is fairly accurate along any fixed direction.  Compared with the identity $\bm{u}=  \frac{1}{\lambda} \bm{M}\bm{u}= \frac{1}{\lambda} (\bm{M}^{\star}+\bm{H})\bm{u}$, our results imply that $\bm{H}\bm{u}$ is exceedingly small along any fixed direction, even though $\bm{H}$ and $\bm{u}$ are highly dependent.
As we shall explain in Section \ref{sec:intuition}, this observation usually cannot happen when $\bm{H}$ is a symmetric random matrix or when one uses the leading singular vector instead, due to the significant bias resulting from symmetry.

	This master theorem has several interesting implications, as we shall elucidate momentarily.

	\subsubsection{Eigenvalue perturbation}
	\label{sec:eigenvalue-perturb-rank1}

	To begin with, Theorem \ref{thm_Linear_Forms_1} immediately yields a much sharper non-asymptotic perturbation bound  regarding the leading eigenvalue $\lambda$ of $\bm{M}$.
	
	\begin{cor}
		\label{cor:eigenvalue-pertubation}
		Under the assumptions of Theorem \ref{thm_Linear_Forms_1}, with probability at least $1-O(n^{-10})$ we have
		\begin{equation}
			\label{eq:eigenvalue-perturbation}
			\left|\lambda- \lambda^{\star} \right| \lesssim \max\left\{ \sigma \sqrt{n  \log n }, B \log n\right\} \sqrt{\frac{\mu}{n}}.
		\end{equation}
	\end{cor}

	\begin{proof}
		Without loss of generality, assume that $\lambda^{\star}=1$. Taking $\bm{a} = \bm{u}^\star$ in Theorem \ref{thm_Linear_Forms_1}, we get
	\begin{align}
		\label{eq:u-corr-UB}
		\big|\bm{u}^{\star\top}\bm{u}\big|\frac{|\lambda-1|}{\lambda}=\Big|\bm{u}^{\star\top}\bm{u}-\bm{u}^{\star\top}\bm{u}^{\star}\frac{\bm{u}^{\star\top}\bm{u}}{\lambda}\Big|
		\lesssim\max\left\{ \sigma \sqrt{n  \log n }, B \log n\right\} \sqrt{\frac{\mu}{n}}.
	\end{align}
		From Lemma \ref{lem:H-norm} and the condition \eqref{Bsigma_cond_2}, we know $\|\bm{H}\|<1/4$, which combines with Lemma \ref{lem:perturbation-BF-rank-r} and Lemma \ref{lem_l2_convergence-rank-r} yields $\lambda\asymp |\bm{u}^{\star\top}\bm{u}| \asymp 1$. Substitution into  \eqref{eq:u-corr-UB} yields
	\begin{align*}
		|\lambda - 1| \lesssim \left|\frac{\lambda}{\bm{u}^{\star\top} \bm{u}}\right| \max\left\{\sigma \sqrt{n  \log n }, B \log n\right\} \sqrt{\frac{\mu}{n}} \lesssim \max\left\{\sigma \sqrt{n  \log n }, B \log n\right\} \sqrt{\frac{\mu}{n}}.
	\end{align*}
	\end{proof}

	For the vast majority of applications we encounter, the maximum possible noise magnitude $B$  (cf.~Assumption \ref{assumption:H})  obeys $B\lesssim \sigma \sqrt{n/\log n}$, in which case the bound in  Corollary \ref{cor:eigenvalue-pertubation} simplifies to
	\begin{equation}
		\label{eq:eigenvalue-perturbation-simple}
		\left|\lambda- \lambda^{\star} \right| \lesssim  \sigma \sqrt{\mu  \log n } .
	\end{equation}
	This means that the eigenvalue estimation error is not much larger than the variability of each noise component.
	In addition, we remind the reader that for a fairly broad class of noise (see Remark \ref{remak:noise-size}), the leading eigenvalue $\lambda$ of $\bm{M}$ is guaranteed to be real-valued, an observation that has been made in Lemma \ref{lem:perturbation-BF-rank-r}.  In practice, however, one might still encounter some scenarios where $\lambda$ is complex-valued. As a result, we recommend the practitioner to use the real part of $\lambda$ as the eigenvalue estimate,  which clearly enjoys the same statistical guarantee as in Corollary \ref{cor:eigenvalue-pertubation}.

	\paragraph{Comparison to the vanilla SVD-based approach.} In order to facilitate comparison, we denote by $\lambda_{\mathsf{svd}}$ the largest singular value of $\bm{M}$, and look at $|\lambda_{\mathsf{svd}} - \lambda^{\star}|$. Combining Weyl's inequality, Lemma \ref{lem:H-norm} and the condition \eqref{Bsigma_cond_2}, we arrive at
	\begin{equation}
		|\lambda_{\mathsf{svd}} - \lambda^{\star}| \leq \|\bm{H}\| \lesssim \max\left\{\sigma \sqrt{n  \log n }, B \log n\right\}  .
	\end{equation}
	When $\mu\asymp 1$, this error bound w.r.t.~this (unadjusted) singular value could be $\sqrt{n}$ times larger than the perturbation bound \eqref{eq:eigenvalue-perturbation} derived for the leading eigenvalue. This corroborates our motivating experiments in Fig.~\ref{fig:eigenvalue-perturb}.

	\paragraph{Comparison to vanilla eigen-decomposition after symmetrization.} The reader might naturally wonder what would happen if we symmetrize the data matrix before performing eigen-decomposition. Consider, for example, the i.i.d.~Gaussian noise case where $H_{ij} \overset{\text{i.i.d.}}{\sim} \mathcal{N}(0,\sigma^2)$, and assume $\lambda^{\star}>0$ for simplicity.
	The leading eigenvalue $\lambda_{\mathsf{sym}}$  of the symmetrized matrix $(\bm{M}+\bm{M}^{\top})/2$ has been extensively studied in the literature
	\cite{furedi1981eigenvalues,yin1988limit,peche2006largest,feral2007largest,benaych2011eigenvalues,renfrew2013finite,knowles2013isotropic}.
	In particular, it has been shown (e.g.~\cite{capitaine2009largest})  that, with probability approaching one,
\begin{align}
	\label{eq:lambda-sym-approx}
	\lambda_{\mathsf{sym}}=\lambda^{\star}+\frac{n\sigma^{2}}{2\lambda^{\star}}+O\big(\sigma\sqrt{\log n}\big) .
\end{align}
If $\sigma = | \lambda^{\star} | \sqrt{1/(n\log n)} $ (which is the setting in our numerical experiment), then this can be translated into
\[
	\frac{\lambda_{\mathsf{sym}}-\lambda^{\star}}{\lambda^{\star}}=\frac{1}{2\log n} + O\left( \frac{1}{\sqrt{n}} \right).
\]
This implies that $\lambda_{\mathsf{sym}}$ suffers from a substantially larger bias than the leading eigenvalue $\lambda$ obtained without symmetrization, since in this case we have (cf.~Corollary \ref{cor:eigenvalue-pertubation})
\begin{align}
	\label{eq:lambda-accuracy-incoherent}
	\left|\frac{\lambda-\lambda^{\star}}{\lambda^{\star}}\right|\lesssim \frac{1}{\sqrt{n}}.
\end{align}

	\paragraph{Comparison to properly adjusted eigen-decomposition and SVD-based methods.}  Armed with the approximation \eqref{eq:lambda-sym-approx} in the i.i.d.~Gaussian noise case, the careful reader might naturally suggest a properly corrected estimate $\lambda_{\mathsf{sym},\mathsf{c}}$ as follows (again assuming $\lambda >0$)
\begin{equation}
	\lambda_{\mathsf{sym},\mathsf{c}}=\frac{1}{2}\Big(\lambda_{\mathsf{sym}}+\sqrt{\lambda_{\mathsf{sym}}^{2}-2n\sigma^{2}}\Big),
	\label{eq:lambda-adjust}
\end{equation}
which is a shrinkage-type estimate chosen to satisfy $\lambda_{\mathsf{sym}}= \lambda_{\mathsf{sym},\mathsf{c}} +\frac{n\sigma^{2}}{2\lambda_{\mathsf{sym},\mathsf{c}}}$. A little algebra reveals that: if $\sigma = 1 / \sqrt{n\log n}$, then
\[
	\left|\frac{ \lambda_{\mathsf{sym},\mathsf{c}} -\lambda^{\star}}{\lambda^{\star}}\right| \lesssim \frac{1}{\sqrt{n}},
\]
thus matching the estimation accuracy of $\lambda$ (cf.~\eqref{eq:lambda-accuracy-incoherent}). In addition, some sort of universality results has been established as well in the literature \citep{capitaine2009largest}, implying that the same approximation and correction are applicable to a broad family of zero-mean noise with identical variance. As we shall illustrate numerically in Section \ref{sec:applications},  this approach (i.e.~$\lambda_{\mathsf{sym},\mathsf{c}}$) performs almost identically to the one using vanilla eigen-decomposition without symmetrization.
In addition, very similar observations have been made for the SVD-based approach \citep{silverstein1994spectral,yin1988limit,peche2006largest,feral2007largest,benaych2012singular,bryc2018singular}; for the sake of brevity, we do not repeat the arguments here.

We would nevertheless like to single out a few statistical advantages of the eigen-decomposition approach without symmetrization. To begin with, $\lambda$ is obtained via vanilla eigen-decomposition, and computing it does not rely on any kind of noise statistics. This is in stark contrast to the bias correction \eqref{eq:lambda-adjust} in the presence of symmetric data, which requires prior knowledge about (or a very precise estimate of)   the noise variance $\sigma^2$. Leaving out this prior knowledge matter,  a more important issue is that the approximation formula \eqref{eq:lambda-sym-approx} assumes identical variance of noise components across all  entries (i.e.~homoscedasticity). While an approximation of this kind has been found for more general cases beyond homoscedastic noise (e.g.~\cite{bryc2018singular}), the approximation formula (e.g.~\cite[Theorem 1.1]{bryc2018singular}) becomes fairly complicated, requires  prior knowledge about all variance parameters, and is thus difficult to implement in practice.  In comparison,  the vanilla eigen-decomposition approach analyzed in Corollary \ref{cor:eigenvalue-pertubation} imposes no restriction on the noise statistics and is fully adaptive to heteroscedastic noise.

\paragraph{Lower bounds.} To complete the picture, we provide a simple information-theoretic lower bound for the i.i.d.~Gaussian noise case, which will be established in Appendix \ref{sec:proof-lower-bound}.

\begin{lem}\label{lem:lower-bound} Fix any small constant $\varepsilon>0$. Suppose that $H_{ij}\overset{\text{i.i.d.}}{\sim}\mathcal{N}(0,\sigma^{2})$.
Consider three matrices
\[
\bm{M}=\lambda^{\star}\bm{u}^{\star}\bm{u}^{\star\top}+\bm{H},\qquad\widetilde{\bm{M}}=\left(\lambda^{\star}+\Delta\right)\bm{u}^{\star}\bm{u}^{\star\top}+\bm{H},\qquad\widehat{\bm{M}}=\left(\lambda^{\star}-\Delta\right)\bm{u}^{\star}\bm{u}^{\star\top}+\bm{H}
\]
with $\|\bm{u}^{\star}\|_{2}=1$.
If $\Delta\leq \sigma \sqrt{(\log_{2}1.5-\varepsilon)\log2}$,
then no algorithm can distinguish $\bm{M}$, $\widetilde{\bm{M}}$
and $\widehat{\bm{M}}$ with $p_{\mathrm{e}} \leq \varepsilon$,
where $p_{\mathrm{e}}$ is the minimax probability of error for testing
three hypotheses (namely, the ones claiming that the true eigenvalues are $\lambda^{\star}$,
$\lambda^{\star}+\Delta$, and $\lambda^{\star}-\Delta$, respectively). \end{lem}

In short, Lemma \ref{lem:lower-bound} asserts that one cannot possibly locate an eigenvalue to within a precision of $\Delta$ much better than $\sigma$, which reveals a fundamental  limit that cannot be broken by any algorithm.  In comparison, the vanilla eigen-decomposition method based on asymmetric data achieves an accuracy of $| \lambda - \lambda^{\star} | \lesssim \sigma \sqrt{\log n}$ (cf.~Corollary \ref{cor:eigenvalue-pertubation} and \eqref{eq:eigenvalue-perturbation-simple}) for the incoherent case, thus matching the information-theoretic lower bound up to some log factor. In fact, the extra $\sqrt{\log n}$ factor arises simply because we are aiming for a high-probability guarantee.

	\subsubsection{Perturbation of linear forms of eigenvectors}
	The master bound in Theorem \ref{thm_Linear_Forms_1} admits a more convenient form when controlling linear functions of the eigenvectors. The result is this:
	\begin{cor}
		\label{cor:linear-form-rank1}
		Under the same setting of Theorem \ref{thm_Linear_Forms_1},  with probability at least $1-O(n^{-10})$ we have
	\begin{equation}
		\min\big\{ \big|\bm{a}^{\top} (\bm{u} - \bm{u}^{\star})\big|, \big|\bm{a}^{\top} (\bm{u} + \bm{u}^{\star})\big| \big\}   \lesssim  \left( |\bm{a}^{\top} \bm{u}^{\star}| + \sqrt{\frac{\mu}{n}}\right) \frac{ \max\left\{\sigma \sqrt{n \log n }, B \log n \right\} }{ \big| \lambda^{\star} \big| }.
	\end{equation}
	\end{cor}

	\begin{proof}
		Without loss of generality, assume that $\bm{u}^{\star\top}\bm{u}\geq0$ and that $\lambda^{\star}=1$. Then one has
	\begin{align*}
		| \bm{a}^{\top} (\bm{u} - \bm{u}^{\star} ) | & \leq\Big|\bm{a}^{\top}\bm{u}-\bm{a}^{\top}\bm{u}^{\star}\frac{\bm{u}^{\star\top}\bm{u}}{\lambda}\Big|+\big|\bm{a}^{\top}\bm{u}^{\star}\big|\left|\frac{\bm{u}^{\star\top}\bm{u}}{\lambda}-1\right| \\
		& \leq  \max\left\{\sigma \sqrt{n  \log n }, B \log n \right\}  \sqrt{\frac{\mu}{n}}+ |\bm{a}^{\top}\bm{u}^{\star}| \left|\frac{\bm{u}^{\star\top}\bm{u}}{\lambda}-1\right|,
	\end{align*}
	where the last inequality arises from Theorem \ref{thm_Linear_Forms_1} as well as the definition of $\mu$.  In addition, apply Lemma \ref{lem:perturbation-BF-rank-r} and Lemma \ref{lem_l2_convergence-rank-r} to obtain
	\begin{align*}
		\left|\frac{\bm{u}^{\star\top}\bm{u}}{\lambda}-1\right| & \leq\frac{\bm{u}^{\star\top}\bm{u}}{\lambda}\left|1-\lambda\right|+\left|\bm{u}^{\star\top}\bm{u}-1\right|
		\lesssim \|\bm{H}\|  \lesssim \max\left\{\sigma \sqrt{n \log n }, B \log n\right\}.
	\end{align*}
	Putting the above bounds together concludes the proof.
	\end{proof}
	The perturbation of linear forms of eigenvectors (or singular vectors) has not yet been well explored even for the symmetric case. One scenario that has been studied is linear forms of singular vectors under i.i.d.~Gaussian noise  \citep{koltchinskii2016perturbation,xia2016statistical}. Our analysis --- which is certainly different from \cite{koltchinskii2016perturbation} as our emphasis is eigen-decomposition --- does not rely on the Gaussianality assumption, and accommodates a much broader class of random noise.  Another work that has looked at linear forms of the leading singular vector is \cite{ma2017implicit} for phase retrieval and blind deconvolution, although the vector $\bm{a}$ therein is  specific to the problems (i.e.~the design vectors) and cannot be made general.

	\begin{remark}
		The perturbation theory for linear forms of eigenvectors has been substantially extended in our follow-up work; the interested reader is referred to \cite{cheng2020inference} for details. 
	\end{remark}

	\subsubsection{Entrywise eigenvector perturbation} A straightforward consequence of Corollary \ref{cor:linear-form-rank1} that is worth emphasizing is sharp entrywise control of the leading eigenvector as follows.
	\begin{cor}
		\label{cor:entrywise-rank1}
		Under the same setting of Theorem \ref{thm_Linear_Forms_1},  with probability at least $1-O(n^{-9})$ we have
	\begin{equation}
		\min\left\{ \left\|\bm{u} - \bm{u}^{\star}\right\|_\infty, \left\|\bm{u} + \bm{u}^{\star}\right\|_\infty \right\} \lesssim \frac{\max\left\{ \sigma \sqrt{n  \log n }, B \log n \right\} }{ \big| \lambda^{\star} \big|} \sqrt{\frac{\mu}{n}}.
	\end{equation}
	\end{cor}
	
	\begin{proof}
		Recognizing that $\|\bm{u}-\bm{u}^{\star}\|_{\infty}= \max_{i} | \bm{e}_i^{\top}\bm{u} - \bm{e}_i^{\top} \bm{u}^{\star} |$ and recalling our assumption $|\bm{e}_i^{\top}\bm{u}|\leq \sqrt{{\mu}/{n}}$, we can invoke Corollary \ref{cor:linear-form-rank1} and the union bound to establish this entrywise bound.
	\end{proof}

We note that: while the  $\ell_2$ perturbation (or $\sin\bm{\Theta}$ distance) of eigenvectors or singular vectors has been extensively studied \cite{davis1970rotation,wedin1972perturbation,vu2011singular,wang2015singular,o2013random,cai2018rate}, the entrywise eigenvector behavior was much less explored.  The prior literature contains only a few entrywise eigenvector perturbation analysis results for  settings very different from ours, e.g.~the i.i.d.~random matrix case \cite{vu2015random,o2016eigenvectors}, the symmetric low-rank case \cite{fan2018eigenvector,abbe2017entrywise,pmlr-v83-eldridge18a}, and the case with transition matrices for reversible Markov chains \citep{chen2017spectral}. Our results add another instance to this body of works in providing entrywise eigenvector perburation bounds.  	

	\subsection{Applications}\label{sec:applications}

	We apply our main results to two concrete matrix estimation problems and examine the effectiveness of these bounds. As before, $\bm{M}^{\star}$ is a rank-1 matrix with incoherence parameter $\mu$ and leading eigenvalue $\lambda^{\star}$.

	\bigskip
	\noindent {\bf Low-rank matrix estimation from Gaussian noise.} Suppose that $\bm{H}$ is composed of i.i.d.~Gaussian random variables $\mathcal{N}(0,\sigma^2)$.\footnote{In this case, one can take $B \asymp \sigma  \sqrt{ \log n }$, which clearly satisfies $B \log n \ll  \sqrt{n\sigma^2\log n}$. }
	%
	If $\sigma \lesssim \frac{1}{\sqrt{n\log n}}$, applying Corollaries \ref{cor:eigenvalue-pertubation}-\ref{cor:entrywise-rank1} reveals that with high probability,
	\begin{subequations} \label{eq:error-rank-1-Gauss-symm}
	\begin{align}
		|\lambda - \lambda^{\star}| &\lesssim \sigma \sqrt{  \mu \log n } \\
		\min\{ \|\bm{u} - \bm{u}^{\star}\|_{\infty}, \|\bm{u} + \bm{u}^{\star}\|_{\infty} \} & \lesssim  \frac{ \sigma \sqrt{  \mu \log n } }{ \big| \lambda^{\star} \big|} \\
		\min\{ |\bm{a}^{\top}(\bm{u} - \bm{u}^{\star})|, |\bm{a}^{\top}(\bm{u} + \bm{u}^{\star})| \} & \lesssim \left( |\bm{a}^{\top} \bm{u}^{\star} | + \sqrt{\frac{\mu}{n}} \right)  \frac{ \sigma \sqrt{ n \log n} }{ \big| \lambda^{\star} \big|}
	\end{align}
	\end{subequations}
	for any fixed unit vector $\bm{a}\in \mathbb{R}^n$. We have conducted additional numerical experiments in Fig.~\ref{fig:gaussian_varyn}, which confirm our findings.  It is also worth noting that empirically, eigen-decomposition and SVD applied to $\bm{M}$ achieve nearly identical $\ell_2$ and $\ell_{\infty}$ errors when estimating the leading eigenvector of $\bm{M}^{\star}$.  In addition, we also include the numerical estimation error of the corrected eigenvalue $\lambda_{\mathsf{sym},\mathsf{c}}$ (cf.~\eqref{eq:lambda-adjust}) of the symmetrized matrix $(\bm{M}+\bm{M}^{\top})/2$. As can be seen from Fig.~\ref{fig:gaussian_varyn}, vanilla eigen-decomposition  without symmetrization performs nearly  identically to the one with symmetrization and proper correction. 	
	

	\begin{figure}[htbp!]
		\centering
		\begin{tabular}{ccc}
			  \includegraphics[width=0.32\linewidth]{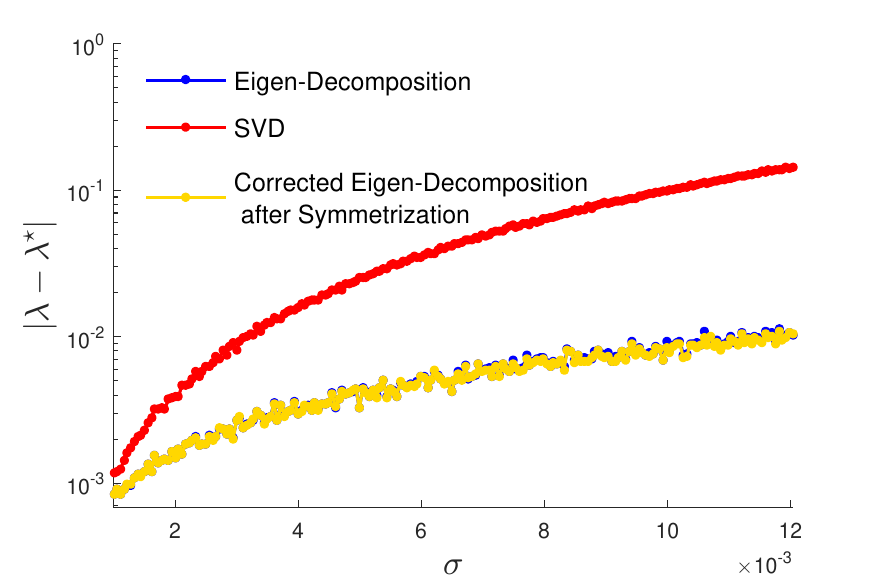}
			& \includegraphics[width=0.32\linewidth]{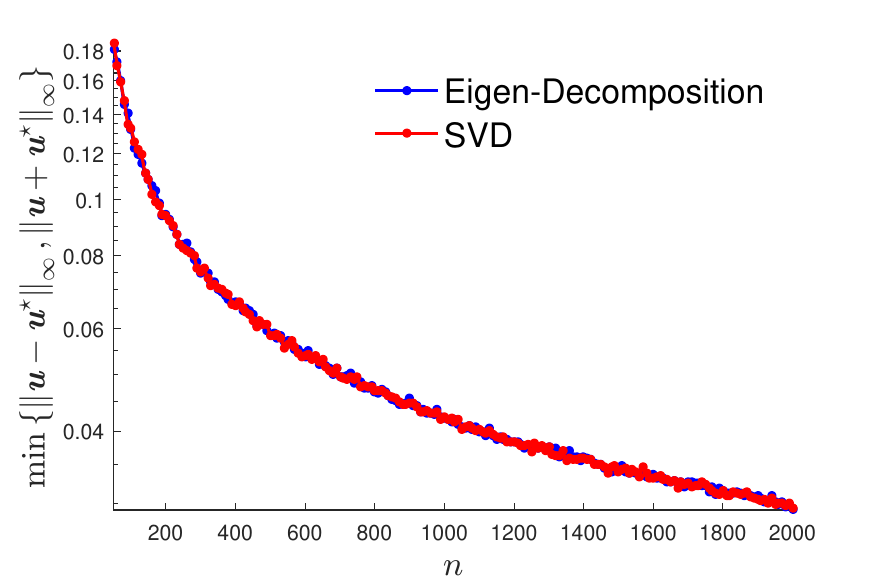}
			& \includegraphics[width=0.32\linewidth]{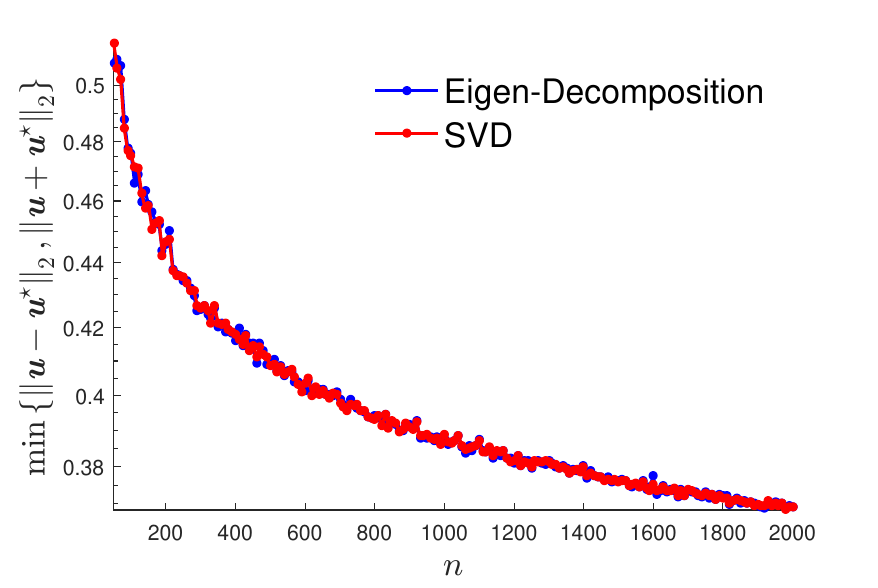}
		\tabularnewline
			(a) eigenvalue perturbation & (b) $\ell_{\infty}$ eigenvector perturbation & (c) $\ell_2$ eigenvector perturbation \tabularnewline
		\end{tabular}

		\caption{Numerical simulation for rank-1 matrix estimation under i.i.d.~Gaussian noise $\mathcal{N}(0,\sigma^2)$, where the rank-1 truth $\bm{M}^{\star}$ is generated randomly with leading eigenvalue 1. (a): $|\lambda- \lambda^{\star}|$ vs.~$\sigma$ when $n=1000$; (b) and (c):   $\ell_{\infty}$ and $\ell_2$ eigenvector estimation errors vs.~$n$ with $\sigma = 1 / \sqrt{n \log n}$, respectively. 		
		The blue (resp.~red) lines represent the average errors over 100 independent trials using the vanilla eigen-decomposition (resp.~SVD) approach applied to $\bm{M}$. The orange line in (a) represents the average errors over 100 independent trials using the corrected leading eigenvalue $\lambda_{\mathsf{sym},\mathsf{c}}$ of the symmetrized  matrix $(\bm{M} + \bm{M}^\top)/2$ (cf.~\eqref{eq:lambda-adjust}).
		} \label{fig:gaussian_varyn}
	\end{figure}

	\bigskip
	\noindent {\bf Low-rank matrix completion.} Suppose that $\bm{M}$ is generated using random partial entries of $\bm{M}^{\star}$ as follows
	\begin{equation}
		M_{ij}= \begin{cases} \frac{1}{p} M_{ij}^{\star},  \qquad & \text{with probability }p, \\ 0, & \text{else}, \end{cases}
	\end{equation}
	where $p$ denotes the fraction of the entries of $\bm{M}^{\star}$ being revealed.  It is straightforward to verify that $\bm{H}=\bm{M}-\bm{M}^{\star}$ is zero-mean and obeys $|H_{ij}| \leq \frac{\mu}{np}:=B$ and $\mathsf{Var}(H_{ij})\leq \frac{\mu^2}{pn^2}$. Consequently, if $p \gtrsim \frac{\mu^2 \log n}{n}$, then invoking Corollaries \ref{cor:eigenvalue-pertubation}-\ref{cor:entrywise-rank1} yields
	\begin{subequations}
	\label{eq:error-rank-1-MC-symm}
	\begin{align}
		\frac{ |\lambda - \lambda^{\star}| }{ \big| \lambda^{\star} \big|} & \lesssim  \frac{1}{\sqrt{n}} \sqrt{ \frac{\mu^3 \log n}{p n} }  \\
		\min\{\|\bm{u} - \bm{u}^{\star}\|_{\infty}, \|\bm{u} + \bm{u}^{\star}\|_{\infty} \} & \lesssim \frac{1}{\sqrt{n}} \sqrt{ \frac{\mu^3 \log n}{p n} } \\
		\min\{|\bm{a}^{\top}(\bm{u}-\bm{u}^{\star})|, |\bm{a}^{\top}(\bm{u}+\bm{u}^{\star})|\} & \lesssim\left(|\bm{a}^{\top}\bm{u}^{\star}|+\sqrt{\frac{\mu}{n}}\right)\sqrt{ \frac{\mu^{2}\log n}{p n}  }
	\end{align}
	\end{subequations}
	with high probability, where $\bm{a}\in \mathbb{R}^n$ is any fixed unit vector. Additional numerical simulations have been carried out in Fig.~\ref{fig:mc_varyn} to verify these findings. Empirically, eigen-decomposition outperforms SVD in estimating both the leading eigenvalue and  eigenvector of $\bm{M}^{\star}$.

	\begin{figure}[htbp!]
		\centering
			
		\begin{tabular}{ccc}

		\includegraphics[width=0.32\linewidth]{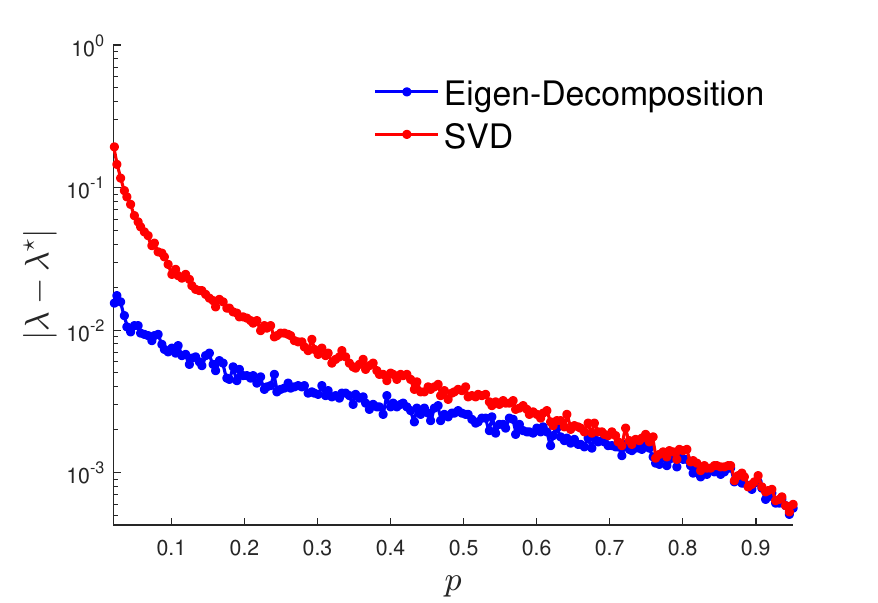}
			& \includegraphics[width=0.32\linewidth]{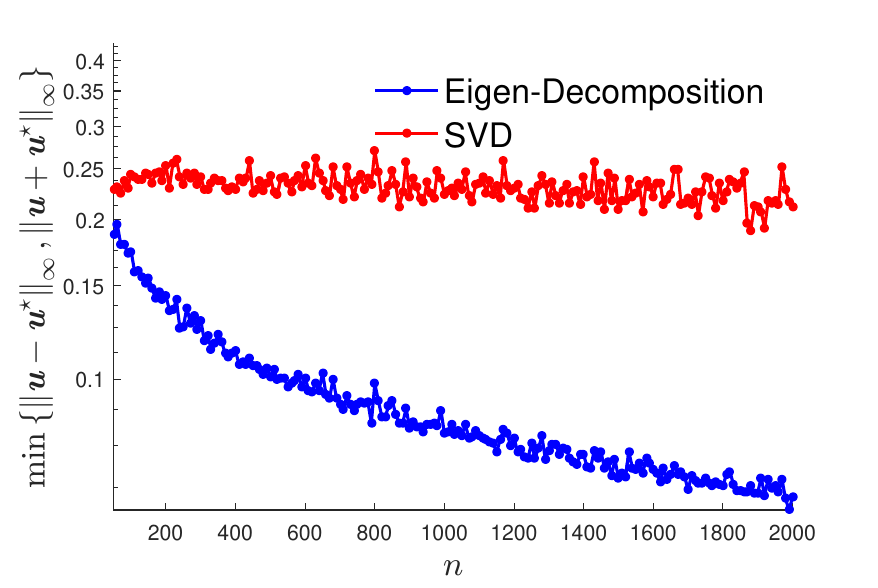}
			& \includegraphics[width=0.32\linewidth]{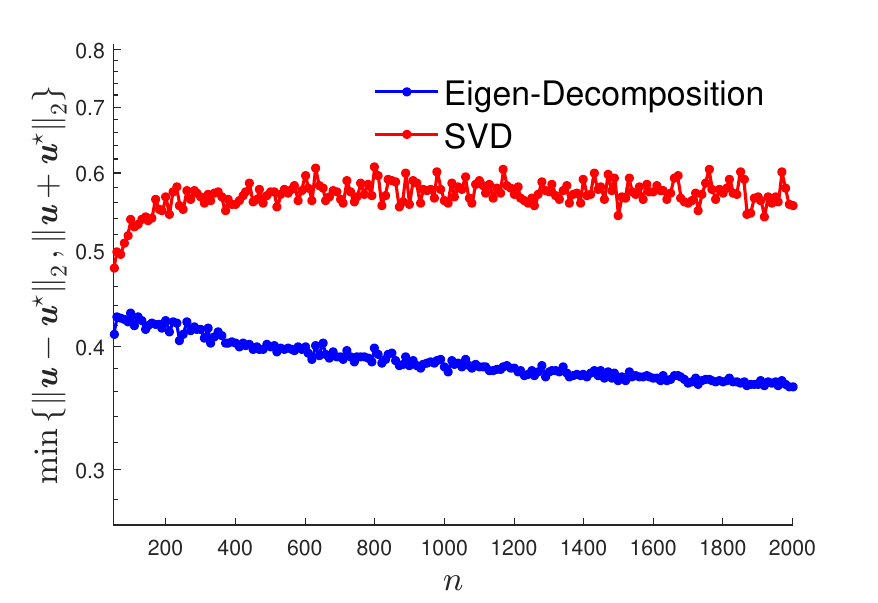}

		\tabularnewline
			(a) eigenvalue perturbation & (b) $\ell_{\infty}$ eigenvector perturbation & (c) $\ell_2$ eigenvector perturbation \tabularnewline
		\end{tabular}
		
		\caption{Numerical simulation for rank-1 matrix completion,  where the rank-1 truth $\bm{M}^{\star}$ is randomly generated with leading eigenvalue 1 and  sampling rate is $p=3\log n/n$.  (a) $|\lambda- \lambda^{\star}|$ vs.~$p$ when $n=1000$; (b) and (c):  $\ell_{\infty}$ and $\ell_2$ eigenvector estimation errors vs.~$n$, respectively.
		The blue (resp.~red) lines represent the average errors over 100 independent trials using the eigen-decomposition (resp.~SVD) approach.
		\label{fig:mc_varyn}}
	\end{figure}

	\bigskip
	\noindent Finally, we remark that all  the above applications assume the availability of an asymmetric data matrix $\bm{M}$. One might naturally wonder whether there is anything useful we can say if  only a symmetric matrix $\bm{M}$ is available. While this is  in general difficult,  our theory does have direct implications for both matrix completion and the case with i.i.d.~Gaussian noise in the presence of symmetric data matrices; that is, it is possible to first asymmetrize the data matrix followed by eigen-decomposition.  The interested reader is referred to Appendix \ref{sec:asymmetrization-rank1} for details.

	\subsection{Why asymmetry helps?}
	\label{sec:intuition}
	
	We take a moment to develop some intuition underlying  Theorem \ref{thm_Linear_Forms_1}, focusing on the case with $\lambda^{\star}=1$ for simplicity. The key ingredient is the Neumann trick stated in Theorem \ref{thm_Neumann}. Specifically, in the rank-1 case we can expand
	$$
	\bm{u} = \frac{1}{\lambda}\left(\bm{u}^{\star \top} \bm{u}\right)\sum\nolimits_{s=0}^{\infty} \frac{1}{\lambda^s} \bm{H}^s \bm{u}^\star .
	$$
 	A little algebra yields	
	\begin{align} \label{eq:au-UB}
	\left|\bm{a}^\top \left(\bm{u}-\frac{\bm{u}^{\star\top} \bm{u}}{\lambda} \bm{u}^\star\right)\right| = \left|\frac{\bm{u}^{\star\top} \bm{u}}{\lambda} \sum_{s=1}^{\infty} \frac{\bm{a}^\top \bm{H}^s \bm{u}^\star}{\lambda^s} \right| 
\lesssim \sum_{s=1}^{\infty} \left| \frac{\bm{a}^\top \bm{H}^s \bm{u}^\star}{\lambda^s} \right|,
	\end{align}
	where the last inequality holds since (i) $|\bm{u}^{\star\top} \bm{u}|\leq 1$, and (ii)  $\lambda$ is real-valued and obeys $\lambda\approx 1$ if $\|\bm{H}\|\ll 1$ (in view of Lemma \ref{lem:perturbation-BF-rank-r}). As a result, the perturbation can be well-controlled as long as  $| \bm{a}^\top \bm{H}^s \bm{u}^\star |$ is small for every $s\geq 1$.

As it turns out, $\bm{a}^{\top}\bm{H}^{s}\bm{u}^{\star}$ might be much
better controlled when $\bm{H}$ is random and asymmetric, in comparison to
 the case where $\bm{H}$ is random and symmetric. To illustrate
this point, it is perhaps the easiest to inspect the second-order term.
\begin{itemize}
\item \textbf{Asymmetric case:} when $\bm{H}$ is composed of independent
zero-mean entries each with variance $\sigma_n^{2}$, one has
\[
\mathbb{E}\left[\bm{a}^{\top}\bm{H}^{2}\bm{u}^{\star}\right]=\bm{a}^{\top}\mathbb{E}\left[\bm{H}^{2}\right]\bm{u}^{\star}=\bm{a}^{\top}(\sigma^{2}\bm{I})\bm{u}^{\star}=\sigma^{2}\bm{a}^{\top}\bm{u}^{\star}.
\]
\item \textbf{Symmetric case:} when $\bm{H}$ is symmetric and its upper
trangular part consists of independent zero-mean entries with variance
$\sigma_n^{2}$, it holds that
\[
\mathbb{E}\left[\bm{a}^{\top}\bm{H}^{2}\bm{u}^{\star}\right]=\bm{a}^{\top}\mathbb{E}\left[\bm{H}^{2}\right]\bm{u}^{\star}=\bm{a}^{\top}(n\sigma^{2}\bm{I})\bm{u}^{\star}=n\sigma^{2}\bm{a}^{\top}\bm{u}^{\star}.
\]
\end{itemize}
In words, the term $\bm{a}^{\top}\bm{H}^{2}\bm{u}^{\star}$ in
the symmetric case might have a significantly larger bias compared
to the asymmetric case. This bias effect is substantial when $\bm{a}^{\top}\bm{u}^{\star}$
is large (e.g.~when $\bm{a}=\bm{u}^{\star}$), which plays a crucial
role in determining the size of eigenvalue perturbation.
	
The vanilla SVD-based approach can be interpreted in a similar manner. Specifically, we recognize that the leading singular value (resp.~left singular vector) can be computed via the leading eigenvalue (resp.~eigenvector) of the symmetric matrix $\bm{M}\bm{M}^{\top}$. Given that  $\bm{M}\bm{M}^{\top} - \bm{M}^{\star}\bm{M}^{\star\top}$ is also symmetric, the aforementioned bias issue arises as well. This explains why vanilla eigen-decomposition might have an advantage over vanilla SVD when dealing with asymmetric matrices.

Finally, we remark that the aforementioned bias issue becomes less severe as $\|\bm{H}\|$ decreases. For example, when $\|\bm{H}\|$ is exceedingly small,  the only dominant term on the right-hand side of \eqref{eq:au-UB} is $\bm{a}^{\top}\bm{H} \bm{u}^{\star}$, with all higher-order terms being vanishingly small. In this case,   $\mathbb{E}[\bm{a}^{\top}\bm{H} \bm{u}^{\star}]=0$ for both symmetric and asymmetric zero-mean noise matrices.  As a consequence, the advantage of eigen-decomposition becomes negligible when dealing with nearly-zero noise.
 This observation is also confirmed in the numerical experiments reported in Fig.~\ref{fig:gaussian_varyn}(a) and Fig.~\ref{fig:mc_varyn}(a), where  the two approaches achieve similar eigenvalue estimation accuracy when $\sigma \rightarrow 0$ (resp.~$p\rightarrow 1$) in matrix estimation under Gaussian noise (resp.~matrix completion).  In fact, the case with very small $\|\bm{H}\|$ has been studied in the literature \cite{o2013random,vu2011singular,pmlr-v83-eldridge18a}. For example, it was shown in  \cite{o2013random} that when $\|\bm{H}\|\lesssim \frac{1}{\sqrt{n}} \| \bm{M}^{\star} \|$,  the singular value perturbation is also $\sqrt{n}$ times smaller than the bound predicted by Weyl's theorem; similar improvement can be observed w.r.t.~eigenvalue perturbation when $\bm{H}$ is symmetric (cf.~\cite[Theorem 6]{pmlr-v83-eldridge18a}). By contrast,
our eigenvalue perturbation results achieve this gain even when  $\|\bm{H}\|$ is nearly as large as  $\| \bm{M}^{\star} \|$ (up to some logarithmic factor).

	\subsection{Proof outline of Theorem \ref{thm_Linear_Forms_1}}

	This subsection outlines the main steps for establishing  Theorem \ref{thm_Linear_Forms_1}. To simplify presentation, we shall  assume without loss of generality that
	\begin{equation}
		\lambda^{\star} = 1 	.
	\end{equation}
	Throughout this paper, all the proofs are provided for the case when Conditions 1-3, 4(a) in Assumption \ref{assumption:H} are valid. Otherwise, if Condition 4(b) is valid, then we can invoke the union bound to show that
	\begin{equation}
		\bm{M} = \bm{M}^{\star} + \tilde{\bm{H}} 	
	\end{equation}
	with probability exceeding $1-O(n^{-10})$, where $\tilde{H}_{ij} \triangleq H_{ij} \ind _{\{ |H_{ij}| \leq B \}}$ is the truncated noise and has magnitude bounded by $B$.  Since $H_{ij}$ has symmetric distribution, it is seen that $\mathbb{E}[\tilde{H}_{ij}]=0$ and $\mathsf{Var}(\tilde{H}_{ij})\leq \sigma^2$, which coincides with the case obeying Conditions 1-3, 4(a) in Assumption \ref{assumption:H}.

	As already mentioned in Section \ref{sec:intuition}, everything boils down to controlling $| \bm{a}^\top \bm{H}^s \bm{u}^\star |$ for  $s\geq 1$. This is accomplished via the following lemma.
	
	\begin{lem}[\textbf{Bounding higher-order terms}] \label{lem_high_order_terms}
		Consider any fixed unit vector $\bm{a} \in \mathbb{R}^{n}$ and any  positive integers $s, k$ satisfying $Bsk \leq 2$ and $n\sigma^2 sk \leq 2$. Under the assumptions of Theorem \ref{thm_Linear_Forms_1},
	\begin{equation}
		\left|\mathbb{E} \left[\left(\bm{a}^\top \bm{H}^s \bm{u}^\star\right)^k\right]\right| \leq  \frac{sk}{2}  \max\left\{\left(Bsk\right)^{sk} ,\left(2n \sigma^2 sk\right)^{sk/2}\right\} \left(\sqrt{\frac{\mu}{n}}\right)^k .
	\end{equation}
	\end{lem}
	\begin{proof} The proof of Lemma \ref{lem_high_order_terms} is combinatorial in nature, which we defer to Appendix \ref{app_lem2}. \end{proof}
	
	\begin{remark}
		A similar result  in \cite[Lemma 2.3]{tao2013outliers} has studied the  bilinear forms of the high order terms of an i.i.d.~random matrix, with a few distinctions. First of all, \cite{tao2013outliers} assumes that each entry of the noise matrix is i.i.d.~and has finite fourth moment (if the noise variance is rescaled to be 1);  these assumptions break in examples like matrix completion.  Moreover, \cite{tao2013outliers}  focuses on the case with $k=2$,  and does not lead to high-probability bounds (which are crucial for, e.g.~entrywise error control).
	\end{remark}

	Using Markov's inequality and the union bound, we can translate Lemma \ref{lem_high_order_terms} into a high probability bound as follows.
	\begin{cor}
		\label{cor_high_order_terms}
		Under the assumptions of Lemma \ref{lem_high_order_terms}, there exists some universal constant $c_2>0$ such that
		$$
			\left|\bm{a}^\top \bm{H}^s \bm{u}^\star\right| \leq \left(c_2 \max\left\{B\log n, \sqrt{n \sigma^2 \log n}\right\}\right)^s \sqrt\frac{\mu}{n}, \qquad \forall s \leq 20 \log n
		$$
		with probability $1-O(n^{-10})$.
	\end{cor}
	\begin{proof}
		 See Appendix \ref{appendix:proof-cor_high_order_terms}.
	\end{proof}
	In addition,  in view of Lemma \ref{lem:H-norm} and the condition \eqref{Bsigma_cond_2}, one has
	\begin{equation}
		\|\bm{H}\| \lesssim \max\left\{B \log n, \sqrt{n \sigma^2 \log n }\right\} < 1/10
	\end{equation}
	with probability $1-O(n^{-10})$, which together with Lemma \ref{lem:perturbation-BF-rank-r} implies
	$\lambda \geq 3\|\bm{H}\|$.
	This further leads to
	\begin{align*}
		\sum_{s: s\geq 20\log n}\left(\frac{\|\bm{H}\|}{\lambda}\right)^{s}  &\leq  \frac{\|\bm{H}\|}{\lambda} \sum_{s: s\geq 20\log n-1}\left(\frac{\|\bm{H}\|}{\lambda}\right)^{s}
		\leq \frac{\|\bm{H}\|}{\lambda} \sum_{s: s\geq 20\log n-1}\frac{1}{3^{s}} \\
		& \lesssim \max\left\{B \log n, \sqrt{n \sigma^2 \log n }\right\} \cdot n^{-10}.
	\end{align*}

	Putting the above bounds together and using the fact that $\lambda$ is real-valued and $\lambda \geq 1/2$ (cf.~Lemma \ref{lem:perturbation-BF-rank-r}), we have
 	\begin{align*}
 	\left|\bm{a}^\top \left(\bm{u}-\frac{\bm{u}^{\star\top} \bm{u}}{\lambda}\bm{u}^\star\right)\right|
 	 = & \left|\frac{\bm{u}^{\star\top} \bm{u}}{\lambda}\sum_{s=1}^{+\infty} \frac{\bm{a}^\top\bm{H}^s\bm{u}^\star}{\lambda^s}\right| \\
 	 \lesssim &  \sum_{s=1}^{20 \log n} \frac{1}{\lambda^s} \left|\bm{a}^\top \bm{H}^s \bm{u}^\star\right| +  \sum_{s= 20 \log n }^{+\infty} \left(\frac{\left\|\bm{H}\right\|}{\lambda} \right)^s \\
		\leq &  \sqrt{\frac{\mu}{n}} \sum_{s=1}^{ 20 \log n} \left(2c_2 \max\left\{B\log n, \sqrt{n \sigma^2 \log n}\right\}\right)^s  + \frac{\max\left\{B \log n, \sqrt{n \sigma^2 \log n }\right\}}{n^{10}}  \\
		\lesssim &  \max\left\{B \log n, \sqrt{n \sigma^2 \log n }\right\} \sqrt{\frac{\mu}{n}},
 	\end{align*}
	as long as  $\max\big\{B \log n, \sqrt{n \sigma^2 \log n } \big\}$ is sufficiently small. Here, the last line also uses the fact that $\mu\geq 1$ (and hence $\sqrt{\mu/n}\gg n^{-10}$). This concludes the proof.

  	\section{Extension: perturbation analysis for the rank-$r$ case}
\label{sec:main-rank-r}

\subsection{Eigenvalue perturbation for the rank-$r$ case}
	The eigenvalue perturbation analysis in Section \ref{sec:rank-1} can be extended to accommodate the case where $\bm{M}^{\star}$ is symmetric and rank-$r$, as detailed in this section.  As before, assume that the $r$ non-zero eigenvalues of $\bm{M}^{\star}$ obey $\lambda_{\max}^{\star} = \big|\lambda_1^{\star}\big| \geq \cdots \geq \big|\lambda_r^{\star}\big|= \lambda_{\min}^{\star}$. Once again, we start with a master bound.
	%
	%
	\begin{thm}[\textbf{Perturbation of linear forms of eigenvectors (rank-$r$)}]
		\label{thm_Linear_Forms_1_rank_r_A}
		Consider a rank-$r$ symmetric matrix $\bm{M}^\star \in \mathbb{R}^{n\times n}$ with
		incoherence parameter $\mu$. Define $\kappa \triangleq \lambda_{\mathrm{max}}^{\star}/\lambda_{\mathrm{min}}^{\star}$. Suppose that
		\begin{equation}
		\label{Bsigma_cond_3}
			\frac{ \max\left\{\sigma \sqrt{n \log n}, B \log n \right\} }{\lambda_{\mathrm{max}}^{\star}} \leq  \frac{c_1}  { \kappa }
		\end{equation}
		for some sufficiently small constants $c_1>0$.
		Then for any fixed unit vector $\bm{a} \in \mathbb{R}^{n}$ and any $1\leq l\leq r$,  with probability at least $1-O(n^{-10})$ one has
		\begin{align}
			\Bigg| \bm{a}^\top \Bigg(\bm{u}_l-\sum_{j=1}^r\frac{\lambda_j^\star\bm{u}_j^{\star\top} \bm{u}_l}{\lambda_l} \bm{u}_j^\star\Bigg) \Bigg| 		
			& \lesssim  \max\left\{\sigma \sqrt{n  \log n }, B \log n \right\}\frac{\kappa}{|\lambda_l|}\sqrt{\frac{\mu r}{n}} \label{thm-rankr-UB1}\\
			& \lesssim  \frac{ \max\left\{\sigma \sqrt{n  \log n }, B \log n \right\} }{ \lambda^{\star}_{\max} } \kappa^2\sqrt{\frac{\mu r}{n}}\label{thm-rankr-UB2}.
		\end{align}
	\end{thm}
	This result allows us to control the perturbation of the linear form of eigenvectors.  The perturbation upper bound grows as either the rank $r$ or the condition number $\kappa$ increases.

One of the most important consequences  of Theorem \ref{thm_Linear_Forms_1_rank_r_A} is a refinement of the Bauer-Fike theorem concerning eigenvalue perturbations as follows.
\begin{cor}
	\label{cor:eigenvalue-pertubation-rank-r-A}
	Consider the $l$th ($1\leq l\leq r$) eigenvalue $\lambda_l$ of $\bm{M}$. Under the assumptions of Theorem \ref{thm_Linear_Forms_1_rank_r_A},  with probability at least $1-O(n^{-10})$, there exists $1\leq j\leq r$ such that
	\begin{equation} \label{eq:eigenvalue-perturbation-rank-r-A}
		\big|\lambda_l-\lambda_j^\star \big| \lesssim    \max\left\{\sigma \sqrt{n  \log n }, B \log n \right\} \kappa r\sqrt{\frac{\mu}{n}} ,
	\end{equation}
	provided that
	\begin{equation}
		\label{Bsigma_cond_4}
		\frac{  \max\left\{\sigma \sqrt{n  \log n }, B \log n \right\} } { \lambda^{\star}_{\max} } \leq c_1 / \kappa^2
	\end{equation}
	for some sufficiently small constant $c_1>0$.
\end{cor}
\begin{proof}
See Appendix \ref{sec:corollary-proof-eigenvalue-rank-r}.
\end{proof}

In comparison, the Bauer-Fike theorem (Lemma \ref{lem:perturbation-BF-rank-r}) together with Lemma \ref{lem:H-norm}   gives a perturbation bound
\begin{equation}
	\label{eq:BF-bound-rank-r}
	\big|\lambda_l-\lambda_j^\star \big| \leq \|\bm{H}\| \lesssim \max\left\{\sigma \sqrt{n  \log n }, B \log n \right\} \qquad \text{for some }1\leq j\leq r.
\end{equation}
For the low-rank case where $r\ll \sqrt{n}$, the eigenvalue perturbation bound derived in Corollary \ref{cor:eigenvalue-pertubation-rank-r-A} can be much sharper than the Bauer-Fike theorem.

Another result that comes from Theorem \ref{thm_Linear_Forms_1_rank_r_A} is the following bound that concerns linear forms of the eigen-subspace.

\begin{cor}
	\label{cor:linear-form-rankr}
	Under the same setting of Theorem \ref{thm_Linear_Forms_1_rank_r_A}, with probability $1-O(n^{-9})$ we have
	\begin{equation} \label{eq:linear-form-rankr}
		\left\|\bm{a}^\top \bm{U}\right\|_2 \lesssim \kappa \sqrt{r} \left\|\bm{a}^\top \bm{U}^\star\right\|_2+ \frac{ \max\left\{\sigma \sqrt{n  \log n }, B \log n \right\} }{  \lambda^{\star}_{\max}  } \kappa^2 r\sqrt\frac{\mu}{n}.
	\end{equation}
\end{cor}

\begin{proof}
See Appendix \ref{sec:corollary-proof-linear-form-rankr}.
\end{proof}

Consequently, by taking $\bm{a} = \bm{e}_i$ ($1\leq i\leq n$) in Corollary \ref{cor:linear-form-rankr}, we arrive at the following statement regarding the alternative definition of the incoherence of the eigenvector matrix $\bm{U}$ (see Remark \ref{remark:alternative-incoherence}).

\begin{cor}
	\label{cor:entrywise-rankr}
	Under the same setting of Theorem \ref{thm_Linear_Forms_1_rank_r_A}, with probability $1-O(n^{-8})$ we have
	\begin{equation} \label{eq:entrywise-rankr}
	\left\| \bm{U}\right\|_{2,\infty} \lesssim  \kappa r\sqrt\frac{\mu}{n}.
	\end{equation}
\end{cor}

\begin{proof}
	Given that $\left\|\bm{U}\right\|_{2, \infty} = \max_{1 \leq i \leq n} \left\|\bm{e}_i^\top \bm{U}\right\|_2$ and recalling our assumption implies $\left\|\bm{U}^\star\right\|_{2, \infty}  \leq \sqrt{\mu r/n}$, we can invoke Corollary \ref{cor:linear-form-rankr} and the union bound to derive the advertised entrywise bounds.
\end{proof}
\begin{remark}
The eigenvector matrix is often employed to form a reasonably good initial guess for several nonconvex statistical estimation problems \cite{KesMonSew2010}, and
the above kind of incoherence property  is crucial in guaranteeing fast convergence of the subsequent nonconvex iterative refinement procedures \cite{ma2017implicit}.
\end{remark}

Unfortunately, these results  fall short of providing simple perturbation bounds for the eigenvectors; in other words, the above-mentioned bounds do not imply the size of the difference between $\bm{U}$ and $\bm{U}^{\star}$.
The challenge arises in part due to the lack of orthonormality of the eigenvectors when dealing with asymmetric matrices.  Analyzing the eigenspace perturbation for the general rank-$r$ case will likely require new analysis techniques, which we leave for future work. There is, however, some special case in which we can develop eigenvector perturbation theory, as detailed in the next subsection.

	\begin{remark}
		The theory for the rank-$r$ case has recently been significantly improved; see our follow-up work \cite{cheng2020inference} for details. 
	\end{remark}

\subsection{Application: spectral estimation when \texorpdfstring{$\bm{M}^{\star}$}{the truth}   is asymmetric and rank-1}
\label{sec:asymmetric-rank1}

In some scenarios,
the above general rank results allow us to improve spectral estimation when $\bm{M}^{\star}$ is asymmetric.  Consider the case where $\bm{M}^{\star}= \lambda^{\star} \bm{u}^\star\bm{v}^{\star\top}\in \mathbb{R}^{n_1\times n_2}$ is an asymmetric rank-1 matrix with leading singular value $\lambda^{\star}$. Suppose that we observe two independent noisy copies of $\bm{M}^{\star}$, namely,
\begin{equation}
	\bm{M}_1 = \bm{M}^{\star} + \bm{H}_1, \qquad \bm{M}_2 = \bm{M}^{\star} + \bm{H}_2,
\end{equation}
where $\bm{H}_1$ and $\bm{H}_2$ are independent noise matrices.   The goal is to estimate the singular value and singular vectors of $\bm{M}^{\star}$ from $\bm{M}_1$ and $\bm{M}_2$.

We attempt estimation via the standard dilation trick (e.g.~\cite{Tao2012RMT}). This consists of embedding the matrices of interest within a larger  block matrix

%
\begin{equation}
	\label{eq:defn-dilation}
	\bm{M}_{\mathsf{d}}^{\star} \triangleq \left[\begin{array}{cc}
\bm{0} & \bm{M}^{\star}\\
\bm{M}^{\star\top} & \bm{0}
\end{array}\right],
\qquad
\bm{M}_{\mathsf{d}} \triangleq \left[\begin{array}{cc}
\bm{0} & \bm{M}_1\\
\bm{M}_2^{\top} & \bm{0}
\end{array}\right].
\end{equation}
%
Here, we place $\bm{M}_1$ and $\bm{M}_2$ in two different subblocks, in order to ``asymmetrize'' the dilation matrix.
The rationale is that $\bm{M}_{\mathsf{d}}^{\star}$ is a rank-2 symmetric matrix with exactly two nonzero eigenvalues
\[
	\lambda_1(\bm{M}_{\mathsf{d}}^{\star}) = \lambda^{\star}  \qquad  and \qquad \lambda_2(\bm{M}_{\mathsf{d}}^{\star}) = - \lambda^{\star} ,
\]
whose corresponding eigenvectors are given by
\[
\frac{1}{\sqrt{2}} \begin{pmatrix}
\bm{u}^{\star} \\
\bm{v}^{\star}
\end{pmatrix}
\qquad  and \qquad
\frac{1}{\sqrt{2}} \begin{pmatrix}
\bm{u}^{\star} \\
-\bm{v}^{\star}
\end{pmatrix},
\]
respectively.
This motivates us to perform eigen-decomposition of $\bm{M}_{\mathsf{d}}$, and use the top-2 eigenvalues and eigenvectors to estimate $\lambda^{\star}$, $\bm{u}^{\star}$ and $\bm{v}^{\star}$, respectively.

\bigskip
\noindent {\bf Eigenvalue perturbation analysis.} As an immediate consequence of  Corollary \ref{cor:eigenvalue-pertubation-rank-r-A},   the two leading eigenvalues of $\bm{M}_{\mathsf{d}}$ provide fairly accurate estimates of  the leading singular value $\lambda^{\star}$  of $\bm{M}^{\star}$, as stated below.

\begin{cor}
	\label{cor:eigenvalue-pertubation-rank-2}
	 Assume  $\bm{M}^{\star}\in \mathbb{R}^{n_1\times n_2}$ is a rank-1 matrix with leading singular value $\lambda^{\star}$ and  incoherence parameter $\mu$. 		Define $n \triangleq n_1 + n_2$.
	 Suppose that $\lambda_1^{\mathsf{d}} \geq \lambda_2^{\mathsf{d}}$ are the two leading eigenvalues of  $\bm{M}_{\mathsf{d}}$ (cf.~\eqref{eq:defn-dilation}), and that $\bm{H}_1$ and $\bm{H}_2$ are independent and satisfy Assumption \ref{assumption:H}.
	Then with probability at least $1-O(n^{-10})$,
	\begin{equation} \label{eq:eigenvalue-perturbation-rank-2}
		\max\big\{ \big| \lambda_1^{\mathsf{d}} - \lambda^{\star} \big|, \,\big| \lambda_2^{\mathsf{d}} + \lambda^{\star} \big| \big\} \lesssim   \max\left\{\sigma \sqrt{n \log n}, B \log n \right\}\sqrt{\frac{\mu}{n}} ,
	\end{equation}
	provided that
	\begin{equation}
		\frac{ \max\left\{\sigma \sqrt{n \log n}, B \log n \right\} }{\lambda^{\star}} \leq c_1
	\end{equation}
	for some sufficiently small constant $c_1>0$.
\end{cor}
\begin{proof}
	To begin with, it follows from Corollary \ref{cor:eigenvalue-pertubation-rank-r-A} that both $\lambda_1^{\mathsf{d}}$ and $\lambda_2^{\mathsf{d}}$  are close to either $\lambda^{\star}$ or $-\lambda^{\star}$.
	Repeating similar arguments as in the proof of Lemma \ref{lem:perturbation-BF-rank-r} (which we omit here), we can immediately show the separation between these two eigenvalues, namely, $\lambda_1^{\mathsf{d}}$ (resp.~$\lambda_2^{\mathsf{d}}$) is close to $\lambda^{\star}$ (resp.~$-\lambda^{\star}$).
\end{proof}

\bigskip
\noindent {\bf Eigenvector perturbation analysis.} We then move on to studying the eigenvector perturbation bounds. Specifically, denote by $\bm{u}_1^{\mathsf{d}}$ and $\bm{u}_2^{\mathsf{d}}$ the eigenvectors of $\bm{M}_{\mathsf{d}}$ associated with its two leading eigenvalues $\lambda_1^{\mathsf{d}}$ and $\lambda_2^{\mathsf{d}}$, respectively. Without loss of generality, we assume that $\lambda_1^{\mathsf{d}}\geq \lambda_2^{\mathsf{d}}$. If we write
$$
\bm{u}_1^{\mathsf{dilation}} = \begin{pmatrix}
\bm{u}_{1,1}^{\mathsf{d}} \\
\bm{u}_{1,2}^{\mathsf{d}}
\end{pmatrix} \qquad \text{with } \bm{u}_{1,1}^{\mathsf{d}} \in \mathbb{R}^{n_1}, \bm{u}_{1,2}^{\mathsf{d}} \in \mathbb{R}^{n_2},
$$
then we can employ $\bm{u}_{1,1}^{\mathsf{d}}$ and $\bm{u}_{1,2}^{\mathsf{d}}$ to estimate $\bm{u}^{\star}$ and $\bm{v}^{\star}$ after proper normalization, namely,
\begin{equation}
	\label{eq:rank-1-inferred-singular-vectors}
	 \bm{u}  \triangleq \frac{\bm{u}_{1,1}^{\mathsf{d}}}{\left\| \bm{u}_{1,1}^{\mathsf{d}} \right\|_2}, \qquad \bm{v} \triangleq  \frac{\bm{u}_{1,2}^{\mathsf{d}}}{ \| \bm{u}_{1,2}^{\mathsf{d}} \|_2 }.
\end{equation}
The following theorem develops error bounds for both $\bm{u}$ and $\bm{v}$, which we establish in Appendix \ref{appendix:proof-rank_2_dilation_matrix}.  Here, we denote $\min \|\bm{x}\pm\bm{y}\|_2  = \min\{ \|\bm{x}-\bm{y}\|_2, \|\bm{x}+\bm{y}\|_2\}$, and $\min \|\bm{x}\pm\bm{y}\|_{\infty}  = \min\{ \|\bm{x}-\bm{y}\|_{\infty}, \|\bm{x}+\bm{y}\|_{\infty}\}$.
\begin{thm}
	\label{cor:l2_convergence-rank-2-dilation}
	Suppose $\bm{M}^{\star} = \lambda^{\star} \bm{u}^\star \bm{v}^{\star\top} \in \mathbb{R}^{n_1 \times n_2}$ is a rank-$1$  matrix with leading singular value $\lambda^{\star}$ and incoherence parameter $\mu$, where $\|\bm{u}^{\star}\|_2=\|\bm{v}^{\star}\|_2=1$. 		Define $n \triangleq n_1 + n_2$, and fix any unit vectors $\bm{a} \in \mathbb{R}^{n_1}$ and $\bm{b} \in \mathbb{R}^{n_2}$. Then with probability at least $1-O(n^{-10})$,  the estimates ${\bm{u}}$ and ${\bm{v}}$ (cf.~\eqref{eq:rank-1-inferred-singular-vectors}) obey
	%
	\begin{subequations}
	\begin{align}
		\max\big\{ \min\|\bm{u}\pm \bm{u}^{\star}\|_{2},  ~\min\|\bm{v} \pm \bm{v}^{\star}\|_{2} \big\}  & \lesssim  \frac{ \max\left\{\sigma \sqrt{n  \log n}, B \log n\right\} }{ \lambda^{\star} } , \\
		\max\big\{ \min \left\|\bm{u} \pm \bm{u}^{\star}\right\|_\infty , ~\min \left\|\bm{v} \pm \bm{v}^{\star}\right\|_\infty  \big\} & \lesssim  \frac{ \max\left\{\sigma \sqrt{n \log n }, B \log n \right\} }{ \lambda^{\star} } \sqrt{\frac{\mu}{n}}, \\
		\min\big\{ \big|\bm{a}^{\top} (\bm{u} - \bm{u}^{\star})\big|, \big|\bm{a}^{\top} (\bm{u} + \bm{u}^{\star})\big| \big\}   & \lesssim  \left( |\bm{a}^{\top} \bm{u}^{\star}| + \sqrt{\frac{\mu}{n}}\right)  \frac{ \max\left\{\sigma \sqrt{n \log n }, B \log n \right\} }{ \lambda^{\star} } , \\
		\min\big\{ \big|\bm{b}^{\top} (\bm{v} - \bm{v}^{\star})\big|, \big|\bm{b}^{\top} (\bm{v} + \bm{v}^{\star})\big| \big\}   & \lesssim  \left( |\bm{b}^{\top} \bm{v}^{\star}| + \sqrt{\frac{\mu}{n}}\right)  \frac{ \max\left\{\sigma \sqrt{n \log n }, B \log n \right\} }{ \lambda^{\star} } ,
	\end{align}
	\end{subequations}
	provided that there exists some some sufficiently small constant $c_1>0$ such that
	\begin{equation}
		 \frac{ \max\left\{\sigma \sqrt{n \log n }, B \log n \right\} }{ \lambda^{\star} } \leq c_1 .
	\end{equation}
\end{thm}
%

Similar to the symmetric rank-1 case, the estimation errors of the estimates ${\bm{u}}$ and ${\bm{v}}$ are well-controlled in any deterministic direction (e.g.~the entrywise errors are well-controlled).  This allows us to complete the theory for the case when $\bm{M}^{\star}$ is a real-valued and rank-1 matrix.

\begin{figure}[htbp!]
	\centering
	\begin{tabular}{ccc}
		\includegraphics[width=0.32\linewidth]{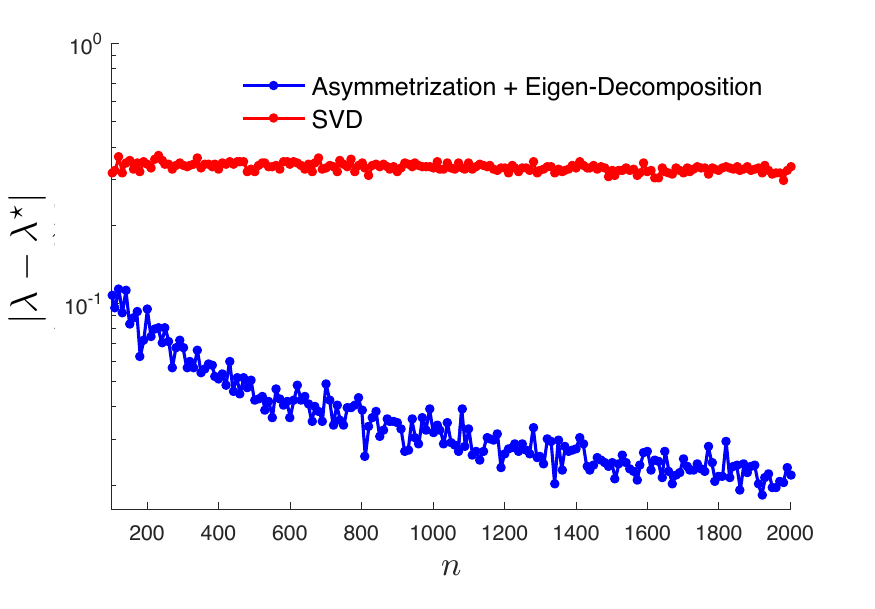}
		& \includegraphics[width=0.32\linewidth]{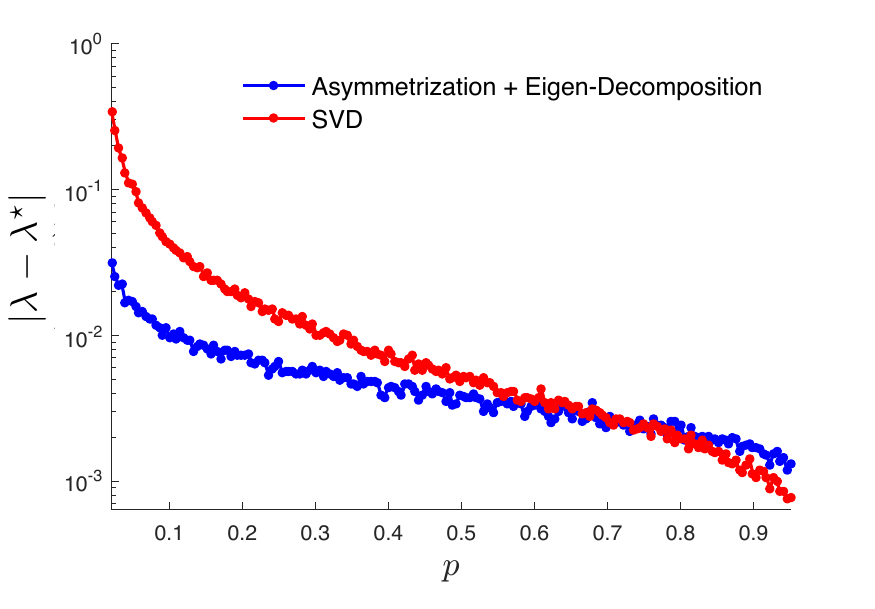}
		& \includegraphics[width=0.32\linewidth]{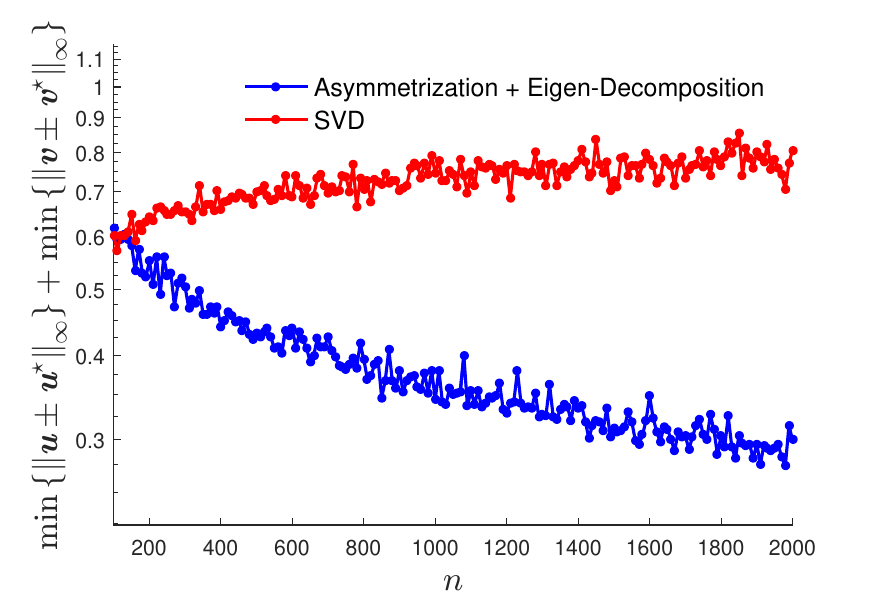}
	\tabularnewline
	(a) eigenvalue perturbation vs.~$n$ & (b) eigenvalue perturbation vs.~$p$ & (c) $\ell_\infty$ eigenvector perturbation \tabularnewline
	\end{tabular}

	\caption{Numerical experiments for rank-1 matrix completion, where $\bm{M}^{\star}=\bm{u}^\star \bm{v}^{\star\top}\in \mathbb{R}^{n_1\times n_2}$ is randomly generated with leading singular value $ \lambda^{\star}=1$. Let $n=n_1=2n_2$. Each entry is observed independently with probability $p$.  (a)  $|\lambda - \lambda^{\star}|$  vs.~$n$ with $p=3\log n\,/\,n$; (b)  $|\lambda -  \lambda^{\star}|$ vs.~$p$ with $n=1000$; (c) $\ell_{\infty}$ eigenvector estimation error vs.~$n$ with $p=3\log n\,/\,n$.
	The blue (resp. red) lines represent the average errors over 100 independent trials using using the eigen-decomposition (resp.~SVD) approach applied to $\bm{M}_{\mathsf{dilation}}$ (resp.~$\bm{M}$).
	\label{fig:mc_rankr_varyn}}
\end{figure}

Further, we conduct numerical experiments for matrix completion when $\bm{M}^{\star}$ is a rank-1 and asymmetric matrix in Fig.~\ref{fig:mc_rankr_varyn}. Here, we suppose that at most 1 sample is observed for  each entry, and
we estimate the singular value and singular vectors of $\bm{M}^{\star}$ via the above-mentioned dilation trick,  coupled with the asymmetrization procedure discussed in Section \ref{sec:asymmetrization-rank1}.
The numerical performance confirms that the proposed technique outperforms vanilla SVD in spectral estimation.

Finally, we remark that  the  asymptotic behavior of the eigenvalues of asymmetric random matrices has been extensively explored in the physics literature  (e.g.~\citep{sommers1988spectrum, khoruzhenko1996large, brezin1998non, chalker1998eigenvector,feinberg1997non,lytova2018delocalization}). Their focus, however, has largely been to pin down the asymptotic density of the eigenvalues, similar to the semi-circle law in the symmetric case. Nevertheless, a sharp perturbation bound for the leading eigenvalue --- particularly for the low-rank case --- is beyond their reach. A few recent papers began to explore the locations of eigenvalue outliers that fall outside the bulk predicted by the circular law \cite{tao2013outliers,rajagopalan2015outlier,benaych2016outliers,bordenave2016outlier}.  The results reported therein either do not focus on obtaining the right convergence rate (e.g.~providing only a bound like $|\lambda - \lambda^\star| = o(|\lambda^{\star}|)$) or are restricted to a special family of ground truth (e.g.~the one with a diagonal block equal to identity) or i.i.d.~noise. As a result, these prior results are insufficient to demonstrate the power and benefits of the eigen-decomposition method in the presence of data asymmetry.

\subsection{Proof of Theorem \ref{thm_Linear_Forms_1_rank_r_A}}


Without loss of generality, we shall assume $\lambda_{\max}^{\star} = \lambda_1^{\star} = 1$ throughout the proof.
To begin with, Lemma \ref{lem:perturbation-BF-rank-r} implies that for all $1\leq l\leq r$,
	\begin{equation}
		\label{eq:lambda-l-LB}
	|\lambda_l| \geq |\lambda_\mathrm{min}^\star| - \left\|\bm{H}\right\| > 1/(2\kappa) > \|\bm{H}\|
	\end{equation}
as long as $\left\|\bm{H}\right\| < 1/(2\kappa)$.
In view of the Neumann trick (Theorem \ref{thm_Neumann}), we can derive
\begin{align}
\Bigg|\bm{a}^{\top}\bm{u}_{l}-\sum_{j=1}^{r}\frac{\lambda_{j}^{\star}\bm{u}_{j}^{\star\top}\bm{u}_{l}}{\lambda_{l}}\bm{a}^{\top}\bm{u}_{j}^{\star}\Bigg| & =\Bigg|\sum_{j=1}^{r}\frac{\lambda_{j}^{\star}}{\lambda_{l}}\big(\bm{u}_{j}^{\star\top}\bm{u}_{l}\big)\Bigg\{\sum_{s=1}^{\infty}\frac{1}{\lambda_{l}^{s}}\bm{a}^{\top}\bm{H}^{s}\bm{u}_{j}^{\star}\Bigg\}\Bigg| \nonumber\\
 & \leq\Bigg(\sum_{j=1}^{r}\frac{\big|\lambda_{j}^{\star}\big|}{|\lambda_{l}|}\big|\bm{u}_{j}^{\star\top}\bm{u}_{l}\big|\Bigg)\Bigg\{\max_{1\leq j\leq r}\sum_{s=1}^{\infty}\frac{1}{|\lambda_{l}|^{s}}\big|\bm{a}^{\top}\bm{H}^{s}\bm{u}_{j}^{\star}\big|\Bigg\} \nonumber\\
 & \leq\sqrt{r\sum_{j=1}^{r}\big|\bm{u}_{j}^{\star\top}\bm{u}_{l}\big|^{2}}\Bigg\{\max_{1\leq j \leq r}\frac{\big|\lambda_{j}^{\star}\big|}{|\lambda_{l}|}\Bigg\}\Bigg\{\max_{1\leq j\leq r}\sum_{s=1}^{\infty}\frac{1}{|\lambda_{l}|^{s}}\big|\bm{a}^{\top}\bm{H}^{s}\bm{u}_{j}^{\star}\big|\Bigg\} \nonumber\\
 & \leq \sqrt{r} \cdot \frac{1}{|\lambda_l|} \cdot \Bigg\{\max_{1\leq j\leq r}\sum_{s=1}^{\infty} \frac{1}{|\lambda_{l}|^{s}} \big|\bm{a}^{\top}\bm{H}^{s}\bm{u}_{j}^{\star}\big|\Bigg\},  \label{eq:rank-r-UB1}
\end{align}
where the third line follows since $\sum_{j=1}^{r}\big|\bm{u}_{j}^{\star\top}\bm{u}_{l}\big|^{2} \leq \|\bm{u}_l\|_2^2 = 1$, and the last inequality makes use of \eqref{eq:lambda-l-LB}. Apply Corollary \ref{cor_high_order_terms} to reach
\begin{align}
	\eqref{eq:rank-r-UB1}
	&\leq \frac{\sqrt{r}}{|\lambda_l|}\sum_{s=1}^{\infty}\left(2c_{2}\kappa\max\left\{ B\log n,\sqrt{n\sigma^{2}\log n}\right\} \right)^{s}\sqrt{\frac{\mu}{n}}
\nonumber\\
	&\lesssim \frac{\kappa}{|\lambda_l|}
\max\left\{ B\log n,\sqrt{n\sigma^{2}\log n}\right\}
\sqrt{\frac{\mu r}{n}}  \nonumber \\
	&\lesssim \kappa^2
	\max\left\{ B\log n,\sqrt{n\sigma^{2}\log n}\right\}
	\sqrt{\frac{\mu r}{n}},  \nonumber
\end{align}
with the proviso that $|\lambda_l| > 1/(2\kappa)$ and $\max\left\{ B\log n,\sqrt{n\sigma^{2}\log n}\right\} \leq c_{1}/\kappa$ for some sufficiently small constant $c_1>0$. The condition $|\lambda_l| > 1/(2\kappa)$ follows immediately by combining Lemma \ref{lem:perturbation-BF-rank-r}, Lemma \ref{lem:H-norm} and the condition \eqref{Bsigma_cond_3}.

	\section{Discussions}

In this paper, we demonstrate the remarkable advantage of eigen-decomposition over SVD  in the presence of asymmetric noise matrices. This is in stark contrast to conventional wisdom, which is generally not in favor of eigen-decomposition for asymmetric matrices. Our results only reflect the tip of an iceberg, and there are many outstanding issues left answered. We conclude the paper with a few future directions.

\medskip
\noindent {\bf Sharper eigenvalue perturbation bounds for the rank-$r$ case.}
Our current results in Section \ref{sec:main-rank-r} provide an eigenvalue perturbation bound on the order of $r/\sqrt{n}$, assuming  the truth is rank-$r$. However, numerical experiments suggest that the dependency on $r$ might be improvable.  It would be interesting to see whether further theoretical refinement is possible, e.g.~whether it is possible to improve it to $O(\sqrt{r/n})$.

\medskip
\noindent {\bf Eigenvector perturbation bounds for the rank-$r$ case.}  As mentioned before, the current theory falls short of providing eigenvector perturbation bounds for the general rank-$r$ case.  The main difficulty lies in the lack of orthogonality of the eigenvectors of the observed matrix $\bm{M}$.  Nevertheless, when the size of the noise is not too large, it is possible to establish certain near-orthogonality of the eigenvectors, which might in turn lead to sharp control of eigenvector perturbation.

\medskip
\noindent {\bf A challenging signal-to-noise ratio regime}. Take the rank-1 case for example: the present work focuses on the regime where $\|\bm{H}\|\lesssim \| \bm{M}^{\star} \| / \sqrt{\log n} $, and it is known that spectral methods fail to yield reliable estimation if  $\|\bm{H}\|\gg \| \bm{M}^{\star} \| $. There is, however, a ``gray'' region (which includes, for example, the case with $\|\bm{H}\| \approx \| \bm{M}^{\star} \| $) that has not been addressed.   Developing non-asymptotic yet informative perturbation bounds for this regime is likely very challenging and requires new analysis techniques,  which we leave for future investigation.


\medskip
\noindent {\bf Correlated noise.} The current theoretical development relies heavily on the assumption that the noise matrix $\bm{H}$ contains independent random entries. There is no shortage of examples where the noise matrix is asymmetric but is not composed of independent entries.  For instance, in blind deconvolution \cite{li2018rapid}, the noise matrix is a sum of independent asymmetric matrices.  Can we develop eigenvalue perturbation theory for this class of noise?

\medskip
\noindent {\bf Statistical inference of eigenvalues and eigenvectors.} In various applications like network analysis and inference,  one might be interested in determining the (asymptotic) eigenvalue and eigenvector distributions  of  a random data matrix,  in order to produce valid confidence intervals \cite{johnstone2001distribution,bai2008central,cai2017limiting,cape2018signal,xia2018confidence,bao2018singular,chen2019inference}. Can we use the current framework to characterize the distributions of the leading eigenvalues as well as certain linear forms of the eigenvectors of $\bm{M}$ when the noise matrix is non-symmetric?

	\begin{figure}
		\centering
		\includegraphics[width=0.4\linewidth]{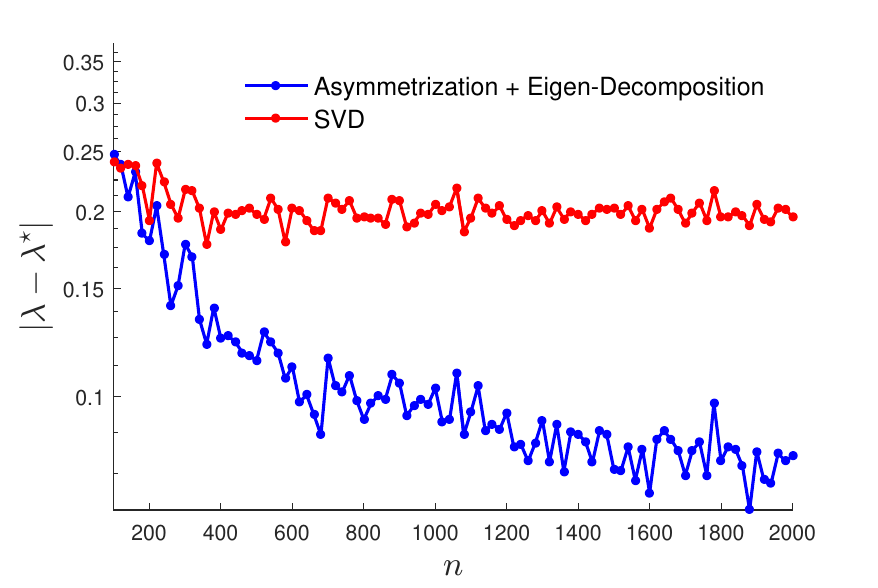}
		\caption{Numerical experiments for the spiked covariance model, where the sample vectors are zero-mean  Gaussian vectors with covariance matrix $\bm{\Sigma}^{\star}$ (cf.~\eqref{eq:true-covariance}).  We plot $|\lambda-\lambda^{\star}|$ vs.~$n$ with $d=n\,/\,10$.
		The blue (resp.~red) lines represent the average  errors over 100 independent trials when $\lambda$ is the leading eigenvalue of  $\widehat{\bm{\Sigma}}_{\mathsf{asym}}$ (resp.~$\hat{\bm{\Sigma}}$).
		\label{fig:covariance_asymmetrization}}
	\end{figure}

\medskip
\noindent {\bf Asymmetrization for other applications.} Given the abundant applications of spectral estimation,  our findings are likely to be useful for other matrix eigenvalue problems and might extend to the tensor case \cite{zhang2018tensor,cai2019subspace}.  Here, we conclude the paper with an example in {\em covariance estimation}
\citep{baik2006eigenvalues,fan2018eigenvector}. Imagine that we observe a collection of $n$ independent Gaussian vectors $\bm{X}_1, \cdots, \bm{X}_n\in \mathbb{R}^d$, which have mean zero and  covariance matrix
\begin{equation}
	\label{eq:true-covariance}
	\bm{\Sigma}^{\star} = \bm{v} \bm{v}^{\top} + \bm{I}_d
\end{equation}
with $\bm{v}$ being a unit vector. This falls under the category of the spiked covariance model \citep{johnstone2009consistency}.  One strategy to estimate the spectral norm $\lambda^{\star}=2$ of $\bm{\Sigma}^{\star}$ is to look at the spectrum of the sample covariance matrix $\hat{\bm{\Sigma}}=\frac{1}{n} \sum_{i=1}^n \bm{X}_i\bm{X}_i^{\top}$. Motivated by the results of this paper, we propose an alternative strategy by looking at the following  asymmetrized sample covariance matrix
\begin{equation}
	\label{eq:asym-covariance-matrix}
	\widehat{\bm{\Sigma}}_{\mathsf{asym}} = \frac{2}{n} \Big(  \sum\nolimits_{i=1}^{n/2} \mathsf{Upper}\left(\bm{X}_i \bm{X}_i^\top\right) +  \sum\nolimits_{i=n/2+1}^n \mathsf{Lower}\left(\bm{X}_i \bm{X}_i^\top\right)  \Big),
\end{equation}
where $\mathsf{Upper}(\cdot)$ (resp.~$\mathsf{Lower}(\cdot)$) extracts out the upper (resp.~lower) triangular part of the matrix, including (resp.~excluding) the diagonal entries.  As can be seen from Fig.~\ref{fig:covariance_asymmetrization}, the largest eigenvalue of the asymmetrized $\widehat{\bm{\Sigma}}_{\mathsf{asym}}$ is much closer to the true spectral norm of $\bm{\Sigma}^{\star}$, compared to the largest singular value of the sample covariance matrix $\hat{\bm{\Sigma}}$.   We leave the theoretical understanding of such findings to future investigation.

	\section*{Acknowledgment}

	Y.~Chen is supported in part by the AFOSR YIP award FA9550-19-1-0030, by the ARO  grant W911NF-18-1-0303, by the ONR grant N00014-19-1-2120, by the NSF grants 
	CCF-1907661 and IIS-1900140,  and by the Princeton SEAS innovation award.
	J.~Fan is supported in part by the NSF grants DMS-1662139 and DMS-1712591, by the ONR grant N00014-19-1-2120, and by the NIH grant 2R01-GM072611-13. 
	C.~Cheng is supported in part by the Elite Undergraduate Training Program of School of Mathematical Sciences in Peking University, and by the William R.~Hewlett Stanford graduate fellowship. 
	We thank Cong Ma for helpful discussions, and Zhou Fan for telling us an example of asymmetrizing the Gaussian matrix.

\bibliographystyle{ims}
\bibliography{bibfile_asymmetric}

	\newpage

	\appendix
	
	\section{Proofs for preliminary facts}
\label{sec:proof-preliminary}

\subsection{Proof of Lemma \ref{lem:perturbation-BF-rank-r}}
	\label{appendix:proof-lem:perturbation-BF}

	Without loss of generality,  suppose the leading eigenvalue of $\bm{M}^{\star}$ is 1. 
	
	(1) We start by proving the rank-1 case,  in which the eigenvalues of $\bm{M}^\star$ are either 0 or 1.   Theorem \ref{thm_BauerFike} immediately implies that 
	\begin{equation} 
	\label{eq:lambda-2disk}
		\Lambda(\bm{M}) ~\subseteq~ \mathcal{B}\left(1, \left\|\bm{H}\right\| \right) \,\cup\, \mathcal{B}\left(0, \left\|\bm{H}\right\| \right) ,
	\end{equation}
	where $\Lambda(\bm{M})$ is the set of eigenvalues of $\bm{M}$, and $\mathcal{B}(z,r) \triangleq \{x\in \mathbb{C}: |x-z|\leq r\}$ denotes a disk of radius $r$ centered at $z$. 
	
	To establish the lemma, our aim is to show that  there is exactly one eigenvalue  of $\bm{M}$ --- denoted by $\lambda$ --- lying in the disk $\mathcal{B}\left(1, \left\|\bm{H}\right\| \right)$, and it has multiplicity 1. If this is true, then the assumption $\left\|\bm{H}\right\| < 1/2$ indicates that $|\lambda| \geq 1- \|\bm{H}\| > \| \bm{H} \| \geq z$ for any $z\in \mathcal{B}(0, \|\bm{H}\|)$, and hence $\lambda$ must be the leading eigenvalue.  Furthermore, since $\bm{M}$ is a real-valued matrix, both $\lambda$ and its complex conjugate $\overline{\lambda}$ are eigenvalues of $\bm{M}$. However, if $\lambda \neq \overline{\lambda}$, then both of these two eigenvalues fall within $\mathcal{B}(1,\|\bm{H}\|)$, resulting in contradiction. As a result, one necessarily has $\lambda= \overline{\lambda}$ and both of them are real-valued. A similar argument demonstrates that the eigenvector $\bm{u}$ associated with $\lambda$ is also real-valued. 
	
	We then justify the existence and uniqueness of an eigenvalue in $\mathcal{B}\left(1, \left\|\bm{H}\right\| \right)$. 
	Denote $\Lambda(\bm{M})=\left\{\lambda_1,  \cdots, \lambda_n\right\}$, and define a set of auxiliary matrices
	$$
	\bm{M}_t = \bm{M}^\star + t \bm{H} , \qquad 0\leq t\leq 1. 
	$$
	Similar to  \eqref{eq:lambda-2disk}, one has
	\begin{equation} 
	\label{eq:lambda-t-2disk}
	\Lambda(\bm{M}_t) ~\subseteq~ \mathcal{B}\left(1, t\left\|\bm{H}\right\| \right) \,\cup\, \mathcal{B}\left(0, t\left\|\bm{H}\right\| \right) .
	\end{equation}
	Recognizing that the set of eigenvalues of $\bm{M}_t$ depends  continuously on $t$ (e.g.~\cite[Theorem 6]{embree2001generalizing}), we can write
	\begin{equation}
	\Lambda(\bm{M}_t) = \left\{\lambda_1(t), \lambda_2(t), \cdots, \lambda_n(t) \right\},
	\end{equation}
	with each $\lambda_j(t), 1 \leq j \leq n$ being a continuous function in $t$.  Meanwhile,  as long as $\left\|\bm{H}\right\| < 1/2$ and $0\leq t\leq 1$, the two disks $\mathcal{B}\left(1, t\left\|\bm{H}\right\| \right)$ and $\mathcal{B}\left(0, t\left\|\bm{H}\right\|\right)$ are always disjoint. Thus, the continuity of the spectrum (w.r.t.~$t$) requires $\lambda_j(t)$ to always stay within the same disk where $\lambda_j(0)\in \{0,1\}$ lies, namely, 
	%
	\begin{equation}
		\lambda_j(t) \in \mathcal{B}(\lambda_j(0), t \left\|\bm{H}\right\|).
	\end{equation}
	Given that $\bm{M}^\star$ (or $\bm{M}_0$) has $n-1$ eigenvalues equal to $0$ and one eigenvalue equal to $1$, we establish the lemma for the rank-1 case.

	(2) We now turn to the rank-$r$ case. Repeating the above argument for the rank-1 case, we can immediately show that: if $\|\bm{H}\|<\lambda_r^{\star}/2$, then (i) there are exactly $n-r$ eigenvalues lying within $\mathcal{B}(0,\|\bm{H}\|)$;  (ii) all other eigenvalues lie within $\cup_{1\leq j\leq r} \mathcal{B}(\lambda_j^{\star},\|\bm{H}\|)$, which are exactly the top-$r$ leading eigenvalues of $\bm{M}$.  This concludes the proof.

\subsection{Proof of Theorem \ref{thm_Neumann}}
	\label{appendix:proof-thm_Neumann}

	From the definition of eigenvectors,
	$$
		\left(\bm{M}^\star + \bm{H}\right) \bm{u}_l = \lambda_l \bm{u}_l, \qquad \text{or equivalently,} \qquad  \frac{1}{\lambda_l}\bm{M}^\star  \bm{u}_l = \Big(  \bm{I} - \frac{1}{\lambda_l} \bm{H} \Big) \bm{u}_l.
	$$
	When $\left\|\bm{H}\right\|_2 < |\lambda_l|$, one can  invert $\bm{I} - \frac{1}{\lambda_l} \bm{H}$ and use the assumption \eqref{eq:M-star-eigen}  to reach
	\begin{align*}
	\bm{u}_l & = \Big(  \bm{I} - \frac{1}{\lambda_l} \bm{H} \Big)^{-1}  \frac{1}{\lambda_l}\bm{M}^\star \bm{u}_l \\
	& = \frac{1}{\lambda_l} \left(\bm{I}-\frac{1}{\lambda_l} \bm{H} \right)^{-1}   \left(\sum\nolimits_{j=1}^r \lambda_j^\star \bm{u}_j^\star \bm{u}_j^{\star \top} \right) \bm{u}_l \\
	& = \sum_{j=1}^r \frac{\lambda_j^\star}{\lambda_l} \big(\bm{u}_j^{\star \top} \bm{u}_l\big) \left(\bm{I}-\frac{1}{\lambda_l}\bm{H}\right)^{-1} \bm{u}_j^\star ,
	\end{align*}
		where the last line follows by rearranging terms. Finally, replacing $\big(\bm{I}-\frac{1}{\lambda_l}\bm{H}\big)^{-1}$ with the Neumann series $\sum_{s=0}^{\infty} \frac{1}{\lambda_l^s} \bm{H}^s$, we establish the theorem.

\subsection{Proof of Lemma \ref{lem_l2_convergence-rank-r}}
	\label{appendix:proof-lem_l2_convergence}
		
	(1) We start with the rank-1 case.  Towards this, we resort to the Neumann trick in Theorem \ref{thm_Neumann}, which in the rank-1 case reads
	\begin{equation}
		\label{eq:u-rank1-expand}
		\bm{u} = \frac{1}{\lambda}\left(\bm{u}^{\star \top} \bm{u}\right)\sum_{s=0}^{\infty} \left(\frac{1}{\lambda}\bm{H}\right)^s \bm{u}^\star .
	\end{equation}
	From Lemma \ref{lem:perturbation-BF-rank-r}, we know that $\lambda$ is real-valued and that $\lambda > 1 - \|\bm{H}\| \geq 3/4 > \|\bm{H} \|$ under our assumption. This together with \eqref{eq:u-rank1-expand} yields
	\begin{align}
	\left\|\bm{u}-\frac{\bm{u}^{\star\top}\bm{u}}{\lambda}\bm{u}^\star\right\|_2 \leq \frac{1}{\lambda} \sum_{s=1}^{\infty} \left\| \frac{1}{\lambda} \bm{H}   \right\|^s \|\bm{u}^{\star} \|_2 = 
\frac{1}{\lambda} \sum_{s=1}^{\infty} \left\| \frac{1}{\lambda} \bm{H}   \right\|^s
	 = \frac{\left\|\bm{H}\right\|}{\lambda\left(\lambda-\left\|\bm{H}\right\|\right)} 
	 \leq \frac{8}{3} \left\|\bm{H}\right\| ,  \label{eq:u-UB1}
	\end{align}
	where the last inequality holds since $\lambda \geq 3/4$ and $\lambda -\|\bm{H}\| \geq 1-2\|\bm{H}\|\geq 1/2$. 

	Next, by  decomposing $\bm{u}$ into two orthogonal components 
	$\bm{u} = (\bm{u}^{\star\top}\bm{u})\bm{u}^\star +  (\bm{u}- (\bm{u}^{\star\top}\bm{u} )\bm{u}^\star ) $,  
	we  obtain 
	\begin{align}
	\left|\bm{u}^{\star\top}\bm{u}\right| & = \left\|\left(\bm{u}^{\star\top}\bm{u}\right)\bm{u}^\star\right\|_2 
	 = \sqrt{1 - \left\|\bm{u}-\left(\bm{u}^{\star\top}\bm{u}\right)\bm{u}^\star\right\|_2^2}  \nonumber \\
	& \geq 1 - \left\|\bm{u}-\left(\bm{u}^{\star\top}\bm{u}\right)\bm{u}^\star\right\|_2^2 \nonumber\\
	& \geq 1 - \left\|\bm{u}-\frac{\bm{u}^{\star\top}\bm{u}}{\lambda}\bm{u}^\star\right\|_2^2 \label{eq:u_LB1} \\
	& \geq 1 - \frac{64}{9} \left\|\bm{H}\right\|^2 . \label{eq:u_LB2}
	\end{align}
	The inequality \eqref{eq:u_LB1} holds since $\left(\bm{u}^{\star\top}\bm{u}\right)\bm{u}^\star$ is orthogonal projection of $\bm{u}$ onto the subspace spanned by $\bm{u}^{\star}$, and hence $\|\bm{u} - \left(\bm{u}^{\star\top}\bm{u}\right)\bm{u}^\star \|_2 \leq  \|\bm{u} - \frac{1}{\lambda}\left(\bm{u}^{\star\top}\bm{u}\right)\bm{u}^\star \|_2$. 

	Finally, \eqref{eq:u_LB2} together with the fact that $\bm{u}$ is real-valued (cf.~Lemma \ref{lem:perturbation-BF-rank-r}) leads to the advertised bound:
	\begin{align*}
		\min\{\|\bm{u}-\bm{u}^{\star}\|_{2},\|\bm{u}+\bm{u}^{\star}\|_{2}\}  =\sqrt{\|\bm{u}\|_{2}^{2}+\|\bm{u}^{\star}\|_{2}^{2}-2\big|\bm{u}^{\star\top}\bm{u}\big|}  \leq\frac{8\sqrt{2}}{3}\|\bm{H}\|. 
	\end{align*}

	(2) For the rank-$r$ case, it is seen that for any $1\leq l\leq r$,
	\begin{align}
	\sum_{j=1}^r |\bm{u}_j^{\star \top} \bm{u}_l|^2 & = \left\|\sum\nolimits_{j=1}^r \big(\bm{u}_j^{\star \top} \bm{u}_l\big) \bm{u}_j^\star\right\|_2^2 
	=  1 - \left\|\bm{u}_l-\sum\nolimits_{j=1}^r \big(\bm{u}_j^{\star \top} \bm{u}_l\big) \bm{u}_j^\star\right\|_2^2 \nonumber\\
	& \geq  1 - \Bigg \|\bm{u}_l-\sum_{j=1}^r \frac{\lambda_j^\star\bm{u}_j^{\star \top} \bm{u}_l}{\lambda_l} \bm{u}_j^\star \Bigg\|_2^2 ,
	\label{eq:LB-prelim-1}
	\end{align}
	where the inequality arises since $\sum\nolimits_{j=1}^r \big(\bm{u}_j^{\star \top} \bm{u}_l\big) \bm{u}_j^\star$ is the Euclidean projection of $\bm{u}_l$ onto the span of $\{\bm{u}_1^{\star},\cdots,\bm{u}_r^{\star}\}$.  In addition, observe that 
	\begin{align*}
		\left\Vert \sum\nolimits _{j=1}^{r}\lambda_{j}^{\star}\big(\bm{u}_{j}^{\star\top}\bm{u}_{l}\big)\bm{u}_{j}^{\star}\right\Vert _{2} 
		\leq \sqrt{\sum\nolimits _{j} \big(\lambda_{j}^{\star}\big)^{2}  \big| \bm{u}_{j}^{\star\top}\bm{u}_{l} \big|^{2}}
		\leq\lambda_{\mathrm{max}}^{\star}\sqrt{\sum\nolimits _{j}\big| \bm{u}_{j}^{\star\top}\bm{u}_{l} \big|^{2}} = \lambda_{\mathrm{max}}^{\star}\|\bm{u}_{l}\|_{2} = 1. 		
	\end{align*}
	This taken collectively with Theorem \ref{thm_Neumann} leads to 
	\begin{align*}
		\Bigg\Vert \bm{u}_{l}-\sum_{j=1}^{r}\frac{\lambda_{j}^{\star}\bm{u}_{j}^{\star\top}\bm{u}_{l}}{\lambda_{l}}\bm{u}_{j}^{\star} \Bigg\Vert _{2} 
		& = \Bigg\Vert \sum_{j=1}^{r} \frac{\lambda_{j}^{\star}}{\lambda_{l}}\big(\bm{u}_{j}^{\star\top}\bm{u}_{l}\big)\left\{ \sum_{s=1}^{\infty}\frac{1}{\lambda_{l}^{s}}\bm{H}^{s}\bm{u}_{j}^{\star}\right\} \Bigg\Vert _{2}\\
		& = \Bigg\Vert \frac{1}{\lambda_{l}}\sum_{s=1}^{\infty}\frac{1}{\lambda_{l}^{s}}\bm{H}^{s}\Bigg\{\sum_{j=1}^r\lambda_{j}^{\star} \big(\bm{u}_{j}^{\star\top}\bm{u}_{l}\big) \bm{u}_{j}^{\star} \Bigg\} \Bigg\Vert _{2}\\
 		& \leq\frac{1}{|\lambda_{l}|}|\sum_{s=1}^{\infty}\frac{1}{|\lambda_{l}|^{s}}\|\bm{H}\|^{s} \Bigg\|\sum_{j=1}^r \lambda_j^\star\left(\bm{u}_j^{\star \top} \bm{u}_l\right)\bm{u}_j^\star\Bigg\|_2\\
 		& \leq \frac{1}{|\lambda_{l}|}\cdot\frac{\|\bm{H}\|}{|\lambda_{l}|-\|\bm{H}\|} \\
 		& \leq\frac{8\kappa^2}{3}\|\bm{H}\|.
	\end{align*}
	The last two lines follow since, when  $\left\|\bm{H}\right\| < 1/(4\kappa) $, one can invoke Lemma \ref{lem:perturbation-BF-rank-r} to show  (i) $|\lambda_l| > |\lambda_\mathrm{min}^{\star}| - \|\bm{H}\| \geq 3/(4\kappa)$ and (ii) $|\lambda_l| - \left\|\bm{H}\right\| \geq |\lambda_\mathrm{min}^{\star}| - 2 \left\|\bm{H}\right\| \geq 1/(2\kappa)$.     Substitution into \eqref{eq:LB-prelim-1} yields
	\begin{align*}
	\eqref{eq:LB-prelim-1}
	 \geq 1 - \frac{64\kappa^4}{9}\left\|\bm{H}\right\|^2
	\end{align*}
	as claimed.

	\section{Proof for the lower bound in Lemma \ref{lem:lower-bound}}
\label{sec:proof-lower-bound}

Here, we establish Lemma \ref{lem:lower-bound}, towards which we intend to invoke Fano's inequality. Let $P_{ij}$ and $\tilde{P}_{ij}$ represent the distributions
of $M_{ij}$ and $\widetilde{M}_{ij}$, respectively, and denote by
$P$, $\tilde{P}$, $\hat{P}$ the distributions of $\bm{M}$, $\widetilde{\bm{M}}$
and $\widehat{\bm{M}}$, respectively. Then the KL divergence between
$P$ and $\tilde{P}$ can be calculated as
\[
\mathsf{KL}\big(\tilde{P}\,\|\,P\big)=\sum_{1\leq i,j\leq n}\mathsf{KL}\big(\tilde{P}_{ij}\,\|\,P_{ij}\big)=\sum_{1\leq i,j\leq n}\frac{(\Delta u_{i}^{\star}u_{j}^{\star})^{2}}{2\sigma^{2}}=\frac{\Delta^{2}}{\sigma^{2}}.
\]
where the first identity comes from the property of KL divergence
	for product measures, the second identity follows since $\mathsf{KL}(\mathcal{N}(\mu_{1},\sigma^{2})\,\|\ \mathcal{N}(\mu_{2},\sigma^{2}))=\frac{(\mu_{1}-\mu_{2})^{2}}{2\sigma^{2}}$, and the third one holds since $\|\bm{u}^{\star}\|_2=1$.
The same argument yields $\mathsf{KL}\big(\hat{P}\,\|\,P\big)=\Delta^{2}/\sigma^{2}$.
In view of \cite[Corollary 2.6]{tsybakov2009introduction}, if 
\[
\alpha\log2\geq\frac{1}{2}\left\{ \mathsf{KL}\big(\tilde{P}\,\|\,P\big)+\mathsf{KL}\big(\hat{P}\,\|\,P\big)\right\} =\frac{\Delta^{2}}{\sigma^{2}},
\]
then one has $p_{\mathrm{e}}\geq\frac{\log3-\log2}{\log2}-\alpha$.
Taking $\alpha=\log_{2}1.5-\varepsilon$ immediately establishes the
proof.

	\section{Proof of Lemma \ref{lem_high_order_terms}} 
\label{app_lem2}

	To establish this lemma, we exploit entrywise independence of $\bm{H}$ and develop a combinatorial trick.


        To begin with, we expand the quantity of interest as
	\begin{align*}
		\left(\bm{a}^\top \bm{H}^s \bm{u}^\star\right)^k & = \sum_{1\leq i_t^{(b)} \leq n, \,0 \leq t \leq s, \, 1\leq b \leq k} \prod_{b=1}^ka_{i_0^{(b)}} \left(\prod_{t=1}^s H_{i_{t-1}^{(b)}i_t^{(b)}}\right) u_{i_s^{(b)}}^\star .  
	\end{align*}
	 Throughout this section, we will use
	\begin{equation}
		\label{eq:defn:I}
		\mathcal{I} := \left\{i_t^{(b)} \mid 0\leq t \leq s , 1 \leq b \leq k\right\} ~\in~ [n]^{(s+1)k} 
	\end{equation}
	to denote such a collection of $(s+1)k$ indices. Thus, one can write
	\begin{align*}
	\mathbb{E}\left[ \left(\bm{a}^\top \bm{H}^s \bm{u}^\star\right)^k\right] 
		& = \sum_{\mathcal{I} \in [n]^{(s+1)k}}  \mathbb{E} \left[\prod_{b=1}^k  a_{i_0^{(b)}} \left(\prod_{t=1}^s H_{i_{t-1}^{(b)}i_t^{(b)}}\right) u_{i_s^{(b)}}^\star\right] .
	\end{align*}
	We shall often think of this sum-product graphically by viewing each index pair $\big(i_{t-1}^{(b)}, i_{t}^{(b)}\big)$ as a directed edge over the vertex set $[n]^{(s+1)k}$. As a result, this gives us a set of $sk$ edges in total $\{\bm{e}_t^{(b)}\mid 1\leq t\leq s, 1\leq b \leq k\}$, where $\bm{e}_t^{(b)}$ represents  $\big(i_{t-1}^{(b)}, i_{t}^{(b)}\big)$. 
	See Fig.~\ref{fig:edges-indices} for an illustration. In what follows, two edges $\bm{e}_{t_1}^{(b_1)}$ and $\bm{e}_{t_2}^{(b_2)}$ are said to be equivalent in $\mathcal{I}$, denoted by $$\bm{e}_{t_1}^{(b_1)}=\bm{e}_{t_2}^{(b_2)},$$ if the values of the corresponding vertices are identical (namely, $i_{t_1-1}^{(b_1)}=i_{t_2-1}^{(b_2)}$ and  $i_{t_1}^{(b_1)}=i_{t_2}^{(b_2)}$). 

	\begin{figure}
		\begin{center}
		\begin{tabular}{ccc}
			
			\includegraphics[width=0.4\textwidth]{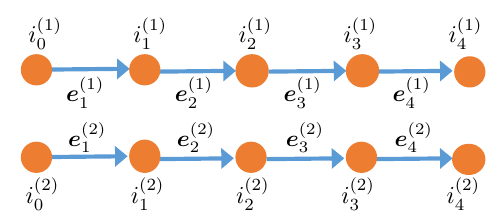} &	\qquad  & \includegraphics[width=0.4\textwidth]{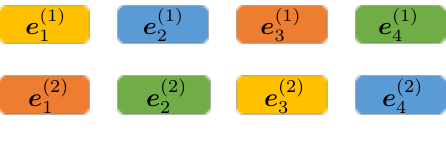} \tabularnewline
		\end{tabular}
		\end{center}
		\vspace{-1em}
		\caption{ (Left) An illustration of $\mathcal{I}:=\big\{i_t^{(b)} \mid 0\leq t \leq s , 1 \leq b \leq k\big\} $ and the corresponding edges $\{\bm{e}_t^{(b)}\mid 1 \leq t \leq s, 1\leq b \leq k\}$ when $s=4$ and $k=2$. (Right) An illustration of a type (or a partition of $sk$ edges), where the edges of the same color belong to the same subset of the partition. \label{fig:edges-indices} }
	\end{figure}

	When $(s+1)k \ll n$, most of the summands in the above expansion vanish.  In fact, for any summand associated with a given $\mathcal{I}$: as long as there exists a distinct edge $\big(i_{t^*-1}^{(b^*)}, i_{t^*}^{(b^*)}\big)$ (i.e.~not equal to any other edge associated with $\mathcal{I}$), then the contribution of this summand is zero, namely, 
	\begin{align}
	\mathbb{E} \left[\prod_{b=1}^k a_{i_0^{(b)}} \left(\prod_{t=1}^s H_{i_{t-1}^{(b)}i_t^{(b)}}\right) u_{i_s^{(b)}}^\star\right] & = \left(\prod_{b=1}^k a_{i_0^{(b)}}u_{i_s^{(b)}}^\star\right) \mathbb{E} \left[H_{i_{t^*-1}^{(b^*)}i_{t^*}^{(b^*)}}\right]  \mathbb{E} \left[\prod_{(t,b)\in[s]\times[k],\,(t,b)\neq (t^*,b^*)} H_{i_{t-1}^{(b)}i_t^{(b)}}\right] \nonumber\\
	& = 0.  \label{eq:singleton}
	\end{align}
	As a result, for any term with non-zero contribution, every edge $\big(i_{t-1}^{(b)}, i_{t}^{(b)} \big)$ must appear at least twice.


	To enable simple yet effective upper bounds on the non-vanishing terms, we group those terms of the same ``type'' and look at each ``type'' separately. Specifically: 
	\begin{itemize}
		\item We represent a collection of edges $\big\{ (i_{t-1}^{(b)}, i_t^{(b)}) \mid 1\leq t\leq s, 1\leq b \leq k \big\}$ as $\big\{ \bm{e}_t^{(b)} \mid 1\leq t\leq s, 1\leq b \leq k \big\}$. See Fig.~\ref{fig:edges-indices} for an illustration. 
		\item Each {\em type} encodes a set of constraints across edges, namely, which edges correspond to the same pair of vertices. To be precise, we define each type to be   
			a partition of $\{ \bm{e}_t^{(b)} \mid 1\leq t\leq s, 1\leq b \leq k\}$ into a disjoint union of subsets, so that all edges falling within the same subset are equivalent.  For instance, when $s=k=2$, one possible type is $\big\{ \{\bm{e}_1^{(1)}, \bm{e}_2^{(2)}\}, \{\bm{e}_1^{(2)}, \bm{e}_2^{(1)}\}    \big\}$, which encodes the constraints $\bm{e}_1^{(1)}=\bm{e}_2^{(2)}$ and $\bm{e}_1^{(2)}=\bm{e}_2^{(1)}$. 

		\item For each index collection $\mathcal{I}$ (defined in \eqref{eq:defn:I}), we write $\mathsf{type}(\mathcal{I})= \mathcal{T}$ if the associated edge set of $\mathcal{I}$ satisfies the constraints encoded by a type $\mathcal{T}$.  

		\item Define the set $\Gamma^{+}$ as 
			\begin{equation}
				\label{defn:Gamma+}
				\Gamma^{+}:~ \text{the set of all types s.t.~each subset of the associated partition has cardinality at least 2}.
			\end{equation}
	\end{itemize}	

	\noindent With this grouping strategy and \eqref{eq:singleton}  in mind, we can derive
	\begin{align}
	 \left|\mathbb{E}\left[\left(\bm{a}^\top \bm{H}^s \bm{u}^\star\right)^k\right]\right| 
	= & \left|\sum_{\mathcal{I} \in [n]^{(s+1)k}}  \mathbb{E} \left[\prod_{b=1}^ka_{i_0^{(b)}} \left(\prod_{t=1}^s H_{i_{t-1}^{(b)}i_t^{(b)}}\right) u_{i_s^{(b)}}^\star\right] \right| \nonumber\\
	= & \left| \sum_{\mathcal{T} \in \Gamma^{+}} ~ \sum_{\mathcal{I} \in [n]^{(s+1)k}: \, \mathsf{type}(\mathcal{I}) = \mathcal{T}} \mathbb{E} \left[\prod_{b=1}^ka_{i_0^{(b)}} \left(\prod_{t=1}^s H_{i_{t-1}^{(b)}i_t^{(b)}}\right) u_{i_s^{(b)}}^\star\right] \right|, \label{eq:expectation-1}
	\end{align}
	where the last line follows since all types outside $\Gamma^{+}$ have vanishing contributions (cf.~\eqref{eq:singleton}). To bound the right-hand side of \eqref{eq:expectation-1}, we need the following lemma. Here and throughout, $|\mathcal{T}|$ represents the number of non-empty subsets in the partition associated with $\mathcal{T}$. 
	\begin{lem} 
	\label{lem_multiple_partition}
	For any fixed unit vector $\bm{a} \in \mathbb{R}^{n \times 1}$ and any type $\mathcal{T} \in \Gamma^+$, one has
	\begin{equation}
	\left| \sum_{\mathcal{I} \in [n]^{(s+1)k}: \, \mathsf{type}(\mathcal{I}) = \mathcal{T}} \mathbb{E} \left[ \prod_{b=1}^k a_{i_0^{(b)}} \left(\prod_{t=1}^s H_{i_{t-1}^{(b)}i_t^{(b)}}\right) u_{i_s^{(b)}}^\star \right] \right| \leq B^{sk} \left(\frac{\sigma^2}{B^2}\right)^{|\mathcal{T}|} \left(\sqrt\frac{\mu}{n}\right)^k n^{|\mathcal{T}|} .
	\end{equation}

	\end{lem}

	Armed with this lemma, we can further obtain
	\begin{align}
	\eqref{eq:expectation-1} \leq & \sum_{\mathcal{T} \in \Gamma^+}B^{sk} \left(\frac{\sigma^2}{B^2}\right)^{|\mathcal{T}|} \left(\sqrt{\frac{\mu}{n}}\right)^kn^{|\mathcal{T}|} 
	= \sum_{\mathcal{T} \in \Gamma^+}B^{sk-2|\mathcal{T}|} \left(n \sigma^2\right)^{|\mathcal{T}|} \left(\sqrt{\frac{\mu}{n}}\right)^k
	\nonumber\\
	\leq & \,  \sum_{\mathcal{T} \in \Gamma^+} \sum_{l=1}^{sk/2} \ind_{\{|\mathcal{T}|=l\}} B^{sk-2|\mathcal{T}|} \left(n \sigma^2\right)^{|\mathcal{T}|} \left(\sqrt{\frac{\mu}{n}}\right)^k  \nonumber\\
	= & \sum_{l=1}^{sk/2} \left\{\sum_{\mathcal{T}\in \Gamma^+} \ind_{\{|\mathcal{T}|=l\}} \right\} B^{sk-2l} \left(n \sigma^2\right)^{l} \left(\sqrt{\frac{\mu}{n}}\right)^k,
	\label{eq:expetation-2}	  
	\end{align}
	where we have grouped the types based on their cardinality. The last inequality results from $|\mathcal{T}| \leq sk/2$ in Lemma \ref{lem_gamma_plus}. The following lemma bounds the number of distinct types having the same cardinality.
	%

	%
%

	 %
	 \begin{lem} 
	 	\label{lem_gamma_plus_l}
	 	$ \sum_{\mathcal{T}\in \Gamma^+} \ind_{\{|\mathcal{T}|=l\}}  \leq \binom{sk}{l}\, l^{sk-l} \,\leq 2^{sk} l^{sk-l}$. 
	 \end{lem}
	 \begin{proof}
	 	The quantity of interest is the number of ways to partition $sk$ edges into $l$ disjoint subsets, where each subset contains at least 2 edges. To bound this quantity, we first pick $l$ edges and assign each of them to a distinct subset; there are $\binom{sk}{l}$ different ways to achieve it.  We still have $sk-l$ edges left, and the number of ways to assign them to $l$ subsets is clearly upper bounded by $l^{sk-l}$. This concludes the proof. 
	 \end{proof}
	
	 Consequently, Lemma \ref{lem_gamma_plus_l} yields
	 \begin{align*}
	 \eqref{eq:expetation-2} \leq  & \, 2^{sk}\sum_{l=1}^{sk/2}  \left(Bl\right)^{sk-2l} \left(n \sigma^2l\right)^{l} \left(\sqrt{\frac{\mu}{n}}\right)^k \\
	 \stackrel{(\mathrm{i})}\leq & \, 2^{sk}\sum_{l=1}^{sk/2}  \left(\frac{Bsk}{2}\right)^{sk-2l} \left(\frac{n \sigma^2 sk}{2}\right)^{l} \left(\sqrt{\frac{\mu}{n}}\right)^k \\
	 \stackrel{(\mathrm{ii})}\leq & \, 2^{sk}\sum_{l=1}^{sk/2}  \max\left\{\left(\frac{Bsk}{2}\right)^{sk} ,\left(\frac{n \sigma^2 sk}{2}\right)^{sk/2}\right\} \left(\sqrt{\frac{\mu}{n}}\right)^k \\
	 = & \, \frac{sk}{2}  \max\left\{\left(Bsk\right)^{sk} ,\left(2n \sigma^2 sk\right)^{sk/2}\right\} \left(\sqrt{\frac{\mu}{n}}\right)^k , 
	 \end{align*}
	 where  (i) uses the condition $l \leq sk/2$, and (ii) relies on the fact that $a^{sk-2l} b^l = a^{sk-2l} (\sqrt{b})^{2l}\leq \max\{ a^{sk}, (\sqrt{b})^{sk} \}$ for any $a,b>0$. 
  	
	Taken collectively, the preceding bounds conclude the proof of Lemma \ref{lem_high_order_terms}, provided that Lemma \ref{lem_multiple_partition} can be established.

\section{Proof of Lemma \ref{lem_multiple_partition}}
\label{appendix:proof-lem_multiple_partition}

To begin with,
\begin{align}
\left| \mathbb{E}\left[\prod_{b=1}^{k}a_{i_{0}^{(b)}}\left(\prod_{t=1}^{s} H_{i_{t-1}^{(b)}i_{t}^{(b)}} \right)u_{i_{s}^{(b)}}^{\star}\right] \right| 
	& \leq\mathbb{E}\left[\prod_{b=1}^{k}\big|a_{i_{0}^{(b)}}\big|\left(\prod_{t=1}^{s}\big|H_{i_{t-1}^{(b)}i_{t}^{(b)}}\big|\right)\big|u_{i_{s}^{(b)}}^{\star}\big|\right]\nonumber \\
 	& \leq\left(\sqrt{\frac{\mu}{n}}\right)^{k}\left(\prod_{b=1}^{k}\big|a_{i_{0}^{(b)}}\big|\right)\mathbb{E}\left[\prod_{b=1}^{k}\prod_{t=1}^{s}\big|H_{i_{t-1}^{(b)}i_{t}^{(b)}}\big|\right],\label{eq:expectation-3}
\end{align}
where the last inequality holds since each entry of $\bm{u}^{\star}$
is bounded in magnitude by $\sqrt{\mu/n}$ (see Definition \ref{def_incoherence_parameter}). 

According to the definition of  
$\Gamma^+$ (see \eqref{defn:Gamma+}),  
for any type $\mathcal{T}\in\Gamma^{+}$, each edge $\big(i_{t-1}^{(b)},i_{t}^{(b)}\big)$
is repeated at least twice. The total number of distinct edges
is exactly $|\mathcal{T}|$. For notational simplicity, suppose the
distinct edges are $\bm{e}_{1},\cdots,\bm{e}_{|\mathcal{T}|}$, where
$\bm{e}_{i}$ is repeated by $l_{i}\geq2$ times. Then we can write
\[
\mathbb{E}\left[\prod_{b=1}^{k}\prod_{t=1}^{s}\big|H_{i_{t-1}^{(b)}i_{t}^{(b)}}\big|\right]=\prod_{i=1}^{|\mathcal{T}|}\mathbb{E}\left[\big|H_{\bm{e}_{i}}\big|^{l_{i}}\right]\leq\prod_{i=1}^{|\mathcal{T}|}\sigma^{2}B^{l_{i}-2}=\left(\frac{\sigma^{2}}{B^{2}}\right)^{|\mathcal{T}|}B^{\sum_{i}l_{i}}=\left(\frac{\sigma^{2}}{B^{2}}\right)^{|\mathcal{T}|}B^{sk},
\]
where the inequality follows since
\[
\mathbb{E}\left[\big|H_{ij}\big|^{l}\right]\leq B^{l-2}\mathbb{E}\left[\big|H_{ij}\big|^{2}\right] \leq \sigma^{2}B^{l-2},\qquad l\geq2.
\]
Substitution into (\ref{eq:expectation-3}) yields
\begin{align}
 & \left|\sum_{\mathcal{I}\in[n]^{(s+1)k}:\,\mathsf{type}(\mathcal{I})=\mathcal{T}}\mathbb{E}\left[\prod_{b=1}^{k}a_{i_{0}^{(b)}}\left(\prod_{t=1}^{s} H_{i_{t-1}^{(b)}i_{t}^{(b)}} \right)u_{i_{s}^{(b)}}^{\star}\right]\right|\nonumber \\
	& \qquad\leq\left(\sqrt{\frac{\mu}{n}}\right)^{k}\left(\frac{\sigma^{2}}{B^{2}}\right)^{|\mathcal{T}|}B^{sk}\left\{ \sum_{\mathcal{I}\in[n]^{(s+1)k}:\,\mathsf{type}(\mathcal{I})=\mathcal{T}} ~\prod_{b=1}^{k}\big|a_{i_{0}^{(b)}}\big|\right\}.\label{eq:expectation-4}
\end{align}

It then boils down to controlling $ \sum_{\mathcal{I}\in[n]^{(s+1)k}:\,\mathsf{type}(\mathcal{I})=\mathcal{T}}\prod_{b=1}^{k}\big|a_{i_{0}^{(b)}}\big|$. This is achieved by the following lemma, which we establish in Appendix \ref{appendix-proof-lem:UB-a-prod}. 
\begin{lem} \label{lem:UB-a-prod}
	For any type $\mathcal{T}\in \Gamma^+$ and any unit vector $\bm{a}\in \mathbb{R}^n$, one has
	\begin{equation}
		\sum_{\mathcal{I}\in[n]^{(s+1)k}:\,\mathsf{type}(\mathcal{I})=\mathcal{T}} ~\prod_{b=1}^{k}\big|a_{i_{0}^{(b)}}\big| \leq n^{|\mathcal{T}|}. 
	\end{equation}
\end{lem}
This lemma together with \eqref{eq:expectation-4} concludes the proof.

\section{Proof of Lemma \ref{lem:UB-a-prod}}
\label{appendix-proof-lem:UB-a-prod}

	\subsection{Graphical tools} 

	The proof of Lemma \ref{lem:UB-a-prod} is combinatorial in nature. Before proceeding, we  introduce a few graphical notions that will be useful.

	To begin with, recall that the index collection we have used so far is 
	\begin{equation}
		\mathcal{I} := \left\{i_t^{(b)} \mid 0\leq t \leq s , 1 \leq b \leq k\right\} ~\in~ [n]^{(s+1)k} ,
	\end{equation}
	which involves $(s+1)k$ vertices. As it turns out, it is helpful to further augment it into $2(s+1)k$ vertices via simple duplication.  Specifically, introduce the following set of $2(s+1)k$ vertices
	\begin{equation}
		\label{eq:defn:V}
		\mathcal{V} = \left\{q_t^{(b)} \mid 0\leq t \leq s, 1 \leq b \leq k\right\}   \,\bigcup\, \left\{p_t^{(b)} \mid 1\leq t \leq s+1, 1 \leq b \leq k\right\}. 
	\end{equation}
	Throughout this paper, we shall 
	\begin{itemize}
		\item view $\{q_t^{(b)} \mid 0\leq t \leq s, 1 \leq b \leq k \}$ as a copy of $\{i_t^{(b)} \mid 0\leq t \leq s , 1 \leq b \leq k \} $;
		\item view  $\{p_t^{(b)} \mid 1\leq t \leq s+1, 1 \leq b \leq k \}$ as another copy of $\{i_{t-1}^{(b)} \mid 1\leq t \leq s+1 , 1 \leq b \leq k \} $.
	\end{itemize}
	The main incentive is that it allows us to reparametrize
	\begin{equation}
		\label{eq:reparametrize}
		\prod_{b=1}^k a_{i_0^{(b)}} \left(\prod_{t=1}^s H_{i_{t-1}^{(b)}i_t^{(b)}}\right) u_{i_s^{(b)}}^\star  ~\longrightarrow~ \prod_{b=1}^ka_{q_0^{(b)}} \left(\prod_{t=1}^s H_{p_{t}^{(b)}q_t^{(b)}}\right) u_{p_{s+1}^{(b)}}^\star 
	\end{equation}

	Recall that we have categorized $\mathcal{I}$ into different ``types''. Given that $\mathcal{V}$ is simply a redundant representation of $\mathcal{I}$, we can also associate a type with each $\mathcal{V}$ by enforcing proper constraints. These constraints can be encoded via the following graphical notion.

	\begin{defn}[\textbf{Induced graph}] 
		\label{def_induced_graph} 
		For any given type $\mathcal{T}$, we define the induced graph $\mathcal{G}(\mathcal{T})$ with the vertex set $\mathcal{V}$ and the (undirected) edge set $\mathcal{E}(\mathcal{T})$ as
		 $\mathcal{E}(\mathcal{T}) = \mathcal{E}_1 \cup {\mathcal{E}_2(\mathcal{T})}$, where
		\begin{equation}
		\begin{aligned}
			\mathcal{E}_1 & \triangleq \left\{ \big( p_t^{(b)},q_{t-1}^{(b)} \big) ~\big|~  1\leq t \leq s+1, 1 \leq b \leq k  \right\}; \\
			\mathcal{E}_2(\mathcal{T}) & \triangleq \left\{ \big( p_{t_1}^{(b_1)}, p_{t_2}^{(b_2)} \big), \big( q_{t_1}^{(b_1)},q_{t_2}^{(b_2)} \big) ~\big|~ \forall t_1, b_1, t_2, b_2: \big(i_{t_{1}-1}^{(b_{1})},i_{t_{1}}^{(b_{1})}\big) = \big(i_{t_{2}-1}^{(b_{2})},i_{t_{2}}^{(b_{2})}\big) \right\} . 
		\end{aligned}
		\end{equation}
	\end{defn}
	In words, 
	\begin{itemize}
		\item[(1)] we connect the vertices $p_t^{(b)}$ (which is a copy of $i_t^{(b)}$) and $q_{t-1}^{(b)}$ (which represents $i_{t-1}^{(r)}$) by an edge;  
		\item[(2)] whenever two edges $\big(i_{t_{1}-1}^{(b_{1})},i_{t_{1}}^{(b_{1})}\big)$ and $\big(i_{t_{2}-1}^{(b_{2})},i_{t_{2}}^{(b_{2})}\big)$ are equivalent in $\mathcal{I}$ (or $\mathcal{T}$), we draw two edges connecting the corresponding vertices in the induced graph.
	\end{itemize}
	As an illustration, Fig.~\ref{fig:induced-graph} displays the induced graph for a simple example with $s=4$ and $k=2$, where
	$$
\mathcal{T}=\Big\{\big\{\underset{\bm{e}_{1}^{(1)}}{\underbrace{\big(i_{0}^{(1)},i_{1}^{(1)}\big)}},\underset{\bm{e}_{1}^{(2)}}{\underbrace{\big(i_{0}^{(2)},i_{1}^{(2)}\big)}}\big\},\big\{\underset{\bm{e}_{2}^{(1)}}{\underbrace{\big(i_{1}^{(1)},i_{2}^{(1)}\big)}},\underset{\bm{e}_{4}^{(1)}}{\underbrace{\big(i_{3}^{(1)},i_{4}^{(1)}\big)}}\big\},\big\{\underset{\bm{e}_{2}^{(2)}}{\underbrace{\big(i_{1}^{(2)},i_{2}^{(2)}\big)}},\underset{\bm{e}_{4}^{(2)}}{\underbrace{\big(i_{3}^{(2)},i_{4}^{(2)}\big)}}\big\},\big\{\underset{\bm{e}_{3}^{(1)}}{\underbrace{\big(i_{2}^{(1)},i_{3}^{(1)}\big)}}\big\},\big\{\underset{\bm{e}_{3}^{(2)}}{\underbrace{\big(i_{2}^{(2)},i_{3}^{(2)}\big)}}\big\}\Big\}	.$$
	Here, the 45-degree lines correspond to the edges in $\mathcal{E}_1$, and the remaining edges come from  $\mathcal{E}_2(\mathcal{T})$.  Throughout the rest of the paper, we will abuse the notation and write $\mathsf{type}(\mathcal{V})=\mathcal{T}$ if $\mathcal{V}$ is induced by an index set of type $\mathcal{T}$.

	\begin{figure}
		\centering
		\begin{tabular}{ccc}
			\includegraphics[width=0.4\textwidth]{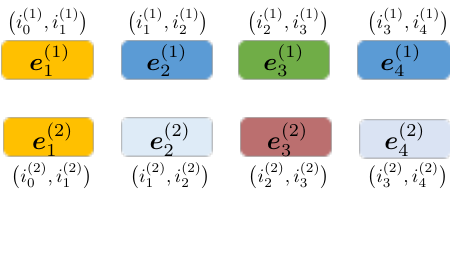} & \qquad \qquad & \includegraphics[width=0.4\textwidth]{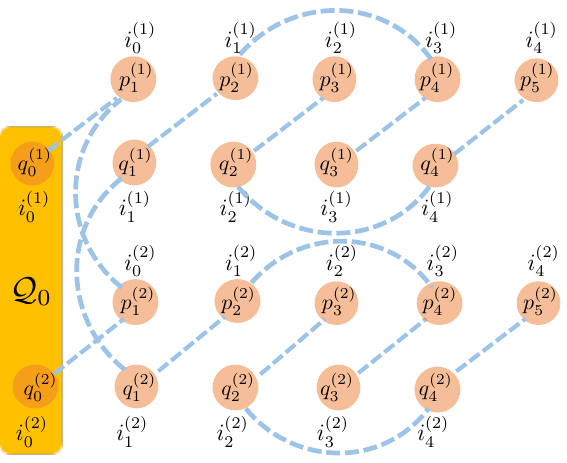}   \tabularnewline 
		\end{tabular} 
		\caption{An example of the induced graph, where $s=4$, $k=2$. The left plot shows the partition, while the right plot displays  the induced graph and the vertex set $\mathcal{Q}_0$ (the yellow region). 		
		  \label{fig:induced-graph}}
	\end{figure}

	One useful feature of the above induced graph is that: each edge subset in the partition associated with $\mathcal{T}$ corresponds to a connected component in $\mathcal{G}(\mathcal{T})$. For our purpose, it is convenient to divide all connected components of  $\mathcal{G}(\mathcal{T})$ into two classes. To this end, we first define $\mathcal{Q}_0 \triangleq \big\{q_0^{(b)} \mid 1\leq b \leq k \big\}$ (as illustrated in Fig.~\ref{fig:induced-graph}). 
	\begin{itemize}
		\item {\em Class 1}: a connected component $\mathcal{C}$  belongs to Class 1 if $| \mathcal{C}\cap \mathcal{Q}_0 | \leq 1$;
		\item {\em Class 2}: a connected component $\mathcal{C}$  belongs to Class 2 if $| \mathcal{C}\cap \mathcal{Q}_0 | \geq 2$. 
	\end{itemize}
	We denote by $m_1(\mathcal{T})$ (resp.~$m_2(\mathcal{T})$) the total number of Class 1 (resp.~2) connected components in $\mathcal{G}(\mathcal{T})$.

	

\subsection{Proof of Lemma \ref{lem:UB-a-prod}}

Making use of the augmented set $\mathcal{V}$ (defined in \eqref{eq:defn:V}) and the way of reparametrization \eqref{eq:reparametrize}, we can write 
	\begin{align}
		 \sum_{\mathcal{I} \in [n]^{(s+1)k}:\, \mathsf{type}(\mathcal{I}) = \mathcal{T}} \mathbb{E} \left[ \prod_{b=1}^k \big| a_{i_0^{(b)} } \big|  \right]   
		\leq \sum_{\substack{\mathcal{V} \in [n]^{2(s+1)k}:\, \mathsf{type}(\mathcal{V})= \mathcal{T} }} \mathbb{E} \left[ \prod_{b=1}^k \big| a_{q_0^{(b)}}  \big|  \right]  . \label{eq:expectation-7}
	\end{align}
	As we shall see, the benefit of this bound is to allow us sum over the connected components of $\mathcal{G}(\mathcal{T})$ in a separate manner, since all indices in the same connected component must be identical. Specifically, 
	\begin{itemize}
		\item Let $\mathcal{X}_1,\mathcal{X}_2,\cdots$ (resp.~$\mathcal{Y}_1,\mathcal{Y}_2,\cdots$) denote the collection of Class 1 (resp.~Class 2) connected components in $\mathcal{G}(\mathcal{T})$;
		\item Denote by $x_i$ (resp.~$y_i$) the value assigned to all indices in $\mathcal{X}_i$ (resp.~$\mathcal{Y}_i$);	 
	\end{itemize}
	With these notations in mind, we can decompose
	\begin{align*}
		\eqref{eq:expectation-7} & \leq\sum_{x_{1}=1}^{n}\sum_{x_{2}=1}^{n}\cdots\sum_{y_{1}=1}^{n}\sum_{y_{2}=1}^{n}\cdots\mathbb{E}\left[\big|a_{x_{1}}\big|^{|\mathcal{X}_{1}\cap\mathcal{Q}_{0}|}\cdot\big|a_{x_{2}}\big|^{|\mathcal{X}_{2}\cap\mathcal{Q}_{0}|}\cdots\big|a_{y_{1}}\big|^{|\mathcal{Y}_{1}\cap\mathcal{Q}_{0}|}\cdot\big|a_{y_{2}}\big|^{|\mathcal{Y}_{2}\cap\mathcal{Q}_{0}|}\cdots\right]\\
 		& =\mathbb{E}\left[\Bigg(\sum_{x_{1}=1}^{n}\big|a_{x_{1}}\big|^{|\mathcal{X}_{1}\cap\mathcal{Q}_{0}|}\Bigg)\Bigg(\sum_{x_{2}=1}^{n}\big|a_{x_{2}}\big|^{|\mathcal{X}_{2}\cap\mathcal{Q}_{0}|}\Bigg)\cdots\Bigg(\sum_{y_{1}=1}^{n}\big|a_{y_{1}}\big|^{|\mathcal{Y}_{1}\cap\mathcal{Q}_{0}|}\Bigg)\Bigg(\sum_{y_{2}=1}^{n}\big|a_{y_{2}}\big|^{|\mathcal{Y}_{2}\cap\mathcal{Q}_{0}|}\Bigg)\cdots\right]\\
 		& \overset{(\text{i})}{\leq}\mathbb{E}\left[\Bigg(\sum_{x_{1}=1}^{n}1\Bigg)\Bigg(\sum_{x_{2}=1}^{n}1\Bigg)\cdots\Bigg(\sum_{y_{1}=1}^{n}\big|a_{y_{1}}\big|^{2}\Bigg)\Bigg(\sum_{y_{2}=1}^{n}\big|a_{y_{2}}\big|^{2}\Bigg)\cdots\right]\\
 		& \overset{(\text{ii})}{=} \Bigg(\sum_{x_{1}=1}^{n}1\Bigg)\Bigg(\sum_{x_{2}=1}^{n}1\Bigg)\cdots 1\cdot 1 \cdots = n^{m_{1}(\mathcal{T})},
\end{align*}
	where $m_{1}(\mathcal{T})$ is the total number of Class 1 connected components in the induced graph $\mathcal{G}(\mathcal{T})$.  Here, (i) comes from the definitions of Class 1 and Class 2 connected components, and (ii) follows since $\|\bm{a}\|_{2}^{2}=1$.

	We can thus  finish the proof by observing that $m_1(\mathcal{T}) \leq |\mathcal{T}| $, as claimed  in the following lemma. 
	\begin{lem} \label{lem_gamma_plus} 
		For any type $\mathcal{T} \in \Gamma^+$, one has $m_1(\mathcal{T}) \leq |\mathcal{T}| \leq sk/2$.
	\end{lem}

	\subsection{Proof of Lemma \ref{lem_gamma_plus}}

	

		Introduce two disjoint vertex sets:
		\begin{align}
			\mathcal{Q}_{\textbackslash{}0} &\triangleq \big\{q_t^{(b)} \mid 1 \leq t \leq s, 1 \leq b \leq k \big\} ~\subset \mathcal{V}; \\
			\mathcal{Q}_0 &\triangleq \big\{q_0^{(b)} \mid 1\leq b \leq k \big\} ~\subset \mathcal{V}.
		\end{align}
		Clearly, one can define a partition of $\mathcal{Q}_{\textbackslash{}0}$ --- denoted by  $\mathcal{T}_{\mathcal{Q}_{\textbackslash{}0}}$ --- induced by $\mathcal{T}$. Specifically, we say that $q_{t_1}^{(b_1)}$ and $q_{t_2}^{(b_2)}$ belong to the same connected subgraph of $\mathcal{Q}_{\textbackslash{}0}$ if $\big(i_{t_{1}-1}^{(b_{1})},i_{t_{1}}^{(b_{1})}\big) = \big(i_{t_{2}-1}^{(b_{2})},i_{t_{2}}^{(b_{2})}\big)$ in $\mathcal{T}$.  Clearly, $|\mathcal{T}| = \left| \mathcal{T}_{\mathcal{Q}_{\textbackslash{}0}} \right|$. 
		%
		%
	
		In addition, for any connected subgraph $\mathcal{C}_q$ in $\mathcal{Q}_{\textbackslash{}0}$, we denote by $\mathcal{C}$ the corresponding connected component in $\mathcal{G}(\mathcal{T})$. 
		We can thus define a mapping $\psi$ that maps $\mathcal{C}_q$ to $\mathcal{C}$, whose domain is the set of all connected subgraphs in $\mathcal{Q}_{\textbackslash{}0}$. 
		%
		We claim that the collection of Class 1 connected components --- denoted by $\{\mathcal{X}_j\}$ --- obeys
		\begin{equation}  \label{eq:map-psi}
			\left\{\mathcal{X}_j \mid 1 \leq j \leq m_1(\mathcal{T})\right\} ~\subseteq~ \mathrm{Im}\psi,
		\end{equation}
		where $\mathrm{Im}\psi$ denotes the image of $\psi$, and $m_1(\mathcal{T})$ is the total number of Class 1 connected components. If this holds true, then 
		$$
			|\mathcal{T}| = \left| \mathcal{T}_{\mathcal{Q}_{\textbackslash{}0}} \right| \geq \left|\mathrm{Im} \psi \right| \geq | \left\{\mathcal{X}_j \mid 1 \leq j \leq m_1(\mathcal{T})\right\} | = m_1(\mathcal{T}) .
		$$

		To justify the claim \eqref{eq:map-psi}, it suffices to show that $\mathcal{X}_j \cap \mathcal{Q}_{\textbackslash{}0} \neq \emptyset$. Given that $\mathcal{T} \in \Gamma^+$ (so that each subset of the partition associated with $\mathcal{T}$ has cardinality at least 2) and that $\mathcal{X}_j$ is defined over the induced graph $\mathcal{G}(\mathcal{T})$ (which is a redundant representation of the original index collection), one must have
		%
		\begin{equation}  
		\max\left\{\left|\mathcal{X}_j \cap \left\{p_t^{(b)} \mid 1\leq t \leq s+1, 1 \leq b \leq k\right\}\right|, \left|\mathcal{X}_j \cap \left\{q_t^{(b)} \mid 0\leq t \leq s, 1 \leq b \leq k\right\}\right| \right\} \geq 2 .
		\end{equation}
		Further, from the construction of the induced graph,  it is easily seen that
		$$
		\left|\mathcal{X}_j \cap \left\{p_t^{(b)} \mid 1\leq t \leq s+1, 1 \leq b \leq k\right\}\right| = \left|\mathcal{X}_j \cap \left\{q_t^{(b)} \mid 0\leq t \leq s, 1 \leq b \leq k\right\}\right|,
		$$
		thus resulting in
		$$
		\left|\mathcal{X}_j \cap \left\{q_t^{(b)} \mid 0\leq t \leq s, 1 \leq b \leq k\right\}\right| = \left|\mathcal{X}_j \cap \left(\mathcal{Q}_0 \cup \mathcal{Q}_{\textbackslash{}0} \right)\right| \geq 2 .
		$$
		However, by definition of Class 1 connected component,  we have $\left|\mathcal{X}_j \cap \mathcal{Q}_0 \right| \leq 1$, implying that $\mathcal{X}_j \cap \mathcal{Q}_{\textbackslash{}0} \neq \emptyset$. This establishes the claim \eqref{eq:map-psi}. 

		Finally, when $\mathcal{T} \in \Gamma^+$, each connected subgraph of $\mathcal{T}_{\mathcal{Q}_{\textbackslash{}0}}$ also contains at least two nodes. 
		As a consequence,  
		$$
		|\mathcal{T}| = | \mathcal{T}_{\mathcal{Q}_{\textbackslash{}0}} | \leq | \mathcal{Q}_{\textbackslash{}0} | / 2 = {sk}/{2} .
		$$

	\section{Proof of Corollary \ref{cor_high_order_terms}}
\label{appendix:proof-cor_high_order_terms}

In the sequel, we assume that $20 \log n$ is an integer to avoid the clumsy notation $\lfloor 20 \log n\rfloor$. But it is straightforward to extend it to the case where $20 \log n$ is not an integer. 

It follows from Markov's inequality that for any even integer $k$,
		\[
			\mathbb{P} \left(\left|\bm{a}^\top \bm{H}^s \bm{u}^\star\right| \geq \tau \right) \leq \frac{  \mathbb{E}\big[ \left|\bm{a}^\top \bm{H}^s \bm{u}^\star\right|^k \big]  }{\tau^k}.
		\]
		This combined with Lemma \ref{lem_high_order_terms} gives
		\begin{align*}
			\mathbb{P} \left(\left|\bm{a}^\top \bm{H}^s \bm{u}^\star\right| \leq \left( c_2\max\left\{B\log n, \sqrt{n \sigma^2 \log n}\right\} \right)^s\sqrt\frac{\mu}{n}\right) 
			& \leq \frac{\left|\mathbb{E}\big[\left(\bm{a}^\top\bm{H}^s \bm{u}^\star\right)^k\big]\right|}{\left(c_2 \max\left\{B\log n, \sqrt{n \sigma^2 \log n}\right\}\right)^{sk}\left(\sqrt\frac{\mu}{n}\right)^k} \\
			& \leq \frac{sk}{2} \max\left\{\left(\frac{sk}{c_2\log n}\right)^{sk}, \left(\frac{2sk}{c_2^2\log n}\right)^{sk/2}\right\}.
		\end{align*}
		For any $s\leq  20 \log n$, choose $k$ such that $sk \in \left[20\log n, 40\log n\right]$. It is straightforward to show --- using the union bound --- that with probability at least $1- O( n^{-10})$, 
		$$
			\left|\bm{a}^\top \bm{H}^s \bm{u}^\star\right| \leq \left(c_2 \max\left\{B\log n, \sqrt{n \sigma^2 \log n}\right\}\right)^s \sqrt{\mu / n}, \qquad \forall s \leq  20 \log n ,
		$$
		as long as $c_2>0$ is some sufficiently large constant. 

	\section{Proof of Corollary \ref{cor:eigenvalue-pertubation-rank-r-A}}
\label{sec:corollary-proof-eigenvalue-rank-r}

Taking $\bm{a}=\bm{u}_{i}^{\star}$ for some $1\leq i\leq r$ in
\eqref{thm-rankr-UB1}, we see that with high probability, 
\[
\Bigg|\bm{u}_{i}^{\star\top}\Bigg(\bm{u}_{l}-\sum_{j=1}^{r}\frac{\lambda_{j}^{\star}\bm{u}_{j}^{\star\top}\bm{u}_{l}}{\lambda_{l}}\bm{u}_{j}^{\star}\Bigg)\Bigg|\lesssim\max\left\{ B \log n,\sqrt{n\sigma^{2}\log n}\right\} \frac{\kappa}{|\lambda_l|}\sqrt{\frac{\mu r}{n}}.
\]
In addition, since $\bm{u}_{1}^{\star},\cdots,\bm{u}_{r}^{\star}$
are orthogonal to each other, we have that: for any $1\leq i\leq r$,
\[
	\Bigg|\bm{u}_{i}^{\star\top}\Bigg(\bm{u}_{l}-\sum_{j=1}^{r}\frac{\lambda_{j}^{\star}\bm{u}_{j}^{\star\top}\bm{u}_{l}}{\lambda_{l}}\bm{u}_{j}^{\star}\Bigg)\Bigg|
	= \Bigg|\bm{u}_{i}^{\star\top}\bm{u}_{l}-\frac{\lambda_{i}^{\star}\bm{u}_{i}^{\star\top}\bm{u}_{l}}{\lambda_{l}}\Bigg|
	= \big|\bm{u}_{i}^{\star\top}\bm{u}_{l}\big|\Bigg|\frac{\lambda_{l}-\lambda_{i}^{\star}}{\lambda_{l}}\Bigg| 
	\geq \big|\bm{u}_{i}^{\star\top}\bm{u}_{l}\big|   \min_{1\leq j\leq r}\Bigg|\frac{\lambda_{l}-\lambda_{j}^{\star}}{\lambda_{l}}\Bigg|.
\]
Therefore, combining the above two bounds yields
\begin{equation}
	\min_{1 \leq j \leq r} \Bigg|\frac{\lambda_{l}-\lambda_{j}^{\star}}{\lambda_{l}}\Bigg| \max _{1 \leq i \leq r} \big|\bm{u}_{i}^{\star\top}\bm{u}_{l}\big| \lesssim \max\left\{ B\log n,\sqrt{n\sigma^{2}\log n}\right\} \frac{\kappa}{|\lambda_l|}\sqrt{\frac{\mu r}{n}}. \label{eq:UB1-lambda-u}
\end{equation}

Finally, it comes from Lemma \ref{lem_l2_convergence-rank-r} that if $\|\bm{H}\|\ll1/\kappa^2$, then $\sum_{1\leq i\leq r}\big|\bm{u}_{i}^{\star\top}\bm{u}_{l}\big|^2 \gtrsim 1$, and hence 
\begin{equation}
	\max_{1\leq j\leq r}\big|\bm{u}_{i}^{\star\top}\bm{u}_{l}\big| \geq  \sqrt{ \frac{1}{r} \sum_{1\leq i \leq r}\big|\bm{u}_{i}^{\star\top}\bm{u}_{l}\big|^2 }  \gtrsim 1 / \sqrt{r}.
\end{equation}
 This combined with \eqref{eq:UB1-lambda-u} establishes the claim,  as long as the spectral norm condition on $\|\bm{H}\|$ can be guaranteed.  In view of Lemma \ref{lem:H-norm}, we have $\|\bm{H}\|\ll1/\kappa^2$ under the condition \eqref{Bsigma_cond_4}, thus concluding the proof.

\section{Proof of Corollary \ref{cor:linear-form-rankr}}
\label{sec:corollary-proof-linear-form-rankr}

	Since $\|\bm{a}\|_2=1$, it follows immediately from \eqref{thm-rankr-UB2} that
	\begin{align}
	\left\|\bm{a}^\top \bm{U}\right\|_2 
	& = \sqrt{\sum_{l=1}^r \big| \bm{a}^\top \bm{u}_l \big|^2} \nonumber \\
	& \lesssim \sqrt{\sum_{l=1}^r   \left( \left| \sum_{j=1}^r \frac{\lambda_j^\star \bm{u}_j^{\star \top}\bm{u}_l}{\lambda_l} \bm{a}^\top \bm{u}_j^\star \right| 
		+ \max\left\{B \log n, \sqrt{n \sigma^2 \log n}\right\}\kappa^2\sqrt\frac{\mu r}{n}\right)^2}. \label{eq:UB-linear-fomrs-rankr}
	\end{align}
	Further, we have
	\begin{align}
	\eqref{eq:UB-linear-fomrs-rankr} 
		& \stackrel{(\mathrm{i})}{\leq}  \sqrt{\sum_{l=1}^r \left|\sum_{j=1}^r \frac{\lambda_j^\star \bm{u}_j^{\star \top}\bm{u}_l}{\lambda_l}\bm{a}^\top \bm{u}_j^\star \right|^2} +  \sqrt{ \sum_{l=1}^r \left( \max\left\{B \log n, \sqrt{n \sigma^2 \log n}\right\}\kappa^2 \sqrt\frac{\mu r}{n}  \right)^2 } \nonumber \\
	& \stackrel{(\mathrm{ii})}{\leq} \sqrt{\sum_{l=1}^r \left(\sum_{j=1}^r \left|\frac{\lambda_j^\star \bm{u}_j^{\star\top} \bm{u}_l}{\lambda_l}\right|^2\right)\left(\sum_{j=1}^r \left|\bm{a}^\top \bm{u}_j^\star\right|^2\right)}+  \max\left\{B \log n, \sqrt{n \sigma^2 \log n}\right\}\kappa^2 r\sqrt\frac{\mu }{n} \nonumber \\
	& \stackrel{(\mathrm{iii})}{\lesssim} \sqrt{r \kappa^2 \left(\sum_{j=1}^r \left|\bm{a}^\top \bm{u}_j^\star\right|^2\right)}+  \max\left\{B \log n, \sqrt{n \sigma^2 \log n}\right\}\kappa^2 r \sqrt\frac{\mu}{n} \nonumber \\
	& =\kappa \sqrt{r} \left\|\bm{a}^\top \bm{U}^\star\right\|_2+\max\left\{B \log n, \sqrt{n \sigma^2 \log n}\right\} \kappa^2 r\sqrt\frac{\mu}{n},
	\end{align}
	where $\mathrm{(i)}$ and $\mathrm{(ii)}$ make use of the Minkowski and the Cauchy-Schwarz inequalities, respectively, and (iii) result from the facts $|\lambda_j^\star/\lambda_l| \lesssim \kappa$ (which follows from \eqref{eq:BF-bound-rank-r} and \eqref{Bsigma_cond_4}) and $\sum_{j=1}^r \left|\bm{u}_j^{\star \top} \bm{u}_l\right|^2 \leq 1$. This concludes the proof.

	\section{Proof of Theorem \ref{cor:l2_convergence-rank-2-dilation}}
\label{appendix:proof-rank_2_dilation_matrix}

To simplify presentation, we introduce the following notation throughout this section: 
\[
\widetilde{\bm{u}}_{1}^{\star}=\frac{1}{\sqrt{2}}\left(\begin{array}{c}
\bm{u}^{\star}\\
\bm{v}^{\star}
\end{array}\right),\qquad\widetilde{\bm{u}}_{2}^{\star}=\frac{1}{\sqrt{2}}\left(\begin{array}{c}
\bm{u}^{\star}\\
-\bm{v}^{\star}
\end{array}\right),\qquad\widetilde{\bm{u}}_{1}=\bm{u}_{1}^{\mathsf{dilation}},\qquad\widetilde{\bm{u}}_{2}=\bm{u}_{2}^{\mathsf{dilation}};
\]
\[
\widetilde{\lambda}_{1}^{\star}=\lambda_{1}(\bm{M}_{\mathsf{dilation}}^{\star})=1,\qquad\widetilde{\lambda}_{2}^{\star}=\lambda_{2}(\bm{M}_{\mathsf{dilation}}^{\star})=-1,\qquad \widetilde{\lambda}_{1}=\lambda_{1}^{\mathsf{dilation}},\qquad\widetilde{\lambda}_{2} = \lambda_{2}^{\mathsf{dilation}}.
\]
In addition, we denote $\widetilde{\bm{u}}_{1,1}=\bm{u}^{\mathsf{dilation}}_{1,1}$  and $\widetilde{\bm{u}}_{1,2}=\bm{u}^{\mathsf{dilation}}_{1,2}$. 
Recall that we assume $\widetilde{\lambda}_{1} \geq \widetilde{\lambda}_{2}$. We also denote $\min \{ \| \bm{x}\pm \bm{y} \|_2\}=\min \{ \|\bm{x} - \bm{y}\|_2, \|\bm{x} + \bm{y}\|_2 \}$. 

\subsection{$\ell_2$ eigenvector perturbation bounds}

The $\ell_2$ perturbation bound  results from near orthogonality between $\widetilde{\bm{u}}_1$ and $\widetilde{\bm{u}}_2^{\star}$. By symmetry, it suffices to establish the result for $\bm{u}$.

To begin with, applying Theorem \ref{thm_Linear_Forms_1_rank_r_A} on  $\bm{M}_{\mathsf{dilation}}$ and taking $\bm{a} = \widetilde{\bm{u}}_2^{\star}$, we derive
\begin{align*}
\left|\widetilde{\bm{u}}_2^{\star\top} \widetilde{\bm{u}}_1 \left(1-\frac{\widetilde{\lambda}_2^{\star}}{\widetilde{\lambda}_1}\right)\right| 
	& = \left|\widetilde{\bm{u}}_{2}^{\star\top}\left(\widetilde{\bm{u}}_{1}-\frac{\widetilde{\lambda}_{1}^{\star}\widetilde{\bm{u}}_{1}^{\star\top}\widetilde{\bm{u}}_{1}}{\widetilde{\lambda}_{1}}\widetilde{\bm{u}}_{1}^{\star}-\frac{\widetilde{\lambda}_{2}^{\star}\widetilde{\bm{u}}_{2}^{\star\top}\widetilde{\bm{u}}_{1}}{\widetilde{\lambda}_{1}}\widetilde{\bm{u}}_{2}^{\star}\right)\right| \\
	& \lesssim \max\left\{B \log n, \sqrt{n \sigma^2 \log n}\right\} \sqrt\frac{\mu}{n},
\end{align*}
where the identity arises since $\widetilde{\bm{u}}_2^{\star\top}\widetilde{\bm{u}}_1^{\star}=0$. 
Given that $\widetilde{\lambda}_2^{\star} = -1$ and $\widetilde{\lambda}_1 > 0$ (see Corollary \ref{cor:eigenvalue-pertubation-rank-2}), we have $1-\widetilde{\lambda}_2^{\star}/\widetilde{\lambda}_1 > 1$.  The near orthogonality property can then be described as follows
\begin{equation} \label{eq:near-orthogonality}
	\left|\widetilde{\bm{u}}_2^{\star \top} \widetilde{\bm{u}}_1\right|  \leq \left|\widetilde{\bm{u}}_2^{\star\top} \widetilde{\bm{u}}_1 \big(1- {\widetilde{\lambda}_2^{\star}}/{\widetilde{\lambda}_1}\big)\right|  \lesssim \max\left\{B \log n, \sqrt{n \sigma^2 \log n}\right\} \sqrt\frac{\mu}{n} .
\end{equation}
This combined with Lemma \ref{lem_l2_convergence-rank-r}  yields
\begin{align}
\min\{\|\widetilde{\bm{u}}_1 \pm \widetilde{\bm{u}}_1^{\star}\|_{2} \} & = \sqrt{\left\|\widetilde{\bm{u}}_1^{\star }\right\|_2^2 + \left\|\widetilde{\bm{u}}_1\right\|_2^2-2\left|\widetilde{\bm{u}}_1^{\star \top} \widetilde{\bm{u}}_1\right|} \nonumber \\
& = \sqrt{2 \left(1-\left|\widetilde{\bm{u}}_1^{\star \top} \widetilde{\bm{u}}_1\right|\right)}  \leq \sqrt{2 \big(1-\left|\widetilde{\bm{u}}_1^{\star \top} \widetilde{\bm{u}}_1\right|^2\big)} \nonumber\\
	& \stackrel{(\mathrm{i})}{\leq} \sqrt{\frac{128}{9} \left\|\bm{H}_{\mathsf{dilation}}\right\|^2+2\left|\widetilde{\bm{u}}_2^{\star \top} \widetilde{\bm{u}}_1\right|^2} \nonumber\\
	& \stackrel{(\mathrm{ii})}{\lesssim}  \left\|\bm{H}_{\mathsf{dilation}}\right\| +  \left|\widetilde{\bm{u}}_2^{\star \top} \widetilde{\bm{u}}_1\right| \nonumber\\
	& \stackrel{(\mathrm{iii})}{\lesssim}  \max\left\{B \log n, \sqrt{n \sigma^2 \log n}\right\},  \label{eq:UB-25}
\end{align}
where $(\mathrm{i})$ results from Lemma \ref{lem_l2_convergence-rank-r} and the fact that the condition number of $\bm{M}_{\mathsf{dilation}}^{\star}$ is 1, $(\mathrm{ii})$ follows since $\sqrt{x^2+y^2} \leq x + y$ for all $x, y \geq 0$, and $(\mathrm{iii})$ arises from (\ref{eq:near-orthogonality}), Lemma \ref{lem:H-norm}, the identity $\left\|\bm{H}_{\mathsf{dilation}}\right\| = \max\{\|\bm{H}_1\|,\|\bm{H}_2\|\}$, and the fact $\sqrt{\mu/n}\leq 1$.

It then boils down to showing that $\min\{ \|\bm{u}\pm \bm{u}^\star\|_2 \} \lesssim \min\{\|\widetilde{\bm{u}}_1 \pm \widetilde{\bm{u}}_1^{\star}\|_{2} \}$. 
To this end, we see that the estimate \eqref{eq:rank-1-inferred-singular-vectors} satisfies
\begin{align}
\|\bm{u}-\bm{u}^\star\|_2  & = \left\| \left(\frac{\widetilde{\bm{u}}_{1,1}}{\left\|\widetilde{\bm{u}}_{1,1}\right\|_2}-\sqrt{2}\widetilde{\bm{u}}_{1,1} \right)+\left(\sqrt{2}\widetilde{\bm{u}}_{1,1}- \bm{u}^\star\right) \right\|_2 \nonumber \\
& \leq \left|\frac{1}{\left\|\widetilde{\bm{u}}_{1,1}\right\|_2}-\sqrt{2}\right| + \sqrt{2} \left\|\widetilde{\bm{u}}_{1,1}-\frac{\bm{u}^{\star}}{\sqrt{2}}\right\|_2  \nonumber \\
& = \frac{\sqrt{2}}{\left\|\widetilde{\bm{u}}_{1,1}\right\|_2} \left|\left\|\widetilde{\bm{u}}_{1,1}\right\|_2-\left\|\frac{\bm{u}^{\star}}{\sqrt{2}}\right\|_2\right| + \sqrt{2}\left\| \widetilde{\bm{u}}_{1,1}- \frac{\bm{u}^\star}{\sqrt{2}} \right\|_2 \nonumber \\
	& \leq \sqrt{2} \left( \frac{1}{\left\| \widetilde{\bm{u}}_{1,1} \right\|_2} + 1\right) \left\| \widetilde{\bm{u}}_{1,1}- \frac{\bm{u}^\star}{\sqrt{2}} \right\|_2 \label{eq:UB-13} \\
	& \leq \sqrt{2} \left(1+\frac{\sqrt{2}}{1-\sqrt{2} \left\Vert \widetilde{\bm{u}}_1 - \widetilde{\bm{u}}_1^\star \right\Vert _{2} }\right)\left\| \widetilde{\bm{u}}_{1}- \widetilde{\bm{u}}_1^\star \right\|_2 . \label{eq:UB-14}
\end{align}
Here, \eqref{eq:UB-13} makes use of the triangle inequality, and \eqref{eq:UB-14} follows since 
\[
\frac{1}{\left\Vert \widetilde{\bm{u}}_{1,1}\right\Vert _{2}}\leq\frac{1}{\|\bm{u}^{\star}\|_{2}/\sqrt{2}-\left\Vert \widetilde{\bm{u}}_{1,1}-\bm{u}^{\star}/\sqrt{2}\right\Vert _{2}}
	= \frac{\sqrt{2}}{1-\left\Vert \sqrt{2}\widetilde{\bm{u}}_{1,1}-\bm{u}^{\star}\right\Vert _{2}} 
	\leq \frac{\sqrt{2}}{1-  \sqrt{2} \left\Vert \widetilde{\bm{u}}_1 - \widetilde{\bm{u}}_1^\star \right\Vert _{2}} , 
\]
where the last inequality relies on the fact that $\left\|\widetilde{\bm{u}}_{1,1} - \bm{u}^{\star} / \sqrt{2} \right\|_2 \leq \left\|\widetilde{\bm{u}}_1 - \widetilde{\bm{u}}_1^\star\right\|_2$ (the first $n_1$ coordinates). 
Similarly, one can derive  the above bounds for $\|\bm{u}+\bm{u}^\star\|_2$ as well. Therefore, we are left with 
\begin{align}
	\min \left\{\left\|\bm{u} \pm \bm{u}^\star\right\|_2 \right\}  
	&\leq  \sqrt{2} \left(1+\frac{\sqrt{2}}{1-\sqrt{2}\min\{\|\widetilde{\bm{u}}_1 \pm \widetilde{\bm{u}}_1^{\star}\|_{2} \}}\right)\min\{\|\widetilde{\bm{u}}_1 \pm \widetilde{\bm{u}}_1^{\star}\|_{2} \} \\
	& \lesssim  \min\{\|\widetilde{\bm{u}}_1 \pm \widetilde{\bm{u}}_1^{\star}\|_{2} \} \\
	& \lesssim   \max\left\{B \log n, \sqrt{n \sigma^2 \log n}\right\}.  \label{eq:UB-26}
\end{align}
This concludes the proof.

\subsection{Perturbation bounds for linear forms of eigenvectors}
\label{sec:linear-form-dilation}

Fix any unit vector 
$\bm{a} \in \mathbb{R}^{n_1}$.  Considering the linear transformation by the vector $\begin{pmatrix} \bm{a} \\ \bm{0} \end{pmatrix}$,  we can apply Theorem \ref{thm_Linear_Forms_1_rank_r_A} to $\bm{M}_{\mathsf{dilation}}$ and get
\begin{align}
& \left|\begin{pmatrix} \bm{a} \\ 0 \end{pmatrix}^\top \left(\widetilde{\bm{u}}_1 - \frac{\widetilde{\lambda}_1^\star \widetilde{\bm{u}}_1^{\star\top} \widetilde{\bm{u}}_1}{\widetilde{\lambda}_1}\widetilde{\bm{u}}_1^\star-\frac{\widetilde{\lambda}_2^\star \widetilde{\bm{u}}_2^{\star\top} \widetilde{\bm{u}}_1}{\widetilde{\lambda}_1}\widetilde{\bm{u}}_2^\star\right)\right| 
	\lesssim \max\left\{B \log n, \sqrt{n \sigma^2 \log n}\right\}\sqrt\frac{\mu}{n} .
\end{align}
Given that the first $n_1$ coordinates of $\widetilde{\bm{u}}_1^{\star}$ and $\widetilde{\bm{u}}_2^{\star}$ are both $\bm{u}^{\star}/\sqrt{2}$, we can invoke Corollary \ref{cor:eigenvalue-pertubation-rank-2} (i.e.~$\widetilde{\lambda}_1 \asymp 1$) to obtain
\begin{align}
  \left|\bm{a}^\top \left(\widetilde{\bm{u}}_{1,1}-\frac{\widetilde{\bm{u}}_1^{\star\top} \widetilde{\bm{u}}_1}{\sqrt{2} \,\widetilde{\lambda}_1} \bm{u}^\star \right)\right| \lesssim \left|\frac{\widetilde{\bm{u}}_2^{\star\top} \widetilde{\bm{u}}_1}{\sqrt{2}}\right|\left|\bm{a}^\top \bm{u}^\star\right|+\max\left\{B \log n, \sqrt{n \sigma^2 \log n}\right\}\sqrt\frac{\mu}{n}.
\end{align}
In view of (\ref{eq:near-orthogonality}) and the fact $ \left|\bm{a}^\top \bm{u}^\star\right| \leq 1$, one has
\begin{equation}
	\left|\bm{a}^\top \left(\widetilde{\bm{u}}_{1,1}-\frac{\widetilde{\bm{u}}_1^{\star\top} \widetilde{\bm{u}}_1}{\sqrt{2} \, \widetilde{\lambda}_1} \bm{u}^\star \right)\right| \lesssim \max\left\{B \log n, \sqrt{n \sigma^2 \log n}\right\} 
	  \sqrt\frac{\mu}{n} .
\end{equation}

Recall that $\widetilde{\lambda}_1 > 0$ (cf.~Corollary \ref{cor:eigenvalue-pertubation-rank-2}). If we further have
\begin{equation} 
	\label{eq:entrywise-coef-rank-2-dilation}
\left| \frac{ \big| \widetilde{\bm{u}}_1^{\star\top} \widetilde{\bm{u}}_1 \big| }{\sqrt{2} \, \widetilde{\lambda}_1 \left\|\widetilde{\bm{u}}_{1,1}\right\|_2} -1\right| \lesssim \max\left\{B \log n, \sqrt{n \sigma^2 \log n}\right\},
\end{equation}
then we can use the triangle inequality to reach
\begin{equation}
\begin{aligned}
\min\big\{ \big|\bm{a}^{\top} (\bm{u} \pm \bm{u}^{\star})\big| \big\}  & \leq \frac{1}{\left\|\widetilde{\bm{u}}_{1,1}\right\|_2} \left|\bm{a}^\top \left(\widetilde{\bm{u}}_{1,1}-\frac{\widetilde{\bm{u}}_1^{\star\top} \widetilde{\bm{u}}_1}{\sqrt{2} \, \widetilde{\lambda}_1} \bm{u}^\star \right)\right| + \left| \frac{ \big| \widetilde{\bm{u}}_1^{\star\top} \widetilde{\bm{u}}_1  \big| }{\sqrt{2} \, \widetilde{\lambda}_1 \left\|\widetilde{\bm{u}}_{1,1}\right\|_2} - 1 \right| \left|\bm{a}^\top \bm{u}^\star \right| \\
& \lesssim \left( \sqrt{\frac{\mu}{n}} + |\bm{a}^{\top} \bm{u}^{\star}| \right) \max\left\{B \log n, \sqrt{n \sigma^2 \log n }\right\}
\end{aligned}
\end{equation}
as claimed.  It then remains to justify (\ref{eq:entrywise-coef-rank-2-dilation}). Towards this, it suffices to combine a series of consequences from \eqref{eq:UB-25},   Corollary \ref{cor:eigenvalue-pertubation-rank-2}, and \eqref{eq:UB-26}, namely,
\begin{align}
	1 - \left|\widetilde{\bm{u}}_1^{\star\top} \widetilde{\bm{u}}_1\right|& =  \frac{1}{2}\min\{\| \widetilde{\bm{u}}_1 \pm \widetilde{\bm{u}}_1^{\star} \|_2^2 \}  \lesssim  \max\left\{B \log n, \sqrt{n \sigma^2 \log n}\right\}, \\
\big|\widetilde{\lambda}_1 - \widetilde{\lambda}_1^\star\big| & = \big|\widetilde{\lambda}_1 - 1 \big| \lesssim \max\left\{B \log n, \sqrt{n \sigma^2 \log n}\right\} \sqrt \frac{\mu}{n}, \\
	\left|\left\|\widetilde{\bm{u}}_{1,1} \right\|_2 - \frac{1}{\sqrt{2}}\right| & \leq  \min\left\{ \left\|\widetilde{\bm{u}}_{1,1}  \pm \frac{\bm{u}^\star}{\sqrt{2}}\right\|_2 \right\} 
	\leq \min \left\{ \left\|\widetilde{\bm{u}}_1 \pm \widetilde{\bm{u}}_1^\star\right\|_2 \right\} \lesssim \max\left\{B \log n, \sqrt{n \sigma^2 \log n}\right\}.
\end{align}

The proof for the bounds on $\bm{v}$ is similar and is thus omitted.

\subsection{Entrywise eigenvector perturbation bounds}

Recognizing that $\left\|\bm{u} - \bm{u}^\star\right\|_{\infty} = \max_i \left|\bm{e}_i^\top \bm{u} - \bm{e}_i^\top \bm{u}^\star\right|$ and using the incoherence  $\left|\bm{e}_i^\top \bm{u}\right| \leq \sqrt{\mu/n}$, we can prove this claim directly by invoking the results established in Appendix \ref{sec:linear-form-dilation} and taking the union bound.


	\section{Asymmetrization of data samples: two examples}
\label{sec:asymmetrization-rank1}

As mentioned earlier, an independent and asymmetric noise matrix arises when we collect two samples for each entry of the matrix of interest and arrange the samples in an asymmetric manner (i.e.~placing 1 sample on the upper triangular part and the other on the lower triangular part). Interestingly, our results might be applicable for some cases where we only have 1 sample for each entry.  In what follows, we describe two examples of this kind similar to the ones discussed in  Section \ref{sec:applications}, but with a symmetric noise matrix.  Once again, it is  assumed that $\bm{M}^{\star}$ is a rank-1 matrix with leading eigenvalue 1 and incoherence parameter $\mu$.


\bigskip
\noindent {\bf Low-rank matrix estimation from Gaussian noise.}\footnote{We thank Prof.~Zhou Fan for telling us this example.} Suppose that the noise matrix $\bm{H}$ is a symmetric matrix, and that the upper triangular part of $\bm{H}$ is composed of $\mathrm{i.i.d.}$ Gaussian random variables $\mathcal{N}(0, \sigma^2)$.

When $\sigma$ is known, one can decouple the upper and low triangular parts of $\bm{H}$ by adding a {\em skew-symmetric} Gaussian matrix $\bm{\Delta}$. Specifically, our strategy is:
\begin{itemize}
	\item[(1)] For each  $1 \leq i \leq j \leq n$, generate $\Delta_{ij} \sim \mathcal{N}(0, \sigma^2)$ independently, and set  ${\Delta}_{ji} = -{\Delta}_{ij}$;
	
	\item[(2)] Compute the leading eigenvalue and eigenvector of $\bm{M} + \bm{\Delta}$.
\end{itemize}
%
This is motivated by a simple observation from Gaussianality: $\bm{H} + \bm{\Delta}$ is now an asymmetric matrix whose off-diagonal entries are i.i.d.~$\mathcal{N}(0, 2\sigma^2)$; in fact, it is easy to verify that $H_{ij}+\Delta_{ij}$ and $H_{ji}-\Delta_{ji}$ are independent Gaussian random variables. As a result, the orderwise bounds \eqref{eq:error-rank-1-Gauss-symm} continue to hold if $\lambda$ and $\bm{u}$ are taken to be the leading eigenvalue and eigenvector of $\bm{M}+\bm{\Delta}$, respectively.

While this asymmetrization procedure comes with the price of doubling the noise variance, the eigenvalue perturbation bound may still be significantly smaller --- up to a factor of $O(\sqrt{n})$ --- than  the bound for the SVD approach.   This is also confirmed in the numerical simulations reported in Fig.~\ref{fig:gaussian_asymmetrization}.

\begin{figure}[htbp!]
	\centering
	\begin{tabular}{ccc}
		\includegraphics[width=0.32\linewidth]{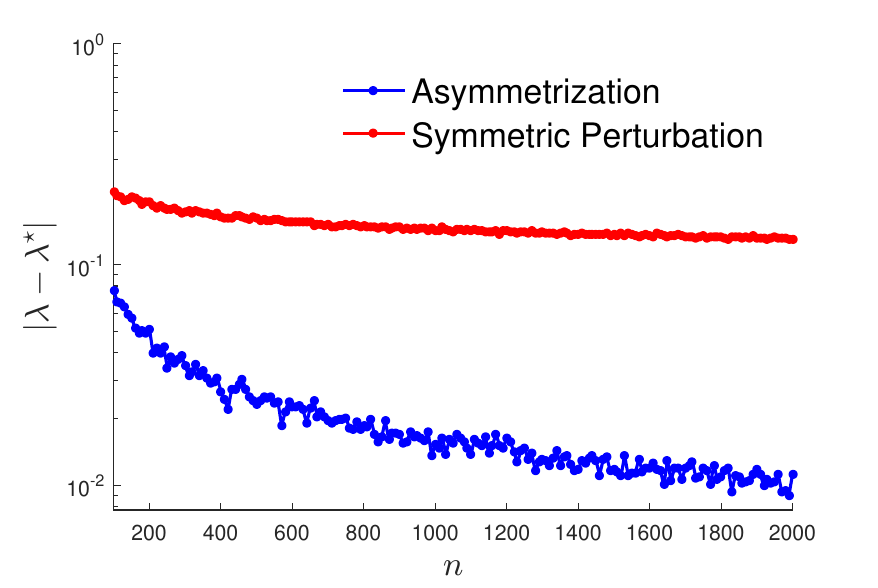}
		& \includegraphics[width=0.32\linewidth]{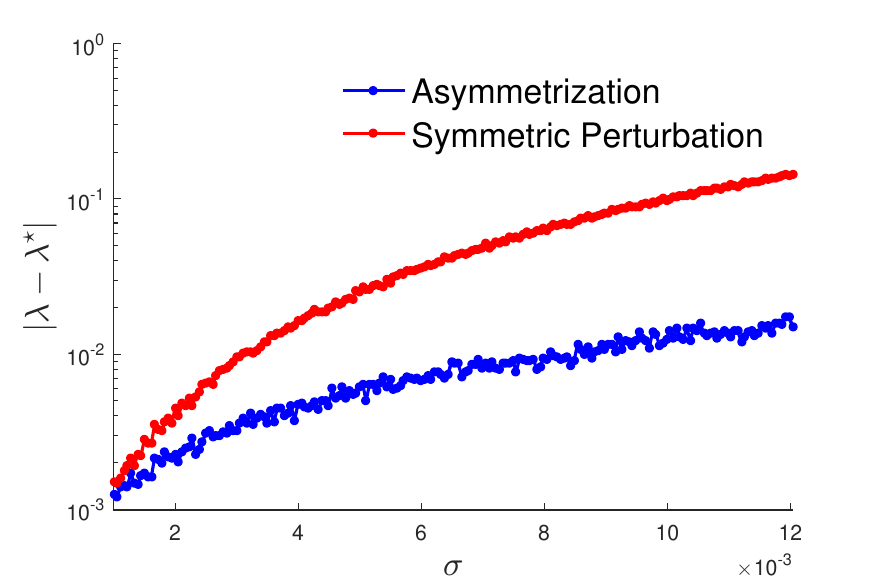}
		& \includegraphics[width=0.32\linewidth]{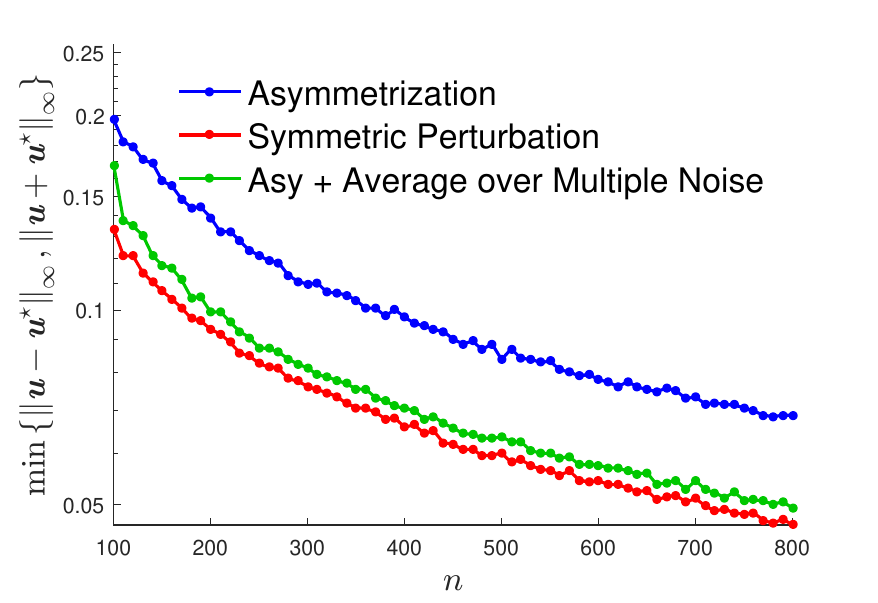}
	\tabularnewline
	(a) eigenvalue perturbation vs.~$n$ & (b) eigenvalue perturbation vs.~$\sigma$ & (c) $\ell_\infty$ eigenvector perturbation \tabularnewline
	\end{tabular}

	\caption{Numerical simulation for rank-1 matrix estimation when $\bm{H}$ is symmetric and its upper triangular part consists of i.i.d.~Gaussian entries $\mathcal{N}(0,\sigma^2)$. The truth $\bm{M}^{\star}$ is rank-1 and is randomly generated with leading eigenvalue $\lambda^{\star}=1$.  (a) $|\lambda-\lambda^{\star}|$ vs.~$n$ with $\sigma = 1 / \sqrt{n \log n}$; (b) $|\lambda-\lambda^{\star}|$ vs.~$\sigma$ with $n=1000$;  (c)  $\ell_{\infty}$ eigenvector perturbation vs.~$n$ with $\sigma = 1 / \sqrt{n \log n}$. We compare eigen-decomposition applied to symmetric matrix samples (the red lines) and asymmetrized data (the blue lines). In the green line, the estimate $\bm{u}$ is obtained by generating 10 independent copies of $\bm{\Delta}$  and aggregating the leading eigenvectors of these copies of $\bm{M}+\bm{\Delta}$ (see Remark \ref{remark:multiple-noise}).
	\label{fig:gaussian_asymmetrization}}
\end{figure}
\begin{remark}\label{remark:multiple-noise}
	While this asymmetrization procedure achieves enhanced  eigenvalue estimation accuracy compared to the SVD approach (see Fig.~\ref{fig:gaussian_asymmetrization}(a)(b)), it results in higher eigenvector estimation errors (see Fig.~\ref{fig:gaussian_asymmetrization}(c)). This is perhaps not surprising as we have added extra noise to the observed matrix.  One way to mitigate this issues is to (1) generate $K$ independent copies of $\bm{\Delta}$; (2) compute the leading eigenvector of each copy of $\bm{M}+\bm{\Delta}$, denoted by $\{\bm{u}^{(l)}\mid 1\leq l\leq K\}$; and (3) aggregate these eigenvectors, namely, compute the leading eigenvector of $\frac{1}{K} \sum_{l=1}^K \bm{u}^{(l)}\bm{u}^{(l)\top}$.  As can be seen in the green line of Fig.~\ref{fig:gaussian_asymmetrization}(c), this allows us to mitigate the effect of the extra noise component.
\end{remark}

\bigskip
\noindent {\bf Low-rank matrix completion.}
Suppose that each entry $M_{ij}^{\star}$ ($i\geq j$) is observed independently with probability $p$. This is different from the settings in Section \ref{sec:applications}, as we do not have additional samples for $M_{ji}^{\star}$ ($i\geq j$).  In order to arrange the data samples in an asymmetric and independent fashion, we employ a simple resampling technique to decouple the statistical dependency between $M_{ij}$ and $M_{ji}$:
\begin{itemize}
	\item[(1)] Define $\bm{M}^{\mathsf{sym}}$ such that
		\[
			M_{ij}^{\mathsf{sym}} = \begin{cases} \frac{1}{p} M_{ij}^{\star}, \quad & \text{if }M_{ij}^{\star}\text{ is observed}; \\ 0, & \text{else}.  \end{cases}
		\]
	\item[(2)] Define $p^{\mathsf{asym}}$ so that $p=1-(1-p^{\mathsf{asym}})^2$ (i.e.~$p^{\mathsf{asym}} = \frac{p}{ 1+\sqrt{1-p}}$).
		   For any pair $i> j$, set
	\begin{equation}
		\left(M_{ij}^{\mathsf{asym}}, M_{ji}^{\mathsf{asym}}\right)= \begin{cases} \frac{p}{p^{\mathsf{asym}}} \big( M_{ij}^{\mathsf{sym}}, 0 \big)  \qquad & \text{with probability } \frac{1-p^{\mathsf{asym}}}{2-p^{\mathsf{asym}}}, \\ \frac{p}{p^{\mathsf{asym}}} \big(0, M_{ij}^{\mathsf{sym}} \big)  \qquad & \text{with probability } \frac{1-p^{\mathsf{asym}}}{2-p^{\mathsf{asym}}}, \\ \frac{p}{p^{\mathsf{asym}}} \big( M_{ij}^{\mathsf{sym}}, M_{ij}^{\mathsf{sym}} \big)  \qquad & \text{with probability }  \frac{ p^{\mathsf{asym}}}{ 2-p^{\mathsf{asym}} }. \end{cases}
	\end{equation}
	For any $i$, set
	$$
	M_{ii}^{\mathsf{asym}}= \begin{cases}
	\frac{p}{p^{\mathsf{asym}}}M_{ii}^{\mathsf{sym}} \qquad & \text{with probability } \frac{p^{\mathsf{asym}}}{p}, \\
	0 \qquad & \text{else}.
	\end{cases}
	$$
	%

	%
	\item[(3)]  Compute the leading eigenvalue and eigenvector of $\bm{M}^{\mathsf{asym}}$.

\end{itemize}
%
%
As can be easily verified, this scheme is equivalent to saying that (i) with probability $p$,  either $M_{ij}^{\mathsf{asym}}$ or $M_{ji}^{\mathsf{asym}}$ is taken to be a rescaled version of $M_{ij}^{\star}$;   (ii) for any $i\neq j$, the entries $\{M^{\mathsf{asym}}_{ij}\}$ are independently drawn.
%
%
Since $p^{\mathsf{asym}}\asymp p$, our results \eqref{eq:error-rank-1-MC-symm} in  Section \ref{sec:applications} remain valid, as long as $\lambda$ and $\bm{u}$ are set to be the leading eigenvalue and eigenvector of $\bm{M}^{\mathsf{asym}}$, respectively.
Numerical simulations have been carried out in Fig.~\ref{fig:mc_asymmetrization} to verify the effectiveness of this scheme.
	
	\begin{figure}[htbp!]
		\centering
		\begin{tabular}{ccc}
		\includegraphics[width=0.32\linewidth]{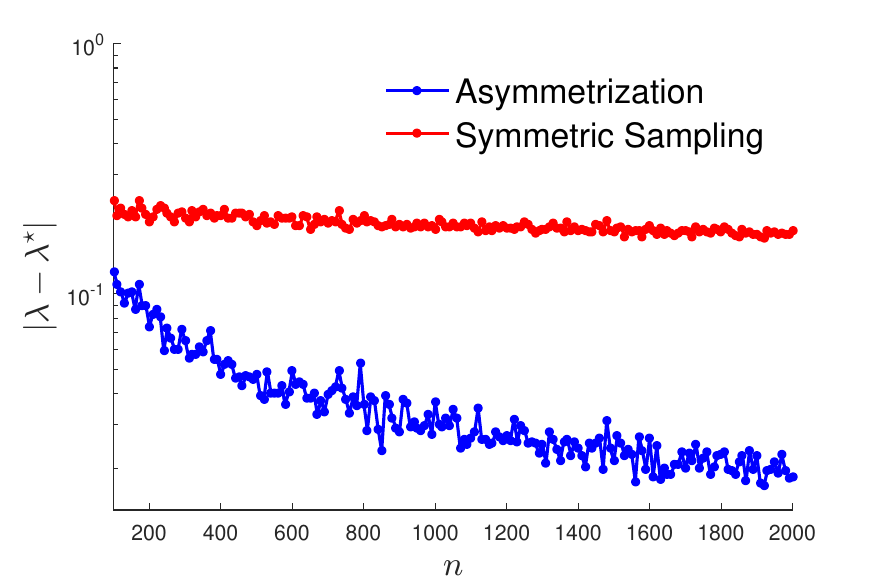}
		& \includegraphics[width=0.32\linewidth]{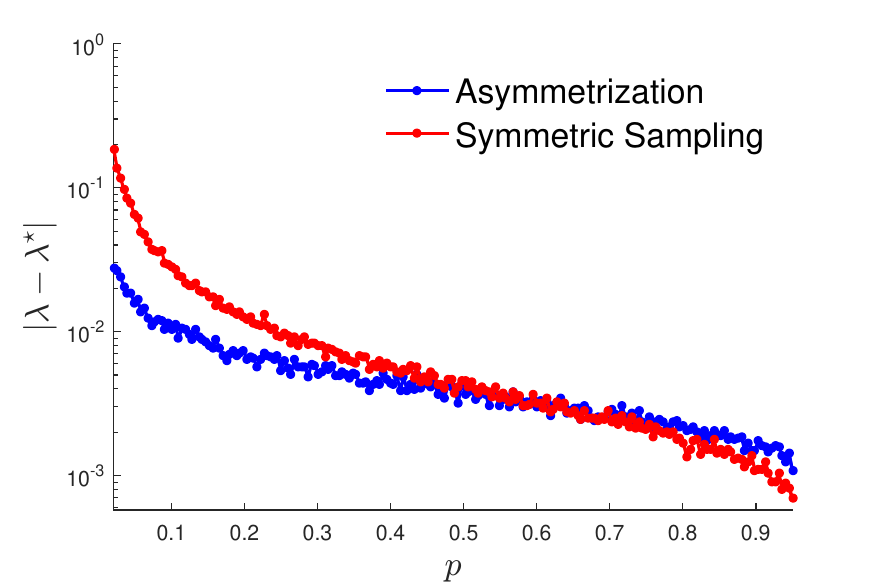}
		& \includegraphics[width=0.32\linewidth]{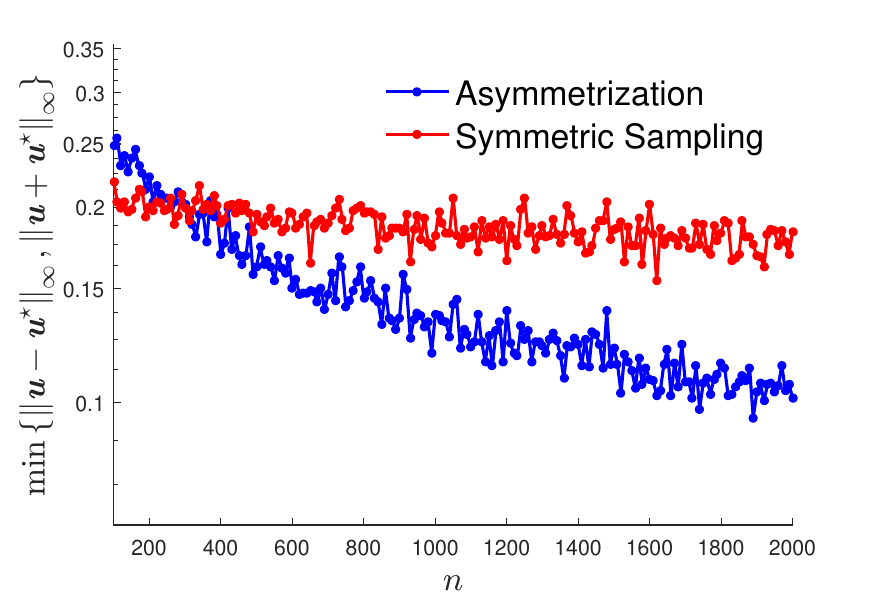}
		\tabularnewline
		(a) eigenvalue perturbation vs.~$n$ & (b) eigenvalue perturbation vs.~$p$ & (c) $\ell_\infty$ eigenvector perturbation \tabularnewline
		\end{tabular}

		\caption{Numerical experiments for rank-1 matrix completion, where the rank-1 truth $\bm{M}^{\star}$ is randomly generated with leading eigenvalue 1.  (a) $|\lambda-\lambda^{\star}|$ vs.~$n$ with $p =3\log n\,/\,n$; (b) $|\lambda-\lambda^{\star}|$ vs.~$p$ with $n=1000$; (c) $\ell_{\infty}$ perturbation error vs.~$n$ with $p =3\log n\,/\,n$. We compare eigen-decomposition applied to symmetric matrix samples and asymmetrized data.
		The blue (resp.~red) lines represent the average errors over 100 independent trials using the eigen-decomposition (resp.~SVD) approach.
		\label{fig:mc_asymmetrization}}
	\end{figure}

\end{document}